\documentclass[a4paper]{amsart}

\usepackage{amssymb}
\usepackage{latexsym}
\usepackage{amsmath}
\usepackage{euscript}
\usepackage{hyperref}

\numberwithin{equation}{section}

\newcommand{\dom}{{\mathrm{dom\,}}}
\newcommand{\ran}{{\mathrm{ran\,}}}

\DeclareMathOperator{\supp}{supp}

\newtheorem{theorem}{Theorem}[section]
\newtheorem{proposition}[theorem]{Proposition}
\newtheorem{lemma}[theorem]{Lemma}
\newtheorem{corollary}[theorem]{Corollary}

\newtheorem{definition}[theorem]{Definition}
\newtheorem{remark}[theorem]{Remark}

\newcommand{\ba}{\begin{array}}
\newcommand{\ea}{\end{array}}
\newcommand{\bea}{\begin{eqnarray}}
\newcommand{\eea}{\end{eqnarray}}
\newcommand{\bead}{\begin{eqnarray*}}
\newcommand{\eead}{\end{eqnarray*}}
\newcommand{\be}{\begin{equation}}
\newcommand{\ee}{\end{equation}}
\newcommand{\bed}{\begin{displaymath}}
\newcommand{\eed}{\end{displaymath}}
\newcommand{\bl}{\begin{lem}}
\newcommand{\el}{\end{lem}}
\newcommand{\bp}{\begin{prop}}
\newcommand{\ep}{\end{prop}}
\newcommand{\bt}{\begin{thm}}
\newcommand{\et}{\end{thm}}

\newcommand{\bc}{\begin{cor}}
\newcommand{\ec}{\end{cor}}

\newcommand{\br}{\begin{rem}}
\newcommand{\er}{\end{rem}}
\newcommand{\bd}{\begin{defn}}
\newcommand{\ed}{\end{defn}}

\DeclareMathOperator\grad{grad}
\DeclareMathOperator\Real{Re}
\DeclareMathOperator\Imag{Im}
\renewcommand\Re{\Real}
\renewcommand\Im{\Imag}

\newcommand\cB{\mathcal B}
\newcommand\cD{\mathcal D}

\newcommand\cG{\mathcal G}
\newcommand\cH{\mathcal H}
\newcommand\cK{\mathcal K}

\newcommand\cL{\mathcal L}

\newcommand\cN{\mathcal N}

\newcommand\CC{\mathbb C}
\newcommand\NN{\mathbb N}
\newcommand\RR{\mathbb R}

\newcommand\frA{\mathfrak A}
\newcommand\frB{\mathfrak B}
\newcommand\frS{\mathfrak S}
\newcommand\fra{\mathfrak a}

\newcommand\fraN{\mathfrak{a}}

\newcommand\dis{\displaystyle}
\newcommand\ov{\overline}
\newcommand\wt{\widetilde}
\newcommand\wh{\widehat}
\newcommand{\defeq}{\mathrel{\mathop:}=}
\newcommand{\eqdef}{=\mathrel{\mathop:}}
\newcommand{\defequ}{\mathrel{\mathop:}\hspace*{-0.72ex}&=}

\newcommand\vect[2]{\begin{pmatrix} #1 \\[1ex] #2 \end{pmatrix}}
\newcommand\sess{\sigma_{\rm ess}}

\newcommand\AD{A_{\rm D}}

\newcommand\tAN{\widetilde{A}_{\rm N}}
\newcommand\AN{A_{\rm N}}

\newcommand\void[1]{}

\def\AD{A_{\rm D}}
\def\AN{A_{\rm N}}

\def\sess{\sigma_{\rm ess}}

\DeclareMathOperator\tr{tr}

\def\sA{{\mathfrak A}}   \def\sB{{\mathfrak B}}

\def\sS{{\mathfrak S}}

      \def\dC{{\mathbb C}}

   \def\dN{{\mathbb N}}   
      \def\dR{{\mathbb R}}

   \def\cB{{\mathcal B}}   \def\cC{{\mathcal C}}
\def\cD{{\mathcal D}}      
\def\cG{{\mathcal G}}   \def\cH{{\mathcal H}}   
   \def\cK{{\mathcal K}}   \def\cL{{\mathcal L}}
   \def\cN{{\mathcal N}}

\def\mul{{\text{\rm mul\,}}}

\begin{document}

\title[Spectral estimates for resolvent differences]{
Spectral estimates for resolvent differences of \\ self-adjoint elliptic operators
}

\author{Jussi Behrndt}
\address{Institut f\"{u}r Mathematik, MA 6-4,
Technische Universit\"{a}t Berlin,
Strasse des 17.~Juni 136,
D-10623 Berlin, Germany}
\email{behrndt@math.tu-berlin.de}

\author{Matthias Langer}

\address{Department of Mathematics and Statistics,
University of Strathclyde, 
26 Richmond Street, Glasgow G1 1XH, United Kingdom}
\email{m.langer@strath.ac.uk}

\author{Vladimir Lotoreichik}

\address{Department of Mathematics,
St.~Petersburg State University of  Information Technologies, Mechanics and
Optics, 
Kronverkskiy, 49,  St.~Petersburg, Russia}
\email{vladimir.lotoreichik@gmail.com}

\begin{abstract}
The notion of quasi boundary triples and their Weyl functions is an abstract concept
to treat spectral and boundary value problems for elliptic partial differential equations.
In the present paper the abstract notion is further developed, and general theorems 
on resolvent differences belonging to operator ideals are proved. The results are applied 
to second order elliptic differential operators on bounded and exterior domains,
and to partial differential operators with $\delta$ and $\delta'$-potentials 
supported on hypersurfaces.
\end{abstract}

\subjclass[2000]{Primary 35P05 \ 35P20; Secondary 47F05 \ 47L20 \ 81C15 \ 81Q10}
\keywords{elliptic operator \
self-adjoint extension \ 
operator ideal \
$\delta$ and $\delta'$-potential}

\maketitle

\section{Introduction}

The extension theory of symmetric operators in Hilbert spaces was one of the
major advances in operator theory in the 20th century, which has numerous
applications to problems in mathematics and physics, among them, differential operators,
moment and interpolation problems, to mention just a few.
There are various approaches to the extension problem of symmetric operators, e.g.\
the use of deficiency subspaces as developed by J.~von Neumann \cite{vN29}
and quadratic form methods as used by
K.\,O.~Friedrichs~\cite{F34}. The extension theory was further developed by
M.\,G.~Kre\u{\i}n~\cite{K47}, M.\,I.~Vishik~\cite{V52}, M.\,Sh.~Birman~\cite{B56},
G.~Grubb~\cite{G68}, and T.~Ando and K.~Nishio \cite{AN70}.
Moreover, the papers \cite{V52} and \cite{G68} contain important applications 
to elliptic differential operators; see also \cite{G70,G71,G74}.
A more recent concept in extension theory of symmetric operators is the notion 
of boundary triples introduced by A.\,N.~Kochubei~\cite{Ko75},
V.\,M.~Bruk~\cite{Br76}, and further studied by V.\,I.~Gorbachuk and
M.\,L.~Gorbachuk~\cite{GG91}, and V.\,A.~Derkach and M.\,M.~Malamud~\cite{DM91,DM95};
a similar abstract concept was already proposed by J.\,W.~Calkin \cite{C39}.

In the approach with boundary triples self-adjoint extensions of a symmetric
operator $A$ in a Hilbert space $\cH$ are described via abstract boundary conditions.
Roughly speaking, two boundary mappings $\Gamma_0$, $\Gamma_1$ are used, which are
defined on the domain of the maximal operator
(i.e.\ the adjoint $A^*$ of the symmetric operator $A$),
map into an auxiliary Hilbert space $\cG$ (the space of boundary values) and satisfy an
abstract Green identity
\begin{equation}\label{introgreen}
  (A^*f,g)-(f,A^*g) = (\Gamma_1f,\Gamma_0g)-(\Gamma_0f,\Gamma_1g)
\end{equation}
where the inner products on the left-hand side are in $\cH$, the ones on the
right-hand side are in the boundary space $\cG$.
The self-adjoint extensions $A_\Theta$ are characterized as restrictions of
$A^*$ to the set of elements $f$ satisfying the abstract boundary condition
\begin{equation}\label{boundc}
  \binom{\Gamma_0 f}{\Gamma_1 f} \in \Theta,
\end{equation}
where $\Theta$ is a self-adjoint linear relation in $\cG$, i.e.\ a ``self-adjoint'' 
subspace of $\cG\times\cG$ (see Section~\ref{2.1} for a discussion about linear relations).
The theory of boundary triples was successfully applied in many situations, in particular,
ordinary differential operators, see, e.g.\ \cite{BL10,BMN08,BMN02,BGP08,DHS03,KM10}.
For second order differential operators on an interval
one usually chooses $\Gamma_0$ to give Dirichlet data at endpoints of the interval
and $\Gamma_1$ Neumann data, or vice versa.

For elliptic partial differential operators the same approach with the boundary mappings
$\Gamma_0$ and $\Gamma_1$ as the Dirichlet trace and the
conormal derivative, respectively, leads to serious
difficulties since Green's identity does not make sense
on the whole domain of the maximal operator.  Moreover, a surjectivity condition
for the boundary mappings that is imposed for boundary triples is also not satisfied.
Based on ideas from \cite{G68}, a boundary triple with regularized versions
of trace and conormal derivatives was used for elliptic operators
in \cite{BGW09,BMNW08,M10}.
However, in order to work with the usual trace and conormal derivative,
a generalization of the notion of boundary triples was introduced in \cite{BL07}:
quasi boundary triples.  In this setting the boundary mappings $\Gamma_0$ and $\Gamma_1$ are not defined on the
whole domain of the maximal operator $A^*$ but only on the domain of some restriction
$T$ whose closure is $A^*$; the abstract Green identity \eqref{introgreen} holds then with $A^*$ replaced by $T$.  For elliptic operators on a bounded domain $\Omega$ one
can choose $T$, for instance, to be defined on $H^2(\Omega)$, and therefore also the
boundary mappings are defined on $H^2(\Omega)$, which is much smaller than the
maximal domain.
The aim of the current paper is to develop the theory of quasi boundary triples
further and use it to prove new results in spectral theory.
We apply these results to elliptic operators on bounded and exterior domains
and to partial differential operators with $\delta$ and $\delta'$-potentials
supported on hypersurfaces in $\RR^n$.

In the following, let $A$ be a symmetric operator in a Hilbert space $\cH$ and
let $\{\cG,\Gamma_0,\Gamma_1\}$ be a quasi boundary triple for $A^*=\overline T$, 
see also Definition~\ref{def:qbt}.
Besides formula \eqref{introgreen} with $A^*$ replaced by $T$ and a density condition 
on the range of the boundary mappings $\Gamma_0$ and $\Gamma_1$, it is also assumed 
in the definition of a quasi boundary triple
that the restriction $A_0$ of $T$ to $\ker\Gamma_0$
is a self-adjoint operator. This operator is often used as a reference extension of $A$
which other extensions of $A$ are compared with.
A very important object that is associated with a quasi boundary triple is the
Weyl function $M(\lambda)$, which, for $\lambda\in\rho(A_0)$,
is an operator in $\cG$ that satisfies
\[
  \Gamma_1f = M(\lambda)\Gamma_0f,
\]
for $f\in\ker(T-\lambda)$.  Hence $M(\lambda)$ connects the two ``boundary values'' $\Gamma_0f$ and
$\Gamma_1f$ for solutions of the equation $Tf=\lambda f$.
In our treatment of elliptic operators in Sections~\ref{4.1} and \ref{4.2} it
will turn out that $M(\lambda)$ is the Neumann-to-Dirichlet map.

In the quasi boundary triple framework a self-adjoint relation $\Theta$ in $\cG$ as abstract
boundary condition in \eqref{boundc} does not automatically induce a self-adjoint 
restriction $A_\Theta$ of $T$ in $\cH$ (as is the case for boundary triples) but 
only a symmetric operator $A_\Theta$. In Theorem~\ref{th.A_Th_self_adj} we provide a 
sufficient condition on the Weyl function $M(\lambda)$ and $\Theta$ so that
the operator $A_\Theta$ becomes self-adjoint.
Applied to elliptic operators, this theorem yields a wide class of local and non-local
boundary conditions for which there exists a self-adjoint realization in
an $H^2$-setting (Theorem~\ref{condsa} and Corollary~\ref{condsacor}).
The proof of Theorem~\ref{th.A_Th_self_adj} uses Krein's formula, in which
the resolvents of $A_\Theta$ and $A_0$ are compared, namely
\[
  (A_\Theta-\lambda)^{-1} = (A_0-\lambda)^{-1}
  + \gamma(\lambda)\bigl(\Theta-M(\lambda)\bigr)^{-1}\gamma(\ov\lambda)^*,
\]
where $\gamma(\lambda)$ is the $\gamma$-field and maps elements
$\varphi\in\ran\Gamma_0\subset\cG$ onto solutions $f$ of $Tf=\lambda f$
with $\Gamma_0f=\varphi$; see Theorem~\ref{th.krein}.
Actually, we provide the formula also in the case when $A_\Theta-\lambda$
is not necessarily surjective but only injective and $\lambda\in\rho(A_0)$;
the formula then has to be read so that it is applied only to elements
in $\ran(A_\Theta-\lambda)$.

Krein's formula is also an important ingredient in the proofs of the results of the core
section~\ref{3.3} in the abstract part of the present paper.  There we prove spectral 
estimates for resolvent differences, in particular, the resolvent difference
\begin{equation}\label{resolv_diff0}
  (A_\Theta-\lambda)^{-1} - (A_0-\lambda)^{-1}
\end{equation}
of a self-adjoint extension $A_\Theta$ described by an abstract boundary
condition \eqref{boundc} and the fixed self-adjoint extension $A_0$.
More precisely, we prove that the resolvent difference is in some operator
ideal provided that $\gamma(\lambda)^*$ is in some related operator ideal;
see Theorem~\ref{th.resolv_diff1} and the following theorems.
The use of operator ideals gives a very general tool to study
resolvent differences but includes, in particular, spectral estimates
of Schatten--von~Neumann type, i.e.\ that the singular values $s_k$ of
\eqref{resolv_diff0} satisfy
\[
  s_k = O(k^{-r})\,\,\,\,\text{or}\,\,\,\,s_k = o(k^{-r}), \quad k\to\infty, 
  \qquad\text{or}\qquad \sum_{k=1}^\infty s_k^p<\infty
\]
for some $r>0$ or $p>0$.
We investigate also the resolvent difference of $A_{\Theta_1}$ and $A_{\Theta_2}$
for two abstract boundary conditions $\Theta_1$, $\Theta_2$ under some
assumptions on $\Theta_1-\Theta_2$; see Theorem~\ref{th.resolv_diff4}.

As mentioned above the first class of operators to which we apply our abstract
results is connected with elliptic partial differential expressions; we study
expressions of the form
\begin{equation}\label{intr.ell_expr}
  \cL = -\sum_{j,k=1}^n \frac{\partial}{\partial x_j} a_{jk}
  \frac{\partial}{\partial x_k}+ a
\end{equation}
on a domain $\Omega$ in $\RR^n$ with compact $C^\infty$-boundary $\partial\Omega$.
The domain $\Omega$ itself is allowed to be either bounded or the complement
of a bounded set.
We define the associated operator $T$ on $H^2(\Omega)$ if $\Omega$ is bounded, 
and on a set of functions which are in $H^2$ in a neighbourhood of $\partial\Omega$
if $\Omega$ is unbounded; for details see Definition~\ref{domt}.
For the space of boundary values $\cG$ we choose $L^2(\partial\Omega)$, and the
boundary mappings are defined by
\[
  \Gamma_0f =  \frac{\partial f}{\partial\nu_\cL}\Bigl|_{\partial\Omega} \defeq
  \sum_{j,k=1}^n a_{jk} n_j \frac{\partial f}{\partial x_k}
  \Bigl|_{\partial\Omega}\quad\text{and} \quad \Gamma_1f = f\big|_{\partial\Omega},
\]
where $n(x)=(n_1(x),\dots, n_n(x))^\top$ is the unit vector at the point
$x\in\partial\Omega$ pointing out of $\Omega$.
After having established in Theorem~\ref{qbtthm} that
$\{L^2(\partial\Omega),\Gamma_0,\Gamma_1\}$ is a quasi boundary triple,
we apply our abstract results from Section~\ref{sec3}.
In Theorem~\ref{condsa} we prove that, for an arbitrary
bounded self-adjoint operator $B$ in $L^2(\partial\Omega)$ that satisfies
$B(H^1(\partial\Omega))\subset H^{1/2}(\partial\Omega)$, the elliptic
expression $\cL$ together with the boundary condition
\begin{equation}\label{boundcell}
  B\bigl(f\big|_{\partial\Omega}\bigr)
  = \frac{\partial f}{\partial\nu_\cL}\Bigl|_{\partial\Omega}
\end{equation}
gives rise to a self-adjoint operator $L^2(\Omega)$ whose domain consists of
functions $f$ which are in $H^2$ in a neighbourhood of the boundary $\partial\Omega$.
The boundary condition in \eqref{boundcell} corresponds to the abstract boundary
condition \eqref{boundc} with $\Theta=B^{-1}$ and contains a large class of Robin boundary
conditions but also non-local boundary conditions.

In order to describe our main results on spectral estimates of resolvent differences
of elliptic operators, we use the following notation here in the introduction.  We write
\begin{equation}\label{notation_resolv_diff}
  H_1 \stackrel{r}{\text{------}} H_2,
\end{equation}
if the singular values $s_k$ of the resolvent difference
$(H_1-\lambda)^{-1} - (H_2-\lambda)^{-1}$
of two self-adjoint operators $H_1$, $H_2$ satisfy
$s_k = O(k^{-r})$, $k\to\infty$,
for all $\lambda\in\rho(H_1)\cap \rho(H_2)$.
In Theorem~\ref{thm1} we prove that
\begin{equation}\label{resdiff_intr1}
  \AN \stackrel{\frac{3}{n-1}}{\text{-----------}} A_\Theta,
\end{equation}
where $\AN$ is the Neumann realization of $\cL$ and
$\Theta$ is a self-adjoint relation in $L^2(\partial\Omega)$ so that
$0\notin\sess(\Theta)$ and $A_\Theta$ is self-adjoint.
For instance,  $\Theta=B^{-1}$ with a bounded self-adjoint $B$ as above,
i.e.\ the partial differential operator with boundary condition \eqref{boundcell},
leads to \eqref{resdiff_intr1}.  A slightly weaker result for the Laplacian
on bounded domains was proved in \cite{BLLLP10}.
M.\,Sh.~Birman \cite{B62} proved that
\[
  \AD \stackrel{\frac{2}{n-1}}{\text{-----------}} \AN,
\]
and later M.\,Sh.~Birman and M.\,Z.~Solomjak \cite{BS80} and G.~Grubb \cite{G84,G84a}
further investigated this relation and obtained the exact spectral
asymptotics of the resolvent difference.  In general, the operator $A_\Theta$ as
above is closer to the Neumann operator $\AN$ in the sense of \eqref{resdiff_intr1}
then to the Dirichlet operator $\AD$.
If $n=2$ or $n=3$, then the resolvent difference of $\AN$ and $A_\Theta$ is
a trace class operator by \eqref{resdiff_intr1}; in Corollary~\ref{co.trace_ell}
we obtain a trace formula for this resolvent difference, which involves the
Neumann-to-Dirichlet map and $\Theta$.
We can compare also two operators with non-local boundary conditions
$A_{\Theta_1}$, $A_{\Theta_2}$ under some assumption on $\Theta_1-\Theta_2$, namely
in Theorem~\ref{thm3} if $s_k(\Theta_1-\Theta_2)=O(k^{-r})$, $k\to\infty$, then
\[
  A_{\Theta_1} \stackrel{\frac{3}{n-1}+r}{\text{--------------}} A_{\Theta_2}.
\]

The second class of operators we study and to which we apply our abstract
results contains elliptic operators on $\RR^n$ with additional $\delta$ and
$\delta'$-potentials which are supported on a bounded $C^\infty$-hypersurface
$\Sigma$, which splits $\RR^n$ into two components:
an interior domain $\Omega_{\rm i}$, which is bounded,
and an exterior domain $\Omega_{\rm e}$.

The spectral theory of Schr\"odinger operators with $\delta$-potentials on
surfaces is developed since the late 80s; see, e.g.\ the papers
\cite{AGS87,BEKS94,E08,EF09,EI01,EK03,S88}.
Nevertheless, several questions remained open.
One of the open questions is to find the domain of
self-adjointness of the operator with $\delta$-potential in the scale of Sobolev
spaces under suitable assumptions on the smoothness of the strength of the
potential. Another question pointed out by P.~Exner in his survey paper~\cite{E08}
is to find a way how to treat $\delta^\prime$-potentials on surfaces.
The case of $\delta'$-potentials is more difficult than that
of $\delta$-potentials because that kind of perturbations are not
form-bounded; see~\cite{E08}.

In Section~\ref{4.3} we use quasi boundary triples and our abstract results
from Section~\ref{sec3} to construct self-adjoint operators
$A_{\delta,\alpha}$ and $A_{\delta',\beta}$ that are differential operators
connected with an elliptic expression $\cL$ as in \eqref{intr.ell_expr} on
$\RR^n$ with interface conditions
\[
  f_{\rm e}\vert_\Sigma = f_{\rm i}\vert_\Sigma, \qquad
  \frac{\partial f_{\rm i}}{\partial\nu_{\cL_{\rm i}}}\Bigl|_{\Sigma}
  +\frac{\partial f_{\rm e} }{\partial\nu_{\cL_{\rm e}}}\Bigl|_{\Sigma}
  = \alpha f\vert_\Sigma,
\]
and
\[
  \frac{\partial f_{\rm i}}{\partial\nu_{\cL_{\rm i}}}\Bigl|_{\Sigma}
  = - \frac{\partial f_{\rm e} }{\partial\nu_{\cL_{\rm e}}}\Bigl|_{\Sigma}, \qquad
  f_{\rm e}\vert_\Sigma - f_{\rm i}\vert_\Sigma =
  \beta \frac{\partial f_{\rm e}}{\partial \nu_{\cL_{\rm e}}}\Bigl|_{\Sigma},
\]
respectively, where $f_{\rm i}$ and $f_{\rm e}$ are the restrictions of $f$
to $\Omega_{\rm i}$ and $\Omega_{\rm e}$; here $\alpha$ and $\beta$ are real-valued
functions in $C^1(\Sigma)$ with $\beta\ne0$ on $\Sigma$ (see Theorem~\ref{thm.sa2}).
These operators can be interpreted as operators with additional $\delta$ and
$\delta'$-potentials of strengths $\alpha$ and $\beta$, respectively.
Finally, we compare these operators with the elliptic operator $A_{\rm free}$
on $\RR^n$ associated with $\cL$ and the direct sums, $A_{\rm N,i}\oplus A_{\rm N,e}$,
$A_{\rm D,i}\oplus A_{\rm D,e}$, of the Neumann and Dirichlet
operators on the interior and exterior domains.
Using our abstract results on resolvent differences we obtain
\[
  A_{\rm N,i}\oplus A_{\rm N,e}
  \stackrel{\frac{3}{n-1}}{\text{\textendash\textendash\textendash\textendash\textendash}}
  A_{\rm\delta^{\prime},\beta}
  \stackrel{\frac{2}{n-1}}{\text{\textendash\textendash\textendash\textendash\textendash}}
  A_{\rm free}
  \stackrel{\frac{3}{n-1}}{\text{\textendash\textendash\textendash\textendash\textendash}}
  A_{\rm\delta,\alpha}
  \stackrel{\frac{2}{n-1}}{\text{\textendash\textendash\textendash\textendash\textendash}}
  A_{\rm D,i}\oplus A_{\rm D,e},
\]
where we used the notation from \eqref{notation_resolv_diff}; see \eqref{resdiff9}
and Theorems~\ref{thm4} and \ref{thm5}.

We mention here that, independently, V.~Ryzhov developed a concept that has
similarities to the concept of quasi boundary triples in \cite{R07,R09}.
Moreover, for extension theory of elliptic operators on non-smooth domains
and Dirichlet-to-Neumann maps we refer to the recent contributions  \cite{AGW10,AtE,AM08,GM08-1,GM08-2,GM09,G08,PR09}.
Spectral properties of resolvent differences using
pseudodifferential methods were recently also studied in \cite{Garxiv4,Garxiv5}.
Let us also mention other generalizations of boundary triples, e.g.\
\cite{AP04,A00,DHMS06,DHMS09,DM95,KK04,MS07,Mog06,Mog09,Posil04,P07}.

The contents of the paper is as follows.
In  Sections~\ref{2.1} and~\ref{2.2} we recall some preliminary material on
linear relations and operator ideals, and prove some lemmas that are needed later.
Our abstract results are contained in Section~\ref{sec3}. In Section~\ref{3.1} we recall
the concept of quasi boundary triples and some basic facts and complement these
with some results, e.g.\ about the imaginary part of the Weyl function.
We formulate and prove our results also for non-densely
defined symmetric operators and relations.
Section~\ref{3.2} contains the statement and proof of Krein's formula (Theorem~\ref{th.krein})
and its application to self-adjointness of certain extensions (Theorem~\ref{th.A_Th_self_adj}).
Section~\ref{3.3} comprises abstract theorems answering the question when resolvent differences
of different extensions are in some operator ideal.  The main results from Section~\ref{3.3} are generalised to
dissipative and accumulative extensions in Section~\ref{3.4}.

In Section~\ref{4.1} we construct a quasi boundary triple for elliptic operators on
bounded and exterior domains and construct self-adjoint realizations with non-local
boundary conditions.  Section~\ref{4.2} contains the results on spectral estimates for
resolvent differences of different self-adjoint realizations of the elliptic expression.
As a consequence we can also estimate differences of eigenvalues of these self-adjoint
realizations (Proposition~\ref{cor0}). Moreover, we prove a trace formula, which involves
the derivative of the Neumann-to-Dirichlet map and the operator that appears in the boundary condition.
Finally, in Section~\ref{4.3} we consider elliptic operators with $\delta$ and $\delta'$-potentials,
where we construct self-adjoint realizations and prove spectral estimates for resolvent
differences.

\section{Preliminaries}
\label{sec2}
\subsection{Notation and linear relations}\label{2.1}

Throughout this paper let $(\cH,(\cdot,\cdot))$ and $(\cG,(\cdot,\cdot))$ be
Hilbert spaces.
In general $\cH$ and $\cG$ are allowed to be non-separable, but in some theorems
separability is assumed.
The linear space of bounded linear operators defined on $\cH$ with values in
$\cG$ is denoted by $\cB(\cH,\cG)$. If $\cH=\cG$, we simply write $\cB(\cH)$.
We shall often deal with (closed) linear relations in $\cH$, that is, (closed)
linear subspaces of $\cH\times\cH$.
The set of closed linear relations in $\cH$ is denoted by $\widetilde\cC(\cH)$,
and for elements in a relation we usually use a vector notation.
Linear operators $T$ in $\cH$ are viewed as linear relations via their graphs.
The domain, range, kernel, multi-valued part and the inverse of a relation $T$
in $\cH$
are denoted by $\dom T$, $\ran T$, $\ker T$, $\mul T$ and $T^{-1}$,
respectively:
\begin{align*}
  \dom T \defequ \left\{f\in\cH\colon \exists\,f'\text{ with }\binom{f}{f'}\in
T\right\}, \\[1ex]
  \ran T \defequ \left\{f'\in\cH\colon \exists\,f\text{ with }\binom{f}{f'}\in
T\right\}, \displaybreak[0]\\[1ex]
  \ker T \defequ \left\{f\in\cH\colon \binom{f}{0}\in T\right\}, \\[1ex]
  \mul T \defequ \left\{f'\in\cH\colon \binom{0}{f'}\in T\right\}, \\[1ex]
  T^{-1} \defequ \left\{\binom{f'}{f}\colon \binom{f}{f'}\in T\right\}.
\end{align*}

Let $S\in\widetilde\cC(\cH)$ be a closed linear relation in $\cH$. The \emph{resolvent
set} $\rho(S)$ of $S$ is the set of all $\lambda\in\dC$ such that
$(S-\lambda)^{-1}\in\cB(\cH)$; the \emph{spectrum} $\sigma(S)$ of $S$
is the complement of $\rho(S)$ in $\dC$.
A point $\lambda\in\dC$ is an \emph{eigenvalue} of a linear relation $S$ if
$\ker(S-\lambda)\not=\{0\}$; we write $\lambda\in\sigma_p(S)$.
For a linear relation $S$ in $\cH$ the \emph{adjoint relation}
$S^*\in\widetilde\cC(\cH)$ is defined as
\[
  S^* \defeq \left\{\binom{g}{g'}\colon (f',g)=(f,g')  \text{ for all }
\binom{f}{f'}\in S\right\}.
\]
Note that this definition extends the usual definition
of the adjoint of a densely defined operator. A linear relation
$S$ in $\cH$ is said to be
\emph{symmetric} (\emph{self-adjoint}) if $S\subset S^*$ ($S=S^*$,
respectively). Recall that a symmetric relation is self-adjoint if
and only if $\ran(S-\lambda_\pm)=\cH$ holds for some $\lambda_+\in\CC^+$ and
some $\lambda_-\in\CC^-$, where $\CC^\pm\defeq\{z\in\CC\colon \pm\Im z > 0\}$;
in this case we have $\ran(S-\lambda)=\cH$ for all $\lambda\in\CC\backslash\RR$.

For a self-adjoint relation $S=S^*$ in $\cH$ the multi-valued part $\mul S$
is the orthogonal complement of $\dom S$ in $\cH$. Setting
$\cH_{\rm op} \defeq \overline{\dom S}$ and $\cH_\infty=\mul S$ one verifies
that $S$ can be written as the direct orthogonal sum of a (in general unbounded)
self-adjoint
operator $S_{\rm op}$ in the Hilbert space $\cH_{\rm op}$ and the ``pure''
relation $S_\infty=\bigl\{\binom{0}{f'}\colon f'\in\mul S\bigr\}$ in the Hilbert
space $\cH_\infty$,
\begin{equation*}
  S=S_{\rm op}\oplus S_\infty,
\end{equation*}
with respect to the decomposition $\cH=\cH_{\rm op}\oplus\cH_{\infty}$.
We say that a point $\lambda\in\dR$ belongs to the \emph{essential
spectrum} $\sess(S)$ of the self-adjoint relation $S$ if
$\lambda\in\sess(S_{\rm op})$. The essential spectrum of a closed operator $T$
in $\cH$ is
the set of $\lambda\in\dC$ such that $T-\lambda$ is not a Fredholm operator.

\subsection{Operator ideals and singular values}
\label{2.2}

In this section let $\cH$ and $\cK$ be separable Hilbert spaces.  Denote by
$\sS_\infty(\cH,\cK)$ the closed subspace of compact operators in
$\cB(\cH,\cK)$; if $\cH=\cK$, we simply write $\sS_\infty(\cH)$.
We define classes of operator ideals along the lines of \cite{P87}.

\begin{definition} \label{def.class_op_ideal}
Suppose that for every pair of Hilbert spaces $\cH$, $\cK$ we are given a subset
$\sA(\cH,\cK)$ of $\sS_\infty(\cH,\cK)$.  The set
\[
  \sA \defeq \bigcup_{\cH,\cK \text{ Hilbert spaces}}\!\!\sA(\cH,\cK)
\]
is said to be a \emph{class of operator ideals} if the following conditions are satisfied:
\begin{itemize}
  \item[(i)]
    the rank-one operators $x\mapsto(x,u)v$ are in $\sA(\cH,\cK)$ for all $u\in\cH$, $v\in\cK$;
  \item[(ii)]
    $A+B\in\sA(\cH,\cK)$ for $A,B\in\sA(\cH,\cK)$;
  \item[(iii)]
    $CAB\in\sA(\cH_1,\cK_1)$ for $A\in\sA(\cH,\cK)$, $B\in\cB(\cH_1,\cH)$, $C\in\cB(\cK,\cK_1)$.
\end{itemize}
Moreover, we write $\sA(\cH)$ for $\sA(\cH,\cH)$.
\end{definition}

If $\sA$ is a class of operator ideals, then the  sets $\sA(\cH,\cK)$ are two-sided operator ideals
for every pair $\cH$, $\cK$;
for the latter notion see also, e.g.\ \cite{GK69,P80}.
For two classes of operator ideals $\sA$, $\sB$ we define the product
\[
  \sA\cdot\sB \defeq \bigl\{T\colon \text{there exist }A\in\sA,B\in\sB\text{ so that } T=AB\bigr\}
\]
and the adjoint of $\sA$ by
\[
  \sA^* \defeq \bigl\{A^*\colon A\in\sA\bigr\}.
\]
These sets are again classes of operator ideals; see \cite{P87}.
The elements in the product $\sA\cdot\sB$ are denoted by $(\sA\cdot\sB)(\cH,\cK)$, so that
\[
  \sA\cdot\sB= \bigcup_{\cH,\cK \text{ Hilbert spaces}}\!\!(\sA\cdot\sB)(\cH,\cK)
  = \bigcup_{\cH,\cK,\cG \text{ Hilbert spaces}}\!\!\sA(\cG,\cK)\cdot\sB(\cH,\cG),
\]
where the products $\sA(\cG,\cK)\cdot\sB(\cH,\cG)$ are defined by
\[
  \sA(\cG,\cK)\cdot\sB(\cH,\cG)
  =\bigl\{T\colon \text{there exist }A\in\sA(\cG,\cK),B\in\sB(\cH,\cG)\text{ so that } T=AB\bigr\}.
\]
Later also the notation $\sA^*(\cK,\cH):=\{A^*\colon A\in \sA(\cH,\cK)\}$ will be used. Observe that
the adjoint $\sA^*$ of $\sA$ can be written in the form
\[
\sA^*=  \bigcup_{\cH,\cK \text{ Hilbert spaces}}\!\!\sA^*(\cK,\cH).
\]

The next lemma is used to extend assertions about resolvent differences
from one $\lambda$ to a bigger set of $\lambda$.

\begin{lemma}\label{resdifflemma}
Let $\sA$ be a class of operator ideals.
Moreover, let $H$ and $K$ be closed linear relations in a separable Hilbert space $\cH$.
If
\begin{equation}\label{resdiffideal}
  (H-\lambda)^{-1}-(K-\lambda)^{-1} \in\sA(\cH)
\end{equation}
for some $\lambda\in\rho(H)\cap\rho(K)$,
then \eqref{resdiffideal} holds for all $\lambda\in\rho(H)\cap\rho(K)$.
\end{lemma}

\begin{proof}
Let $\lambda,\mu\in\rho(H)\cap\rho(K)$ and define
\[
  E \defeq I+(\mu-\lambda)(H-\mu)^{-1},\qquad
  F \defeq I+(\mu-\lambda)(K-\mu)^{-1},
\]
which are both bounded operators in $\cH$.  The resolvent identity implies that
\begin{equation*}
  E(H-\lambda)^{-1} = (H-\mu)^{-1} \quad\text{and}\quad
  (K-\lambda)^{-1}F = (K-\mu)^{-1}.
\end{equation*}
Using this and the definition of $E$, $F$ one easily computes
\begin{equation*}
  (H-\mu)^{-1}-(K-\mu)^{-1}
  =E \,\bigl((H-\lambda)^{-1}-(K-\lambda)^{-1}\bigr)\,F.
\end{equation*}
Now the assertion follows from the ideal property of $\sA(\cH)$.
\end{proof}

Recall that the \emph{singular values} (or \emph{$s$-numbers}) $s_k(A)$,
$k=1,2,\dots$,
of a compact operator $A\in\sS_\infty(\cH,\cK)$ are defined as the eigenvalues
$\lambda_k(\vert A\vert)$ of the
non-negative compact operator $\vert
A\vert=(A^*A)^\frac{1}{2}\in\sS_\infty(\cH)$, which are
enumerated in decreasing order and with multiplicities taken into account. Note
that for a
non-negative operator $A\in\sS_\infty(\cH)$ the eigenvalues $\lambda_k(A)$ and
singular values
$s_k(A)$, $k=1,2,\dots$, coincide. Let $A\in\sS_\infty(\cH,\cK)$ and assume that
$\cH$ and $\cK$ are infinite dimensional Hilbert spaces.
Then there exist orthonormal systems
$\{\varphi_1,\varphi_2,\dots\}$ and $\{\psi_1,\psi_2,\dots\}$
in $\cH$ and $\cK$, respectively, such that $A$ admits the \emph{Schmidt expansion}
\begin{equation}\label{schmidt}
  A=\sum_{k=1}^{\infty}s_k(A)(\,\cdot\,,\varphi_k)\psi_k.
\end{equation}
It follows, for instance, from \eqref{schmidt} and the corresponding expansion for
$A^*\in\sS_\infty(\cK,\cH)$ that the singular values of $A$ and $A^*$ coincide:
$s_k(A)=s_k(A^*)$ for $k=1,2,\dots$; see, e.g.\ \cite[II.\S2.2]{GK69}.
Moreover, if $\cG$ and $\cL$ are separable Hilbert spaces, $B\in\cB(\cG,\cH)$
and $C\in\cB(\cK,\cL)$,
then the estimates
\begin{equation}\label{sest}
  s_k(AB)\leq\Vert B\Vert s_k(A)\quad\text{and}\quad s_k(CA)\leq\Vert C\Vert
  s_k(A),\qquad k=1,2,\dots,
\end{equation}
hold.  If, in addition, $B\in\sS_\infty(\cG,\cH)$ we have
\begin{equation}\label{smn}
  s_{m+n-1}(AB)\leq s_m(A)s_n(B),\qquad m,n=1,2\dots.
\end{equation}
The proofs of the inequalities \eqref{sest} and \eqref{smn} are the same as in
\cite[II.\S2.1 and \S2.2]{GK69}
where these facts are shown for operators acting in the same space.

Recall that the \emph{Schatten--von Neumann ideals} $\sS_p(\cH,\cK)$ are defined
by
\[
  \frS_p(\cH,\cK) \defeq \biggl\{A\in\sS_\infty(\cH,\cK)\colon\sum_{k=1}^\infty
  (s_k(A))^p < \infty\biggr\},\qquad p>0.
\]
Besides the Schatten--von Neumann ideals also
the operator ideals
\begin{equation*}
  \begin{aligned}
    \sS_{r,\infty}(\cH,\cK) \defequ \bigl\{A\in\sS_\infty(\cH,\cK)\colon s_k(A)
  = O(k^{-r}),\,k\to\infty\bigr\},\\[1ex]
    \sS_{r,\infty}^{(0)}(\cH,\cK) \defequ \bigl\{A\in\sS_\infty(\cH,\cK)\colon
    s_k(A) = o(k^{-r}),\,k\to\infty\bigr\},
  \end{aligned}
  \qquad r>0,
\end{equation*}
will play an important role later on.
The sets
\[
  \sS_p \defeq \bigcup_{\cH,\cK}\sS_p(\cH,\cK), \quad
  \sS_{r,\infty} \defeq \bigcup_{\cH,\cK}\sS_{r,\infty}(\cH,\cK), \quad
  \sS_{r,\infty}^{(0)} \defeq \bigcup_{\cH,\cK}\sS_{r,\infty}^{(0)}(\cH,\cK)
\]
are classes of operator ideals in the sense of Definition~\ref{def.class_op_ideal}.

We refer the reader to \cite[III.\S7 and III.\S14]{GK69} for a detailed study
of the classes $\sS_p$, $\sS_{r,\infty}$ and $\sS_{r,\infty}^{(0)}$.
We list only some basic and well-know properties, which will be useful for us.
It follows from $s_k(A)=s_k(A^*)$ that $\sS_p^*=\sS_p$, $\sS_{r,\infty}^*=\sS_{r,\infty}$
and $\bigl(\sS_{r,\infty}^{(0)}\bigr)^*=\sS_{r,\infty}^{(0)}$ hold.

\begin{lemma}\label{splemma}
Let $p,q,r,s>0$. Then the following relations are true:
\begin{itemize}
  \item[(i)]
    $\sS_p\subset\sS_{p^{-1},\infty}^{(0)}\subset\sS_{p^{-1},\infty}$;
  \item[(ii)]
    $\sS_{r,\infty}\subset\sS_q$ for all $q>r^{-1}$;
  \item[(iii)]
    $\sS_{r,\infty}\cdot\sS_{s,\infty}=\sS_{r+s,\infty}$;
  \item[(iv)]
    $\sS_{r,\infty}^{(0)}\cdot\sS_{s,\infty}^{(0)}=\sS_{r+s,\infty}^{(0)}$;
  \item[(v)]
    $\sS_p\cdot\sS_q = \sS_r$ if $\frac{1}{r}=\frac{1}{p}+\frac{1}{q}$.
\end{itemize}
\end{lemma}

\begin{proof}
The first inclusion in (i) is a consequence of the fact that $\sum (s_k(A))^p
<\infty$ implies $k(s_k(A))^p\rightarrow 0$ for
$k\rightarrow\infty$, and the second inclusion is clear. Assertion (ii) follows
immediately from the definitions. In order to verify (iii) let $r,s>0$ and let
$A\in\sS_{r,\infty}(\cH,\cK)$ and $B\in\sS_{s,\infty}(\cG,\cH)$, that is, the
inequalities $s_n(A)\leq c_a n^{-r}$ and $s_n(B)\leq c_b n^{-s}$, $n\in\dN$,
hold with some constants $c_a,c_b>0$.
From \eqref{smn} we obtain
\begin{equation*}
  s_{2n}(AB)\leq s_{2n-1}(AB)\leq s_n(A) s_n(B)\leq \frac{c_ac_b}{n^r n^s}
  \leq \frac{2^{r+s}c_ac_b}{(2n)^{r+s}}
  \leq \frac{2^{r+s}c_ac_b}{(2n-1)^{r+s}}\,,
\end{equation*}
which implies $AB\in\sS_{r+s,\infty}(\cG,\cK)$.
In order to show equality, let $A\in\sS_{r+s,\infty}(\cH,\cK)$ with
Schmidt expansion
\[
  A=\sum_{k} s_k(A)(\,\cdot\,,\varphi_k)\psi_k.
\]
Define operators $B\colon\cH\to\cK$ and $C\colon\cH\to\cH$ by
\[
  B=\sum_{k} \bigl(s_k(A)\bigr)^{\frac{r}{r+s}}(\,\cdot\,,\varphi_k)\psi_k, \qquad
  C=\sum_{k} \bigl(s_k(A)\bigr)^{\frac{s}{r+s}}(\,\cdot\,,\varphi_k)\varphi_k.
\]
The relations $A=BC$, $B\in\sS_{r,\infty}(\cH,\cK)$, $C\in\sS_{s,\infty}(\cH,\cH)$
show that $A\in\sS_{r,\infty}\cdot\sS_{s,\infty}$.
The same arguments as in (iii) can be used to show (iv).
The inclusion ``$\subset$'' in (v) follows from \cite[III.\S7.2]{GK69}.
The converse inclusion follows in a similar way as in (iii).
\end{proof}

Sometimes we need also the notion of a symmetrically normed ideal: a two-sided
ideal $\sA(\cH,\cG)$ is a \emph{symmetrically normed ideal} if it is a Banach
space with respect to some norm $\|\cdot\|_{\sA}$ such that $\|CAB\|_\sA \le
\|C\|\,\|A\|_\sA\,\|B\|$ for $A\in\sA(\cH,\cG)$, $B\in\cB(\cH)$, $C\in\cB(\cG)$
and $\|A\|_\sA=s_1(A)$ for rank one operators $A$; see \cite[III.\S2.1,\S2.2]{GK69}.
If a class of operator ideals consists of symmetrically normed ideals, then
we call it a \emph{class of symmetrically normed ideals}.
The classes $\sS_p$, $\sS_{r,\infty}$ and $\sS_{r,\infty}^{(0)}$ are
classes of symmetrically normed ideals for $p\ge 1$ and $r<1$;
see \cite[III.\S7 and \S14]{GK69}.

The following lemma is needed in the proof of Proposition~\ref{pr.suff_cond}.

\begin{lemma}\label{ideallem}
Let $\frA(\cG)$ be a symmetrically normed ideal of $\cB(\cG)$, let
$C\in\cB(\cH)$ and assume
that $A\in\frA(\cG)$ admits the factorization $A=B^*B$ with $B\in\cB(\cG,\cH)$.
Then
also $B^*CB\in\frA(\cG)$.
\end{lemma}

\begin{proof}
If $\cG$ is finite-dimensional, then the assertion is trivial.
So let us assume that $\cG$ is infinite-dimensional.
Observe first that $(s_k(A))^\frac{1}{2}=s_k(B)=s_k(B^*)$ and
$\lambda_k(A)=s_k(A)$
hold for all $k=1,2,\dots$.
Together with \eqref{smn} and the first inequality in \eqref{sest} we obtain
\begin{equation*}
s_{2n}\bigl(B^*CB\bigr)\leq
s_{2n-1}\bigl(B^*CB\bigr)\leq s_n(B^*)
s_n(CB)\leq\Vert C\Vert s_n(A)
\end{equation*}
for $n=1,2,\dots$.
Let us write the non-negative compact operator $A\in\frA(\cG)$ in the form
\begin{equation*}
A=\sum_{k=1}^{\infty} \lambda_k(A)(\cdot\,,\varphi_k)\varphi_k
\end{equation*}
with an orthonormal bases $\{\varphi_1,\varphi_2,\dots\}$
of eigenvectors corresponding to the eigenvalues $\lambda_k(A)$.

Define operators $V_1, V_2 \in \cB(\cG)$ by
\[
  V_1: \left\{ \begin{aligned}
    \varphi_{2k-1} &\mapsto \varphi_k, \\[1ex]
    \varphi_{2k} &\mapsto 0,
  \end{aligned} \right. \hspace*{7ex}
  V_2: \left\{ \begin{aligned}
    \varphi_{2k-1} &\mapsto 0, \\[1ex]
    \varphi_{2k} &\mapsto \varphi_k,
  \end{aligned} \right. \hspace*{5ex} k\in\NN.
\]
Then the non-negative operator
\[
  \widetilde A \defeq V_1AV_1^* + V_2AV_2^*
  = \sum_{k=1}^{\infty} \lambda_k(A)
\bigl((\cdot\,,\varphi_{2k-1})\varphi_{2k-1}+(\cdot\,,\varphi_{2k})\varphi_{2k}
\bigr)
\]
belongs to $\frA(\cG)$, and its eigenvalues satisfy
$\lambda_{2n-1}(\widetilde A) = \lambda_{2n}(\widetilde A) = \lambda_n(A)$.
Hence we have $s_k(B^*CB)\leq \|C\| s_k(\widetilde A)$, $k=1,2,\dots,$
and the claim follows from \cite[III.\S2.2]{GK69}.
\end{proof}

\section{Quasi boundary triples and Krein's formula}
\label{sec3}
\subsection{Quasi boundary triples, $\gamma$-fields and Weyl
functions}\label{3.1}

The notion of quasi boundary triples was introduced in connection with elliptic
boundary value problems
by the first two authors in \cite{BL07} as a generalization of the
notion of ordinary and generalized boundary triples from
\cite{Br76,BGP08,DM91,DM95,GG91,Ko75,M92}.
We emphasize that a
quasi boundary triple is in general not a boundary relation in the sense of
\cite{DHMS06}.
Let us start by recalling the basic definition from \cite{BL07}.

\begin{definition}
\label{def:qbt}
Let $A$ be a closed symmetric relation in a Hilbert space $(\cH,(\cdot,\cdot))$.
We say that $\{\cG,\Gamma_0,\Gamma_1\}$ is a \emph{quasi boundary triple}
for $A^*$ if $\Gamma_0$ and $\Gamma_1$ are linear mappings defined on a dense
subspace
$T$ of $A^*$ with values in the Hilbert space $(\cG,(\cdot,\cdot))$ such that
$\Gamma \defeq \binom{\Gamma_0}{\Gamma_1}\colon
T\rightarrow\cG\times\cG$ has dense range, $\ker\Gamma_0$ is self-adjoint and
the identity
\begin{equation}\label{green1}
  (f',g)-(f,g')
  =(\Gamma_1\hat f,\Gamma_0\hat g)-(\Gamma_0\hat f,\Gamma_1\hat g)
\end{equation}
holds for all $\hat f=\binom{f}{f'},\,\hat g=\binom{g}{g'}\in T$.
\end{definition}

We recall some basic facts for quasi boundary triples, which can be found
in \cite{BL07}.
Let $A$ be a closed symmetric relation in the Hilbert space $\cH$.
We note first that a quasi boundary triple for $A^*$ exists if and only if
the deficiency indices $n_\pm(A)=\dim\ker(A^*\mp i)$ of $A$ coincide.
In the following, let
$\{\cG,\Gamma_0,\Gamma_1\}$ be a quasi boundary triple for $A^*$.
Then $A$ coincides with $\ker\Gamma=\ker\Gamma_0\cap\ker\Gamma_1$
and $\Gamma=\binom{\Gamma_0}{\Gamma_1}$ regarded as a mapping from
$\cH\times\cH$ into $\cG\times\cG$ is closable, cf.\
\cite[Proposition~2.2]{BL07}. Furthermore, as an immediate consequence of \eqref{green1},
the extension $A_1:=\ker\Gamma_1$ is a symmetric relation in $\cH$.

The next theorem (cf.\ \cite[Theorem~2.3]{BL07}) contains a sufficient
condition for a
triple $\{\cG,\Gamma_0,\Gamma_1\}$ to be a quasi boundary triple.  One does not
have to
show that $T$ is dense in $A^*$, but this follows from the theorem.
Moreover, one only has to show that $\ker\Gamma_0$ \emph{contains} a
self-adjoint relation.

\begin{theorem}\label{suff_cond_qbt}
Let $\cH$ and $\cG$ be Hilbert spaces and let $T$ be a linear relation in $\cH$.
Assume that
$\Gamma_0,\Gamma_1\colon T\rightarrow\cG$ are linear mappings such that the
following
conditions are satisfied:
\begin{itemize}
\setlength{\itemsep}{1.2ex}
\item[(a)]
$\ker\Gamma_0$ contains a self-adjoint relation;
\item[(b)]
$\Gamma\defeq \binom{\Gamma_0}{\Gamma_1}\colon T\rightarrow\cG\times\cG$
has dense range;
\item[(c)]
identity \eqref{green1} holds for all $\hat f=\binom{f}{f'},\,\hat
g=\binom{g}{g'}\in T$.
\end{itemize}
Then the following assertions hold.
\begin{itemize}
\item[(i)]
$A\defeq\ker\Gamma$ is a closed symmetric relation in $\cH$ and
$\{\cG,\Gamma_0,\Gamma_1\}$
is a quasi boundary triple for $A^*$;
\item[(ii)]
$T=A^*$ if and only if $\ran\Gamma=\cG\times\cG$.
\end{itemize}
\end{theorem}

Let again $A$ be a closed symmetric relation in $\cH$ and let
$\{\cG,\Gamma_0,\Gamma_1\}$ be a quasi boundary triple for $A^*$ with $T=\dom\Gamma$.
Next we consider extensions of $A$ which are restrictions of $T$ defined by
some abstract boundary condition.
For a linear relation $\Theta\subset\cG\times\cG$ we define
\begin{equation}\label{atheta}
  A_\Theta \defeq \bigl\{\hat f\in T\colon \Gamma\hat
f\in\Theta\bigr\}=\Gamma^{-1}(\Theta).
\end{equation}
If $\Theta\subset\cG\times\cG$ is an operator, then we have
\begin{equation}\label{athetaop}
  A_\Theta=\ker(\Gamma_1-\Theta\Gamma_0),
\end{equation}
and \eqref{athetaop} holds also for linear relations $\Theta$ in $\cG$ if the
product and the sum on the right-hand side are understood in the sense of linear relations.
Observe that the self-adjoint relation $A_0 \defeq \ker\Gamma_0$ corresponds to
the purely multi-valued relation $\Theta=0^{-1}=\bigl\{\binom{0}{g}\colon g\in\cG\bigr\}$ in $\cG$.
This little inconsistency in the notation should not lead to misunderstandings.
It is not difficult to see that $\Theta\subset\Theta^*$ implies $A_\Theta\subset
A_\Theta^*$.
However, in contrast to ordinary boundary triples, self-adjointness of $\Theta$
does not imply self-adjointness or essential self-adjointness of $A_\Theta$;
cf.\ \cite[Proposition~4.11]{BL07} for a counterexample,
\cite[Proposition~2.4]{BL07}
and Theorem~\ref{th.A_Th_self_adj} below for sufficient conditions.

In the following we set $\cG_0 \defeq \ran\Gamma_0$ and $\cG_1 \defeq
\ran\Gamma_1$.
Because $\ran\Gamma$ is dense in $\cG\times\cG$, it follows that $\cG_0$ and
$\cG_1$
are dense subspaces of $\cG$.
Since $A_0 \defeq \ker\Gamma_0\subset T=\dom\Gamma$ is a self-adjoint extension
of $A$ in $\cH$, the decomposition
\begin{equation*}
  T=A_0\,\widehat +\, \hat\cN_{\lambda,T},\qquad \hat\cN_{\lambda,T} \defeq
  \left\{\binom{f_\lambda}{\lambda f_{\lambda}}
  \colon f_\lambda\in\cN_\lambda(T) \defeq \ker(T-\lambda)\right\},
\end{equation*}
holds for all $\lambda\in\rho(A_0)$. Here $\widehat +$ denotes the direct sum of the subspaces $A_0$ and $\hat\cN_{\lambda,T}$.
It follows that the mapping
\begin{equation*}
  \bigl(\Gamma_0\upharpoonright \hat\cN_{\lambda,T}\bigr)^{-1}\colon
  \cG_0\rightarrow\hat\cN_{\lambda,T},\quad\lambda\in\rho(A_0),
\end{equation*}
is well defined and bijective.
Denote the orthogonal projection in $\cH\oplus\cH$ onto the first
component of $\cH\oplus\cH$ by $\pi_1$.

\begin{definition}
\label{def:gammaWeyl}
Let $A$ be a closed symmetric relation in $\cH$ and let
$\{\cG,\Gamma_0,\Gamma_1\}$ be a quasi boundary triple for $A^*$
with $A_0=\ker\Gamma_0$.
Then the (operator-valued) functions $\gamma$ and $M$ defined by
\begin{equation*}
  \gamma(\lambda) \defeq \pi_1\bigl(\Gamma_0\upharpoonright
\hat\cN_{\lambda,T}\bigr)^{-1}
  \,\,\,\text{and}\,\,\,
  M(\lambda) \defeq \Gamma_1\bigl(\Gamma_0\upharpoonright
\hat\cN_{\lambda,T}\bigr)^{-1},\quad
  \lambda\in\rho(A_0),
\end{equation*}
are called the $\gamma$\emph{-field} and \emph{Weyl function} corresponding to
the quasi boundary triple $\{\cG,\Gamma_0,\Gamma_1\}$.
\end{definition}

Note that $\gamma(\lambda)$ is a mapping from $\cG_0$ to $\cH$, and $M(\lambda)$
is a mapping from $\cG_0$ to $\cG_1\subset\cG$ for $\lambda\in\rho(A_0)$.
These definitions coincide with the definition of the $\gamma$-field and
Weyl function or Weyl family in the case that $\{\cG,\Gamma_0,\Gamma_1\}$ is an ordinary
boundary triple, generalized boundary triple or a boundary relation as
in \cite{DHMS06,DHMS09,DM91,DM95}.
In the next proposition we collect some properties of the $\gamma$-field
and the Weyl function of a quasi boundary triple, which are extensions of
well-known properties of the $\gamma$-field and Weyl function of an
ordinary boundary triple.
The first six items were stated and proved in \cite[Proposition~2.6]{BL07}.

\begin{proposition} \label{gammaprop}
Let $A$ be a closed symmetric relation in $\cH$ and let
$\{\cG,\Gamma_0,\Gamma_1\}$ be a quasi boundary triple for $A^*$ with
$\gamma$-field $\gamma$ and Weyl function $M$. For
$\lambda,\mu\in\rho(A_0)$ the following assertions hold.
\begin{itemize}
\item[(i)] 
$\gamma(\lambda)$ is a densely defined bounded
operator from $\cG$ into $\cH$ with domain $\dom\gamma(\lambda)=\cG_0$,
$\overline{\gamma(\lambda)}\in\cB(\cG,\cH)$,
the function $\lambda\mapsto\gamma(\lambda)g$ is holomorphic on $\rho(A_0)$
for every $g\in\cG_0$, and the relation
\begin{equation*}
  \gamma(\lambda)=\bigl(I+(\lambda-\mu)(A_0-\lambda)^{-1}\bigr)\gamma(\mu)
\end{equation*}
holds.
\item[(ii)] 
$\gamma(\overline\lambda)^*\in\cB(\cH,\cG)$,
$\ran\gamma(\overline\lambda)^*\subset\cG_1$
and for all $h\in\cH$
we have
\begin{equation*}
  \gamma(\overline\lambda)^* h=\Gamma_1
  \begin{pmatrix} (A_0-\lambda)^{-1}h\\[0.5ex]
  (I+\lambda(A_0-\lambda)^{-1})h
  \end{pmatrix}.
\end{equation*}
\item[(iii)] 
$M(\lambda)$ maps $\cG_0$ into $\cG_1$. If, in
addition, $A_1 \defeq \ker\Gamma_1\subset T$ is a self-adjoint relation in $\cH$
and $\lambda\in\rho(A_1)$, then $M(\lambda)$ maps $\cG_0$
onto $\cG_1$.
\item[(iv)] 
$M(\lambda)\Gamma_0\hat f_{\lambda}=\Gamma_1\hat
f_\lambda$ for all $\hat f_\lambda\in \hat\cN_{\lambda,T}$.
\item[(v)] 
$M(\lambda)\subset M(\overline\lambda)^*$ and
$M(\lambda)-M(\mu)^*=(\lambda-\overline\mu)\gamma(\mu)^*\gamma(\lambda)$. The
function $\lambda\mapsto M(\lambda)$ is holomorphic in the sense that
it can be written as the sum of the possibly unbounded operator $\Real
M(\mu)$ and a bounded holomorphic operator function,
\begin{equation*}
\begin{split}
  M(\lambda)=&\Real M(\mu) \\
  & \quad+\gamma(\mu)^*\bigl((\lambda-\Real\mu)
  +(\lambda-\mu)(\lambda-\overline\mu)(A_0-\lambda)^{-1}
  \bigr)\gamma(\mu).
\end{split}
\end{equation*}
\item[(vi)] 
$\Imag
M(\lambda)=\tfrac{1}{2i}(M(\lambda)-M(\overline\lambda))$ is a densely
defined bounded operator in $\cG$. For $\lambda\in\dC^+ (\dC^-)$ the
operator $\Imag M(\lambda)$ is positive (negative, respectively).
\item[(vii)] 
For $x\in\cG_0$, the function $\lambda\mapsto M(\lambda)x$ is differentiable on
$\rho(A_0)$ and
\begin{equation} \label{deriv_M}
  \frac{d}{d\lambda}M(\lambda)x = \gamma(\ov\lambda)^*\gamma(\lambda)x, \qquad
\lambda\in\rho(A_0).
\end{equation}
\end{itemize}
\end{proposition}

\begin{proof}
Since (i)--(vi) were already proved in \cite[Proposition~2.6]{BL07}, we only
have to show (vii).
Let $x\in\cG_0$ and $\lambda_0,\lambda\in\rho(A_0)$.  It follows from (v) with
$\mu=\ov{\lambda_0}$ that
\[
  \frac{1}{\lambda-\lambda_0}\bigl(M(\lambda)x-M(\lambda_0)x\bigr)
  = \frac{1}{\lambda-\lambda_0}\bigl(M(\lambda)x-M(\ov{\lambda_0})^*x\bigr)
  = \gamma(\ov{\lambda_0})^*\gamma(\lambda)x.
\]
If we let $\lambda\to\lambda_0$, then the right-hand side converges, which shows
that the
derivative exists and that \eqref{deriv_M} is true for $\lambda$ replaced by
$\lambda_0$.
\end{proof}

\begin{remark} \label{re.deriv_M}
Note that the closure of the operator on the right-hand side of \eqref{deriv_M}
is
$\gamma(\ov{\lambda})^*\ov{\gamma(\lambda)}$, which is in $\cB(\cG)$.  Hence
also $\frac{d}{d\lambda}M(\lambda)$ has a bounded, everywhere defined closure,
which
we denote by $\ov{M'(\lambda)}$.  With this notation we have the identity
\begin{equation}\label{deriv_closM}
  \ov{M'(\lambda)} = \gamma(\ov\lambda)^*\ov{\gamma(\lambda)}.
\end{equation}
\end{remark}

The next lemma on the closure of the values $M(\lambda)$ of the Weyl function
$M$
will be useful in Sections \ref{3.2} and \ref{3.3}.

\begin{lemma} \label{pr.Mpos}
Let $A$ be a closed symmetric relation in $\cH$ and let
$\{\cG,\Gamma_0,\Gamma_1\}$ be a quasi boundary triple for $A^*$ with
$A_0=\ker\Gamma_0$,
$\gamma$-field $\gamma$ and Weyl function $M$. If
$M(\lambda_0)$ is bounded for some $\lambda_0\in\rho(A_0)$, then $M(\lambda)$ is
bounded
for all $\lambda\in\rho(A_0)$.  In this case,
\begin{equation}\label{e3}
  \frac1{\Im\lambda}\Im\ov{M(\lambda)} > 0,\qquad \lambda\in\dC\backslash\dR,
\end{equation}
and, in particular, $\ker\ov{M(\lambda)} = \{0\}$ for
$\lambda\in\dC\backslash\dR$.
\end{lemma}

\begin{proof}
It follows from Proposition~\ref{gammaprop}(v) that $M(\lambda)$ is bounded for
all $\lambda\in\rho(A_0)$ if $M(\lambda_0)$ is bounded for one
$\lambda_0\in\rho(A_0)$.
For the inequality \eqref{e3}, assume without loss of generality that
$\Im\lambda>0$.
Observe that $\Im\ov{M(\lambda)} = \ov{\Im M(\lambda)}$ since $M(\lambda)$
is bounded.  It follows from Proposition~\ref{gammaprop}(vi) that
$\Im M(\lambda)>0$.  Hence it is sufficient to show that
\begin{equation}\label{kern}
  \ker\bigl(\Im\ov{M(\lambda)}\bigr) = \{0\}.
\end{equation}
Let $x\in\ker(\Im\ov{M(\lambda)}) = \ker(\ov{\Im M(\lambda)})$.
Then there exist $x_n\in\dom M(\lambda)$ so that $x_n \to x$ and $(\Im
M(\lambda))x_n \to 0$
when $n\rightarrow\infty$.
By Proposition~\ref{gammaprop}(v) we have
\begin{equation*}
  \bigl((\Im M(\lambda))x_n,x_n\bigr)
  = \bigl((\Im\lambda)\gamma(\lambda)^*\gamma(\lambda)x_n,x_n\bigr)
  = (\Im\lambda)\|\gamma(\lambda)x_n\|^2,
\end{equation*}
and since $\Im\lambda\ne0$, this implies that $\gamma(\lambda)x_n\to0$.
Let $\hat u_n \defeq \bigl(\begin{smallmatrix}\gamma(\lambda)x_n \\
\lambda\gamma(\lambda)x_n\end{smallmatrix}\bigr)$; then
\[
  \hat u_n\in\hat\cN_{\lambda,T}, \qquad
  \hat u_n  \to 0, \qquad\text{and}\qquad
  \Gamma_0 \hat u_n = x_n \to x\quad\text{for}\,\,\,n\rightarrow\infty.
\]
Moreover, by Proposition~\ref{gammaprop}(iv) and the boundedness
of $M(\lambda)$ we obtain that $\Gamma_1 \hat u_n = M(\lambda)\Gamma_0 \hat
u_n=M(\lambda)x_n \to \ov{M(\lambda)}x$ when $n\rightarrow\infty$,
and therefore
\[
\Gamma \hat u_n = \vect{\Gamma_0 \hat u_n}{\Gamma_1 \hat u_n}
  \to \vect{x}{\ov{M(\lambda)}x},\qquad n\rightarrow\infty.
\]
Now $\hat u_n\rightarrow 0$ and the closability of $\Gamma$ imply that $x=0$,
that is, \eqref{kern} holds.
The last assertion follows easily from \eqref{e3}.
\end{proof}

For the rest of this subsection we assume that $A$ is a closed symmetric
relation
in a separable Hilbert space $\cH$.
If $\{\cG,\Gamma_0,\Gamma_1\}$ is a quasi boundary triple for $A^*$, then also
the Hilbert space $\cG$ is separable. The following proposition shows that,
roughly speaking, the property of $\ov{\gamma(\lambda)}$, $\gamma(\lambda)^*$
and $\ov{M(\lambda)}$
belonging to some two-sided operator ideal is independent of $\lambda$.

\begin{proposition} \label{pr.suff_cond}
Let $A$ be a closed symmetric relation in a separable Hilbert space $\cH$ and
let
$\{\cG,\Gamma_0,\Gamma_1\}$ be a quasi boundary triple for $A^*$ with
$A_0=\ker\Gamma_0$,
$\gamma$-field $\gamma$ and Weyl function $M$.
Moreover, let $\sA$ be a class of operator ideals.
Then the following assertions are true.
\begin{itemize}
\item[(i)]
  If\, $\ov{\gamma(\lambda_0)}\in\frA(\cG,\cH)$ for some $\lambda_0\in\rho(A_0)$,
  then $\ov{\gamma(\lambda)}\in\frA(\cG,\cH)$ for all $\lambda\in\rho(A_0)$.
\item[(ii)]
  If\, $\gamma(\lambda_0)^*\in\frA^*(\cH,\cG)$ for some $\lambda_0\in\rho(A_0)$,
  then $\gamma(\lambda)^*\in\frA^*(\cH,\cG)$ for all $\lambda\in\rho(A_0)$.
\item[(iii)]
  Assume that $\sA$ is even a class of symmetrically normed ideals.
  If $\ov{M(\lambda_0)}\in\sA(\cG)$ for some $\lambda_0\in\dC\backslash\dR$, then
  $\ov{M(\lambda)}\in\frA(\cG)$ for all $\lambda\in\rho(A_0)$.
\end{itemize}
\end{proposition}

\begin{proof}
(i)
It follows immediately from $I+(\lambda-\lambda_0)(A_0-\lambda)^{-1}\in\cB(\cH)$
and Proposition~\ref{gammaprop}(i) that
\[
  \ov{\gamma(\lambda)}
  = \bigl(I+(\lambda-\lambda_0)(A_0-\lambda)^{-1}\bigr)\ov{\gamma(\lambda_0)}
\]
holds for all $\lambda,\lambda\in\rho(A_0)$. The ideal property directly implies
the assertion.

(ii)
If $\gamma(\lambda_0)^*\in\sA(\cH,\cG)$, then
$\ov{\gamma(\lambda_0)}=\gamma(\lambda_0)^{**}\in\sA^*(\cG,\cH)$.
By (i) this implies that $\ov{\gamma(\lambda)}\in\sA^*(\cG,\cH)$ for
all $\lambda\in\rho(A_0)$ and hence
$\gamma(\lambda)^*\in\sA(\cH,\cG)$ for all $\lambda\in\rho(A_0)$.

(iii)
Assume that $\ov{M(\lambda_0)}\in\frA(\cG)$ for some
$\lambda_0\in\dC\backslash\dR$. Then also
$\Re\ov{M(\lambda_0)}$ and $\Im\ov{M(\lambda_0)}$ belong to $\frA(\cG)$, and
by Proposition~\ref{gammaprop}(v) we have
\begin{equation*}
  \frac1{\Im\lambda_0}\Im\ov{M(\lambda_0)} =
\gamma(\lambda_0)^*\ov{\gamma(\lambda_0)}\in\frA(\cG).
\end{equation*}
Since $\ov{\gamma(\lambda_0)}\in\cB(\cG,\cH)$ and
$\gamma(\lambda_0)^*=\ov{\gamma(\lambda_0)}^{\,*}\in\cB(\cH,\cG)$, we can use
Lemma~\ref{ideallem} to conclude that for every $\lambda\in\rho(A_0)$ also
\begin{equation}\label{yxcv}
  \gamma(\lambda_0)^*\bigl((\lambda-\Re\lambda_0)+(\lambda-\lambda_0)
  (\lambda-\ov{\lambda_0})(A_0-\lambda)^{-1}\bigr)\ov{\gamma(\lambda_0)}
  \in\frA(\cG).
\end{equation}
It follows from Proposition~\ref{gammaprop}(v) that for $\lambda\in\rho(A_0)$ we
have
\begin{equation*}
  \ov{M(\lambda)}=\Re\ov{M(\lambda_0)}
  +\gamma(\lambda_0)^*\bigl((\lambda-\Re\lambda_0)+(\lambda-\lambda_0)
  (\lambda-\ov{\lambda_0})(A_0-\lambda)^{-1}\bigr)\ov{\gamma(\lambda_0)}.
\end{equation*}
Therefore $\Re\ov{M(\lambda_0)}\in\frA(\cG)$ and \eqref{yxcv} imply that
$\ov{M(\lambda)}\in\frA(\cG)$ for all $\lambda\in\rho(A_0)$.
\end{proof}

\begin{remark}\label{canbeuseful}
Note that in Proposition~\ref{pr.suff_cond}(iii) it is assumed that $\lambda_0$
is non-real.
However, it follows from the proof of Proposition~\ref{pr.suff_cond}(iii) that
the assumptions $\overline{M(\lambda_1)}\in\frA(\cG)$ and
$\gamma(\lambda_1)^*\overline{\gamma(\lambda_1)}\in\frA(\cG)$ for some
$\lambda_1\in\dR\cap\rho(A_0)$ also yield $\overline{M(\lambda)}\in\frA(\cG)$
for all $\lambda\in\rho(A_0)$.  However, the assumption
$\ov{M(\lambda_1)}\in\frA(\cG)$
for some $\lambda_1\in\RR\cap\rho(A_0)$ alone does not imply that
$\ov{M(\lambda)}\in\frA(\cG)$ for all $\lambda\in\rho(A_0)$.
\end{remark}

\begin{proposition} \label{pr.suff_cond2}
Let $\sA$ be a class of operator ideals.  Moreover, let $\gamma$ be the $\gamma$-field
associated with some quasi boundary triple $\{\cG,\Gamma_0,\Gamma_1\}$,
let $\wt\cG_1$ be a Hilbert space such that $\cG_1\subset\wt\cG_1\subset\cG$ and
the embedding $\iota_{\wt\cG_1\to\cG}$ belongs to $\frA(\widetilde\cG_1,\cG)$.
Then
\begin{equation}\label{lustig}
  \gamma(\lambda)^*\in\frA(\cH,\cG)
\end{equation}
for all $\lambda\in\rho(A_0)$.
\end{proposition}

\begin{proof}
For every $\lambda\in\rho(A_0)$ we have $\gamma(\lambda)^*\in\cB(\cH,\cG)$ and
$\ran\gamma(\lambda)^*\subset\cG_1$ by Proposition~\ref{gammaprop}(ii).
Hence $\gamma(\lambda)^*$ is closed as an operator from  $\cH$ to $\cG$.
Because $\iota_{\wt\cG_1\to\cG}$ is bounded, $\gamma(\lambda)^*$ regarded as
an operator from $\cH$ into $\widetilde\cG_1$ is also closed and hence bounded
by the closed graph theorem, that is,
$\gamma(\lambda)^*\in\cB(\cH,\widetilde\cG_1)$.
Hence, by the ideal property, \eqref{lustig} holds.
\end{proof}

\subsection{Krein's formula and self-adjoint extensions}\label{3.2}

The following theorem and corollary contain a variant of Krein's formula for
the resolvents of canonical extensions parameterized
with the help of quasi boundary triples; cf. \eqref{atheta} and
\eqref{athetaop}.
The proof is given after the next corollary.

\begin{theorem} \label{th.krein}
Let $A$ be a closed symmetric relation in $\cH$ and let $\{\cG,\Gamma_0,\Gamma_1\}$
be a quasi boundary triple for $A^*$ with $A_0=\ker\Gamma_0$, $\gamma$-field
$\gamma$
and Weyl function $M$.
Further, let $\Theta$ be a relation in $\cG$ and assume that
$\lambda\in\rho(A_0)$ is not an eigenvalue of $A_\Theta$, or, equivalently, that
$\ker(\Theta-M(\lambda))=\{0\}$. Then the following assertions are true:
\begin{itemize}
\item[(i)]
$g\in\ran(A_\Theta-\lambda)$ if and only if
$\gamma(\ov\lambda)^*g\in\dom (\Theta-M(\lambda))^{-1}$;
\item[(ii)]
for all $g\in\ran(A_\Theta-\lambda)$ we have
\begin{equation} \label{krein1}
  (A_\Theta-\lambda)^{-1}g = (A_0-\lambda)^{-1}g
  + \gamma(\lambda)\bigl(\Theta-M(\lambda)\bigr)^{-1}\gamma(\ov\lambda)^*g.
\end{equation}
\end{itemize}
\end{theorem}

If $\rho(A_\Theta)\cap\rho(A_0)\ne\varnothing$ or
$\rho(\ov{A_\Theta})\cap\rho(A_0)\ne\varnothing$,
e.g.\ if $A_\Theta$ is self-adjoint or essentially self-adjoint, respectively,
then
for $\lambda\in\rho(\ov{A_\Theta})\cap\rho(A_0)$, relation \eqref{krein1} is
valid on $\cH$
or a dense subset of $\cH$, respectively.  This, together with the fact that
$\gamma(\bar\lambda)^*$ is an everywhere defined bounded operator and
\[
  \gamma(\lambda)\bigl(\Theta-M(\lambda)\bigr)^{-1}\gamma(\bar\lambda)^*
  \subset
\overline{\gamma(\lambda)\bigl(\Theta-M(\lambda)\bigr)^{-1}}
\gamma(\bar\lambda)^*
\]
implies the following corollary.

\begin{corollary}\label{cor.krein}
Let the assumptions be as in Theorem~\ref{th.krein}. Then the following
assertions hold.
\begin{itemize}
\item[(i)]
If $\lambda\in\rho(A_\Theta)\cap\rho(A_0)$, then
\begin{equation} \label{krein2}
  (A_\Theta-\lambda)^{-1} = (A_0-\lambda)^{-1}
  + \gamma(\lambda)\bigl(\Theta-M(\lambda)\bigr)^{-1}\gamma(\ov\lambda)^*.
\end{equation}
\item[(ii)]
If $\lambda\in\rho(\ov{A_\Theta})\cap\rho(A_0)$, then
\begin{equation} \label{krein3}
  (\ov{A_\Theta}-\lambda)^{-1} = (A_0-\lambda)^{-1}
  + \ov{\gamma(\lambda)\bigl(\Theta-M(\lambda)\bigr)^{-1}}\gamma(\ov\lambda)^*.
\end{equation}
\end{itemize}
\end{corollary}

In particular, if $A_\Theta$ is self-adjoint or essentially self-adjoint, then
Krein's formula \eqref{krein2} or \eqref{krein3}, respectively, holds
at least for all non-real $\lambda$.

Let us now turn to the proof of Theorem~\ref{th.krein}.

\begin{proof}[Proof of Theorem \ref{th.krein}]
First note that by \cite[Theorem~2.8(i)]{BL07} the point $\lambda\in\rho(A_0)$
is not an eigenvalue of $A_\Theta$ if and only if
$\ker(\Theta-M(\lambda))=\{0\}$.
Fix some point $\lambda\in\rho(A_0)$ which is not an eigenvalue of $A_\Theta$.
Then the inverses $(A_\Theta-\lambda)^{-1}$ and $(\Theta-M(\lambda))^{-1}$ are
operators in $\cH$ and $\cG$, respectively.

Let $g\in\ran(A_\Theta-\lambda)$.
We show that $\gamma(\ov\lambda)^*g\in\dom (\Theta-M(\lambda))^{-1}$
and that formula \eqref{krein1} holds.
Set
\[
  f \defeq (A_\Theta-\lambda)^{-1}g-(A_0-\lambda)^{-1}g \quad\text{and}\quad
  \wh h \defeq
\vect{(A_\Theta-\lambda)^{-1}g}{g+\lambda(A_\Theta-\lambda)^{-1}g}.
\]
Then we have $f\in\cN_\lambda(T)=\ker(T-\lambda)$ and $\wh h\in A_\Theta$.
Moreover,
\begin{equation*}
  \Gamma_0\binom{f}{\lambda f}
  = \Gamma_0\wh h
  -\Gamma_0
  \underbrace{\vect{(A_0-\lambda)^{-1}g}{g+\lambda(A_0-\lambda)^{-1}g}}_{\in
A_0}
  = \Gamma_0\wh h
\end{equation*}
and
\begin{equation*}
  \Gamma_1\binom{f}{\lambda f}
  = \Gamma_1\wh h
  -\Gamma_1\vect{(A_0-\lambda)^{-1}g}{g+\lambda(A_0-\lambda)^{-1}g}
  = \Gamma_1\wh h
  -\gamma(\ov\lambda)^*g
\end{equation*}
by Proposition~\ref{gammaprop}(ii). These equalities together with
Proposition~\ref{gammaprop}(iv) yield
\begin{equation*}
  \gamma(\ov\lambda)^*g
  = \Gamma_1\wh h - \Gamma_1\binom{f}{\lambda f}
  = \Gamma_1\wh h - M(\lambda)\Gamma_0\binom{f}{\lambda f}
  = \Gamma_1\wh h -M(\lambda)\Gamma_0\wh h.
\end{equation*}
Since $\wh h\in A_\Theta$, we have $\binom{\Gamma_0\wh h}{\Gamma_1\wh
h}\in\Theta$
by \eqref{atheta} and hence
\begin{equation}\label{asdf}
  \vect{\Gamma_0\wh h}{\gamma(\ov\lambda)^*g}
  = \vect{\Gamma_0\wh h}{\Gamma_1\wh h - M(\lambda)\Gamma_0\wh h} \in
\Theta-M(\lambda),
\end{equation}
which implies $\gamma(\ov\lambda)^*g\in\dom (\Theta-M(\lambda))^{-1}$, that is,
one implication in (i) is proved.  Furthermore, it follows from \eqref{asdf}
that
$\Gamma_0\wh h = (\Theta-M(\lambda))^{-1}\gamma(\ov\lambda)^*g$
since $(\Theta-M(\lambda))^{-1}$ is an operator.  Therefore
\begin{align*}
  \gamma(\lambda)\bigl(\Theta-M(\lambda)\bigr)^{-1}\gamma(\ov\lambda)^*g
  &= \gamma(\lambda)\Gamma_0\wh h
  = \gamma(\lambda)\Gamma_0\binom{f}{\lambda f} \\[1ex]
  &= f = (A_\Theta-\lambda)^{-1}g-(A_0-\lambda)^{-1}g,
\end{align*}
which shows the relation \eqref{krein1}.
The converse implication in (i) was shown in the proof of
\cite[Theorem~2.8(ii)]{BL07}.
\end{proof}

With the help of Krein's formula and the next lemma we will obtain a
sufficient condition for self-adjointness of extensions $A_\Theta$ in
Theorem~\ref{th.A_Th_self_adj} below. Recall that $\sS_\infty(\cG)$
denotes the two-sided ideal of compact operators in $\cB(\cG)$.

\begin{lemma} \label{pr.inv_bdd}
Let $\{\cG,\Gamma_0,\Gamma_1\}$ be a quasi boundary triple with associated Weyl
function $M$.
Assume that $\ov{M(\lambda_0)}\in\sS_\infty(\cG)$ for some
$\lambda_0\in\CC\backslash\RR$ and
let $\Theta$ be a self-adjoint relation in $\cG$ such that
$0\notin\sess(\Theta)$.
Then
\[
  \bigl(\Theta-\ov{M(\lambda)}\bigr)^{-1}\in\cB(\cG)
\]
for all $\lambda\in\CC\backslash\RR$.
\end{lemma}

\begin{proof}
According to Proposition~\ref{pr.suff_cond}(iii) the operator $\ov{M(\lambda)}$
is compact
for all $\lambda\in\CC\backslash\RR$ because $\frS_\infty(\cG)$ is a
symmetrically normed ideal.
Without loss of generality let in the following $\lambda\in\CC^+$.
We can decompose the self-adjoint relation $\Theta$ into its self-adjoint
operator part
and the purely multi-valued part:
$\Theta=\Theta_{\rm op}\oplus\Theta_\infty$ with a corresponding decomposition
of the space $\cG=\cG_{\rm op}\oplus\cG_\infty$; cf. Section~\ref{2.1}.
Denote by $P_{\rm op}$ the orthogonal projection in $\cG$ onto $\cG_{\rm op}$.
Since $0\notin\sess(\Theta_{\rm op})$ and $\ov{M(\lambda)}$ is compact, the
operator $\Theta_{\rm op}-P_{\rm op}\ov{M(\lambda)}\vert_{\cG_{\rm op}}$
is a Fredholm operator in $\cG_{\rm op}$ with index $0$.
For $x\in\dom\Theta_{\rm op}$, $x\ne0$, we have
\begin{align*}
  \Im\bigl((\Theta_{\rm op}-P_{\rm op}\ov{M(\lambda)}\vert_{\cG_{\rm op}})x,
  x\bigr)_{\cG_{\rm op}}
  &= -\Im(\ov{M(\lambda)}x,x) \\
  &= -\bigl((\Im\ov{M(\lambda)})x,x\bigr) < 0,
\end{align*}
by Lemma~\ref{pr.Mpos}; hence
$\Theta_{\rm op}-P_{\rm op}\ov{M(\lambda)}\vert_{\cG_{\rm op}}$ has a trivial
kernel.
Since its index is zero, it is also surjective.  Because of the closedness,
its inverse is a bounded and everywhere defined operator in $\cG_{\rm op}$.
By \cite[p.~137]{LT77} we have
\[
  \bigl(\Theta-\ov{M(\lambda)}\bigr)^{-1}
  = \bigl(\Theta_{\rm op}-P_{\rm op}\ov{M(\lambda)}\vert_{\cG_{\rm
op}}\bigr)^{-1}
  P_{\rm op}
\]
and hence $(\Theta-\ov{M(\lambda)})^{-1}\in\cB(\cG)$.
\end{proof}

In the assumptions of the next theorem we make use of the notation
\[
  \Theta^{-1}(X) \defeq \left\{x\in\cG\colon \exists\, y\in X
  \text{ so that }\binom{x}{y}\in\Theta\right\}
\]
for a linear relation $\Theta$ in $\cG$ and a subspace $X\subset\cG$.
This theorem gives a sufficient condition for $A_\Theta$ being self-adjoint.

\begin{theorem} \label{th.A_Th_self_adj}
Let $A$ be a closed symmetric relation in $\cH$ and let
$\{\cG,\Gamma_0,\Gamma_1\}$
be a quasi boundary triple for $A^*$
with $A_i=\ker\Gamma_i$, $i=0,1$, and Weyl function $M$.
Assume that $A_1$ is self-adjoint and that $\ov{M(\lambda_0)}\in\sS_\infty(\cG)$
for some
$\lambda_0\in\dC\backslash\dR$.
If $\Theta$ is a self-adjoint relation in $\cG$ such that
\begin{equation}\label{ass12}
0\notin\sess(\Theta) \qquad\text{and}\qquad
\Theta^{-1}\bigl(\ran\ov{M(\lambda_\pm)}\bigr) \subset \cG_0
\end{equation}
hold for some $\lambda_+\in\dC^+$ and some $\lambda_-\in\dC^-$,
then $A_\Theta=\{\hat f\in T\colon \Gamma\hat f\in \Theta\}$ is self-adjoint in
$\cH$.
In particular, the second condition in \eqref{ass12} is satisfied if
$\dom\Theta\subset\cG_0$.
\end{theorem}

\begin{proof}
Note first that $\Theta=\Theta^*$ implies that $A_\Theta$ is a symmetric
relation
in $\cH$ and hence the eigenvalues of $A_\Theta$ are real.
Therefore it remains to check that $\ran(A_\Theta-\lambda_\pm)=\cH$ holds
for some (and hence for all) points $\lambda_\pm\in\dC^\pm$.
Since $\ran\gamma(\ov\lambda_\pm)^*\subset\cG_1$ by
Proposition~\ref{gammaprop}(ii),
we find from Theorem~\ref{th.krein}(i) that it is sufficient to verify the
inclusion
\[
  \cG_1 \subset
\dom\bigl(\Theta-M(\lambda_\pm)\bigr)^{-1}
=\ran\bigl(\Theta-M(\lambda_\pm)\bigr).
\]
Let $y\in\cG_1$ and let $\lambda_+\in\dC^+$ be such that the
second relation in \eqref{ass12} holds.
For $\lambda_-\in\dC^-$ the same reasoning applies.
With $x\defeq(\Theta-\ov{M(\lambda_+)})^{-1}y$, which is well defined by
Lemma~\ref{pr.inv_bdd}, we have
\[
  \begin{pmatrix} x \\ y+\ov{M(\lambda_+)}x\end{pmatrix} \in \Theta.
\]
Since $A_1$ is self-adjoint, we have $\ran M(\lambda_+) = \cG_1$
by Proposition~\ref{gammaprop}(iii) and hence
\[
  y + \ov{M(\lambda_+)}x \in \cG_1 + \ran\ov{M(\lambda_+)} =
\ran\ov{M(\lambda_+)}.
\]
It follows from the second assumption in \eqref{ass12} that $x\in\cG_0 = \dom
M(\lambda_+)$.  Therefore
$\bigl(\begin{smallmatrix}x\\y\end{smallmatrix}\bigr)\in\Theta-M(\lambda_+)$,
which shows that $y\in\ran(\Theta-M(\lambda_+))$.
\end{proof}

\begin{remark}\label{remremrem}
If $\Theta$ is a self-adjoint relation with $0\notin\sess(\Theta)$, then its
kernel is finite-dimensional.  If $\ker\Theta=\{0\}$, then $B\defeq\Theta^{-1}$
is a bounded, self-adjoint operator in $\cG$.
In this case, the second condition in \eqref{ass12} becomes
\begin{equation*}
  B\bigl(\ran\ov{M(\lambda_\pm)}\bigr) \subset \cG_0
\end{equation*}
and the relation $A_\Theta$ can be written as $A_\Theta
=\ker(B\Gamma_1-\Gamma_0)$.
If $\ker\Theta\ne\{0\}$, then one can write the abstract boundary condition
$\Gamma \hat f\in\Theta$, $\hat f\in T\subset A^*$, with the finite rank
projection $P$ onto
$\ker\Theta$ and the bounded operator
\[
  B =\bigl(\Theta\cap\bigl((\ker\Theta)^\perp\times
(\ker\Theta)^\perp\bigr)\bigr)^{-1}
  \in\cB\bigl((\ker\Theta)^\perp\bigr)
\]
in the form
\[
  P\Gamma_1\hat f = 0 \quad\text{and}\quad (1-P)\Gamma_0\hat f
=B(1-P)\Gamma_1\hat f, \quad
  \hat f\in\dom\Gamma=T.
\]
\end{remark}

\subsection{Resolvent differences in operator ideals}\label{3.3}

Let $A$ be a closed symmetric relation in a separable Hilbert space $\cH$,
let $\{\cG,\Gamma_0,\Gamma_1\}$ be a quasi boundary triple for $A^*$,
and let $\frA$ be a class of operator ideals.
With the help of Krein's formula we find sufficient conditions on the parameter
$\Theta$, the $\gamma$-field $\gamma$ and the Weyl function $M$ such that the difference
of the resolvents of the self-adjoint relations $A_\Theta$ and $A_0$ belongs to
some appropriate ideal, e.g.\ $\frA(\cH)$ or $(\frA\cdot\frA^*)(\cH)$.
These abstract results will turn out to be particularly useful in
Section~\ref{sec4} when we investigate Schatten--von~Neumann type properties of
resolvent differences of self-adjoint elliptic differential operators.

The first theorem of this subsection is one of the main results of the paper. Here
we consider the resolvent difference of $A_\Theta$ and $A_0$ under some
assumptions on $M(\lambda)$, $\gamma(\lambda)^*$ and $\Theta$.

\begin{theorem} \label{th.resolv_diff1}
Let $A$ be a closed symmetric relation in $\cH$ and let
$\{\cG,\Gamma_0,\Gamma_1\}$ be a quasi boundary triple for $A^*$
with $A_0=\ker\Gamma_0$, $\gamma$-field $\gamma$ and Weyl function $M$. Let
$\frA$ be a class of operator ideals and let $\Theta$ be a self-adjoint relation
in $\cG$ such that the following conditions hold:
\begin{itemize}
\item[(i)]
  $\ov{M(\lambda_0)}\in\sS_\infty(\cG)$ for some
  $\lambda_0\in\dC\backslash\dR$;
\item[(ii)]
  $\gamma(\lambda_1)^*\in\frA^*(\cH,\cG)$ for some $\lambda_1\in\rho(A_0)$;
\item[(iii)]
  $0\notin\sess(\Theta)$ and $A_\Theta=A_\Theta^*$.
\end{itemize}
Then
\begin{equation}\label{resdiffaa}
  (A_\Theta-\lambda)^{-1} - (A_0-\lambda)^{-1} \in
  (\frA\cdot\frA^*)(\cH)
\end{equation}
for all $\lambda\in\rho(A_\Theta)\cap\rho(A_0)$.
\end{theorem}

\begin{proof}
Note that the assumptions (i) and (ii) together with
Proposition~\ref{pr.suff_cond} imply
that $\ov{M(\lambda)}\in\sS_\infty(\cG)$, $\gamma(\lambda)^*\in\frA^*(\cH,\cG)$
and
$\gamma(\lambda)^{**}=\overline{\gamma(\lambda)}\in\frA(\cG,\cH)$ for all
$\lambda\in\rho(A_0)$.
Corollary~\ref{cor.krein}(i) yields
that the resolvent difference of the self-adjoint relations $A_\Theta$ and $A_0$
has the form
\begin{equation} \label{clos.resdiff}
 \begin{split}
  (A_\Theta-\lambda)^{-1} - (A_0-\lambda)^{-1}
  &= \gamma(\lambda)\bigl(\Theta-M(\lambda)\bigr)^{-1}\gamma(\ov\lambda)^* \\
  &= \ov{\gamma(\lambda)}\bigl(\Theta-\ov{M(\lambda)}\bigr)^{-1}\gamma(\ov\lambda)^*
 \end{split}
\end{equation}
for all $\lambda\in\rho(A_\Theta)\cap\rho(A_0)$.
Furthermore, since the operator $\ov{M(\lambda_0)}$ is compact, we have
$(\Theta-\ov{M(\lambda)})^{-1}\in\cB(\cG)$ for all $\lambda\in\dC\backslash\dR$
by
Lemma~\ref{pr.inv_bdd}. Therefore, if $\lambda\in\dC\backslash\dR$, then
\begin{equation*}
\bigl(\Theta-\ov{M(\lambda)}\bigr)^{-1}\gamma(\ov\lambda)^*\in\frA^*(\cH,
\cG)\quad\text{and}
  \quad \ov{\gamma(\lambda)}\in\frA(\cG,\cH),
\end{equation*}
and hence \eqref{resdiffaa} follows.
Lemma~\ref{resdifflemma} implies that \eqref{resdiffaa} holds also for
all $\lambda$ in the (possibly larger) set $\rho(A_\Theta)\cap\rho(A_0)$.
\end{proof}

Note that Theorem~\ref{th.A_Th_self_adj} provides a sufficient condition for
the second assumption in (iii) of Theorem~\ref{th.resolv_diff1}.

\begin{remark}
As a corollary one immediately obtains the same result for the resolvent difference
\[
  (A_{\Theta_1}-\lambda)^{-1}-(A_{\Theta_2}-\lambda)^{-1}
\]
of $A_{\Theta_1}$, $A_{\Theta_2}$, where $\Theta_1$ and $\Theta_2$ both satisfy
the assumptions in Theorem~\ref{th.resolv_diff1}.
In Theorem~\ref{th.resolv_diff4} we improve this under the additional assumption
that $\Theta_1-\Theta_2$ is in some class of operator ideals.
\end{remark}

\begin{remark} \label{rem.resdiff_Sp}
If $\frA$ is equal to $\frS_p$, $\frS_{r,\infty}$ or
$\frS_{r,\infty}^{(0)}$, then the resolvent difference in
\eqref{resdiffaa} is in $\frS_{p/2}(\cH)$, $\frS_{2r,\infty}(\cH)$ or
$\frS_{2r,\infty}^{(0)}(\cH)$, respectively.  This follows from Lemma~\ref{splemma}.
\end{remark}

Krein's formula can be used to prove a trace formula if the resolvent difference
is a trace class operator.

\begin{corollary} \label{co.trace_formula}
Let $A$ be a closed symmetric relation in a separable Hilbert space $\cH$ and
let $\{\cG,\Gamma_0,\Gamma_1\}$ be a quasi boundary triple for $A^*$ with
$A_0=\ker\Gamma_0$, $\gamma$-field $\gamma$ and Weyl function $M$.
Further, let $\Theta$ be a self-adjoint relation in $\cG$ such that the following conditions hold:
\begin{itemize}
\item[(i)] $\ov{M(\lambda_0)}\in\frS_\infty(\cG)$ for some
$\lambda_0\in\dC\backslash\dR$;
\item[(ii)] $\gamma(\lambda_1)^*\in\frS_2(\cH,\cG)$ for some $\lambda_1\in\rho(A_0)$;
\item[(iii)] $0\notin\sess(\Theta)$ and $A_\Theta=A_\Theta^*$.
\end{itemize}
Then
\[
  (A_\Theta-\lambda)^{-1}-(A_0-\lambda)^{-1} \in \frS_1(\cH)
\]
and
\[
  \tr\bigl((A_\Theta-\lambda)^{-1}-(A_0-\lambda)^{-1}\bigr)
  = \tr\Bigl(\ov{M'(\lambda)}\bigl(\Theta-\ov{M(\lambda)}\bigr)^{-1}\Bigr)
\]
for $\lambda\in\rho(A_\Theta)\cap\rho(A_0)$,
where $\ov{M'(\lambda)}$ is defined as in Remark~\ref{re.deriv_M}.
\end{corollary}

\begin{proof}
The first assertion is clear from Theorem~\ref{th.resolv_diff1}
and Remark~\ref{rem.resdiff_Sp}.  Hence we can apply the trace to both sides
of \eqref{clos.resdiff}.
Using \eqref{deriv_closM} and the relation $\tr(AB)=\tr(BA)$ (see, e.g.\ \cite[Theorem~III.8.2]{GK69})
we obtain
\begin{align}
  &\tr\bigl((A_\Theta-\lambda)^{-1}-(A_0-\lambda)^{-1}\bigr)
  = \tr\Bigl(\ov{\gamma(\lambda)}\bigl(\Theta-\ov{M(\lambda)}\bigr)^{-1}
    \gamma(\ov\lambda)^*\Bigr)
  \notag \\[1ex]
  &= \tr\Bigl(\gamma(\ov\lambda)^*\ov{\gamma(\lambda)}\bigl(\Theta-\ov{M(\lambda)}\bigr)^{-1}\Bigr)
  \label{83827} \\[1ex]
  &= \tr\Bigl(\ov{M'(\lambda)}\bigl(\Theta-\ov{M(\lambda)}\bigr)^{-1}\Bigr); \notag
\end{align}
note that also the operator $\ov{M'(\lambda)}(\Theta-\ov{M(\lambda)})^{-1}$ in \eqref{83827} is a trace class operator.
\end{proof}

In the following theorem the assumptions $\ov{M(\lambda_0)}\in\frS_\infty(\cG)$,
$0\notin\sess(\Theta)$ are replaced by a weaker assumption on
$\Theta-M(\lambda)$;
the conclusion is also weaker than the one in Theorem~\ref{th.resolv_diff1}.

\begin{theorem} \label{th.resolv_diff2}
Let $A$ be a closed symmetric relation in $\cH$ and let
$\{\cG,\Gamma_0,\Gamma_1\}$ be a quasi boundary triple for $A^*$
with $A_0=\ker\Gamma_0$, $\gamma$-field $\gamma$ and Weyl function $M$. Let
$\frA$ be a class of operator ideals and let $\Theta$ be a symmetric relation
in $\cG$ such that the following conditions hold:
\begin{itemize}
\item[(i)] $\ov{\Theta-M(\lambda_0)}$ is injective for some
$\lambda_0\in\dC\backslash\dR$;
\item[(ii)] $\gamma(\lambda_1)^*\in\frA^*(\cH,\cG)$ for some $\lambda_1\in\rho(A_0)$
\item[(iii)] $A_\Theta=A_\Theta^*$.
\end{itemize}
Then
\begin{equation}\label{resdiffaaa}
  (A_\Theta-\lambda)^{-1} - (A_0-\lambda)^{-1} \in \frA(\cH)
\end{equation}
for all $\lambda\in\rho(A_\Theta)\cap\rho(A_0)$.
\end{theorem}

\begin{proof}
According to Corollary~\ref{cor.krein}(i) we can write the resolvent difference
at
the point $\lambda_0\in\rho(A_\Theta)\cap\rho(A_0)$ as
\begin{equation*}
  (A_\Theta-\lambda_0)^{-1}-(A_0-\lambda_0)^{-1}
  =
\overline{\gamma(\lambda_0)}\bigl(\Theta-M(\lambda_0)\bigr)^{-1}
\gamma(\ov\lambda_0)^*.
\end{equation*}
In particular, it follows that both products on the right-hand side are
well defined, and hence
\begin{equation}\label{e2}
\bigl(\Theta-M(\lambda_0)\bigr)^{-1}\gamma(\ov\lambda_0)^*
\end{equation}
is everywhere defined. Since the relation $\ov{\Theta-M(\lambda_0)}$ is
injective, it follows
that $(\Theta-M(\lambda_0))^{-1}$ is a closable operator.
Therefore, because $\gamma(\ov\lambda_0)^*$ is a bounded operator, the product
in \eqref{e2} is a closable, everywhere defined operator and hence in
$\cB(\cH,\cG)$.
Moreover, since $\gamma(\lambda_1)^*\in\frA^*(\cH,\cG)$, it follows from
Proposition~\ref{pr.suff_cond} that $\ov{\gamma(\lambda)}$ belongs to
$\frA(\cG,\cH)$ for all $\lambda\in\rho(A_0)$.
Hence the difference of the resolvents in \eqref{resdiffaaa} is in
$\frA(\cH)$ for $\lambda=\lambda_0$.
Then Lemma~\ref{resdifflemma} implies
\eqref{resdiffaaa} for all $\lambda\in\rho(A_\Theta)\cap\rho(A_0)$.
\end{proof}

In the case $\Theta=0$ the above theorem together with Lemma~\ref{pr.Mpos} imply
the next
corollary.

\begin{corollary} \label{co.resolv_diff10}
Let $A$, $\{\cG,\Gamma_0,\Gamma_1\}$, $\gamma$, $M$ and $\frA$ be as in
Theorem~\ref{th.resolv_diff2}.
Assume that $A_1=\ker\Gamma_1$ is self-adjoint, that $M(\lambda_0)$ is bounded
for some $\lambda_0\in\rho(A_0)$ and that $\gamma(\lambda_1)^*\in\frA^*(\cH,\cG)$
for some $\lambda_1\in\rho(A_0)$. Then
\[
  (A_1-\lambda)^{-1} - (A_0-\lambda)^{-1} \in \frA(\cH)
\]
for all $\lambda\in\rho(A_1)\cap\rho(A_0)$
\end{corollary}

Corollary~\ref{co.resolv_diff10} can be generalized as follows.

\begin{theorem} \label{th.resolv_diff3}
Let $A$ be a closed symmetric relation in $\cH$ and let
$\{\cG,\Gamma_0,\Gamma_1\}$ be a quasi boundary triple for $A^*$
with $A_0=\ker\Gamma_0$, $\gamma$-field $\gamma$ and Weyl function $M$.
Furthermore, let $\frA$, $\frB$ be classes of operator ideals and assume that
the following conditions hold:
\begin{itemize}
\item[(i)]
  $A_1=\ker\Gamma_1$ is self-adjoint;
\item[(ii)]
  $M(\lambda_0)$ is bounded for some $\lambda_0\in\rho(A_0)$;
\item[(iii)]
  $\gamma(\lambda_1)^*\in\frA^*(\cH,\cG)$ for some $\lambda_1\in\rho(A_0)$;
\item[(iv)]
  there exists a Hilbert space $\wt\cG_0$ such that
  $\cG_0\subset\wt\cG_0\subset\cG$ and the embedding $\iota_{\wt\cG_0\to\cG}$
  belongs to $\frB(\widetilde\cG_0,\cG)$.
\end{itemize}
Then
\begin{equation}\label{resdiffb}
  (A_1-\lambda)^{-1} - (A_0-\lambda)^{-1} \in
  (\frA\cdot\frB)(\cH)
\end{equation}
for all $\lambda\in\rho(A_1)\cap\rho(A_0)$.
\end{theorem}

\begin{proof}
Since $M(\lambda_0)$ is bounded, Lemma~\ref{pr.Mpos} implies that
$\ker\ov{M(\lambda)} = \{0\}$ for every $\lambda\in\CC\backslash\RR$ and
hence $M(\lambda)^{-1}\gamma(\ov\lambda)^*$ is
closable from $\cH$ into $\cG$ with values in $\cG_0=\dom M(\lambda)$. The
boundedness of the embedding
$\iota_{\wt\cG_0\to\cG}$ implies that $M(\lambda)^{-1}\gamma(\ov\lambda)^*$
regarded as an operator
from $\cH$ into $\wt\cG_0$ is
also closable. Furthermore, this operator is everywhere defined and hence we
have
\[
  M(\lambda)^{-1}\gamma(\ov\lambda)^*\in\cB(\cH,\widetilde\cG_0)\quad\text{and}
  \quad
  \iota_{\wt\cG_0\to\cG}M(\lambda)^{-1}
  \gamma(\ov\lambda)^*\in\frB(\cH,\cG)
\]
for all $\lambda\in\dC\backslash\dR$ by assumption (iv).
Assumption (iii) implies
$\ov{\gamma(\lambda)}\in\frA(\cG,\cH)$ for all $\lambda\in\rho(A_0)$; cf.\
Proposition~\ref{pr.suff_cond}(i).
By the self-adjointness of $A_1$ and by Corollary~\ref{cor.krein} we have
\[
  (A_1-\lambda)^{-1} - (A_0-\lambda)^{-1}
  = -\ov{\gamma(\lambda)}
\iota_{\wt\cG_0\to\cG}M(\lambda)^{-1}\gamma(\ov\lambda)^*,
\]
which is in $(\frA\cdot\frB)(\cH)$ for all
$\lambda\in\dC\backslash\dR$.
An application of Lemma~\ref{resdifflemma} shows that \eqref{resdiffb} holds
also for all
$\lambda\in\rho(A_1)\cap\rho(A_0)$.
\end{proof}

In the next theorem the difference of the resolvents of two self-adjoint
extensions $A_{\Theta_1}$ and $A_{\Theta_2}$
is considered under additional assumptions on $\Theta_1-\Theta_2$; cf.\
\cite[Theorem~2 and Corollary~4]{DM91} for
the case that $\{\cG,\Gamma_0,\Gamma_1\}$ is an ordinary boundary triple.

\begin{theorem} \label{th.resolv_diff4}
Let $A$ be a closed symmetric relation in $\cH$ and let
$\{\cG,\Gamma_0,\Gamma_1\}$ be a quasi boundary triple for $A^*$
with $A_0=\ker\Gamma_0$, $\gamma$-field $\gamma$ and Weyl function $M$. Let
$\frB$ be a class of operator ideals,
let $\Theta_1$ and $\Theta_2$ be two self-adjoint bounded operators in $\cG$
and assume that the following conditions hold:
\begin{itemize}
  \item[(i)] $\overline{M(\lambda_0)}\in\sS_\infty(\cG)$ for some
    $\lambda_0\in\rho(A_0)$;
  \item[(ii)] $0\notin\sess(\Theta_i)$ and $A_{\Theta_i}=A_{\Theta_i}^*$ for $i=1,2$;
  \item[(iii)] $\Theta_1-\Theta_2\in\frB(\cG)$.
\end{itemize}
Then
\[
  (A_{\Theta_1}-\lambda)^{-1} - (A_{\Theta_2}-\lambda)^{-1} \in \frB(\cH)
\]
holds for all $\lambda\in\rho(A_{\Theta_1})\cap\rho(A_{\Theta_2})$.
If, in addition, $\frA$ is another class of operator ideals and
$\gamma(\lambda_1)^*\in\frA^*(\cH,\cG)$ for some $\lambda_1\in\rho(A_0)$, then
\[
  (A_{\Theta_1}-\lambda)^{-1} - (A_{\Theta_2}-\lambda)^{-1} \in
  (\frA\cdot\frB\cdot\frA^*)(\cH)
\]
for all $\lambda\in\rho(A_{\Theta_1})\cap\rho(A_{\Theta_2})$.
\end{theorem}

\begin{proof}
Because $\Theta_i$ and $(\Theta_i-\ov{M(\lambda)})^{-1}$ are bounded for all
$\lambda\in\dC\backslash\dR$ by
Lemma~\ref{pr.inv_bdd}, the difference of \eqref{krein2} for $\Theta=\Theta_1$
and $\Theta=\Theta_2$ can be rewritten as follows
\begin{equation*}
 \begin{split}
  &(A_{\Theta_1}-\lambda)^{-1} - (A_{\Theta_2}-\lambda)^{-1}\\
  &\qquad=
\ov{\gamma(\lambda)}\bigl(\Theta_1-\ov{M(\lambda)}\bigr)^{-1}(\Theta_2-\Theta_1)
  \bigl(\Theta_2-\ov{M(\lambda)}\bigr)^{-1}\gamma(\ov\lambda)^*.
 \end{split}
\end{equation*}
All five factors on the right-hand side are bounded, the middle factor is in
$\frB(\cG)$;
hence the product is in $\frB(\cH)$ for $\lambda\in\dC\backslash\dR$,
and Lemma~\ref{resdifflemma} implies
that this is true for all $\lambda\in\rho(A_{\Theta_1})\cap\rho(A_{\Theta_2})$.
If, in addition, $\gamma(\lambda_1)^*\in\frA^*(\cH,\cG)$, then
$\gamma(\ov\lambda)^*\in\frA^*(\cH,\cG)$
and
$\ov{\gamma(\lambda)}\in\frA(\cG,\cH)$ for all $\lambda\in\rho(A_0)$ by
Proposition~\ref{pr.suff_cond}(ii) and
hence the second assertions holds for all
$\lambda\in\rho(A_{\Theta_1})\cap\rho(A_{\Theta_2})\cap\rho(A_0)$.
It remains to use the argument in the proof of Lemma~\ref{resdifflemma} to
conclude the assertion
for all $\lambda\in\rho(A_{\Theta_1})\cap\rho(A_{\Theta_2})$.
\end{proof}

\subsection{Dissipative and accumulative realizations}
\label{3.4}
In this section we indicate some generalizations of the results from the
previous section for the case that $A_\Theta$ is only maximal dissipative or maximal
accumulative instead of self-adjoint.

For this we first recall some necessary definitions and facts. A linear relation
$S$ in $\cH$ is said to be \emph{dissipative} if $\Im(f^\prime,f)\geq 0$
for all $(f,f^\prime)^\top\in S$ and \emph{accumulative} if
$\Im(f^\prime,f)\leq 0$ for all $(f,f^\prime)^\top\in S$.
The relation $S$ is said to be \emph{maximal dissipative} (\emph{maximal accumulative})
if $S$ is dissipative (accumulative, respectively) and has no proper dissipative
(accumulative, respectively) extensions in $\cH$.
A dissipative (accumulative) relation $S$ in $\cH$ is maximal dissipative
(maximal accumulative, respectively) if and only if $\ran(S-\lambda_-)=\cH$
($\ran(S-\lambda_+)=\cH$, respectively) for some (and hence for all)
$\lambda_-\in\dC^-$
($\lambda_+\in\dC^+$, respectively). Similarly as for self-adjoint relations, a
maximal dissipative (maximal accumulative) relation $S$ can be
written as the orthogonal sum $S_{\rm op}\oplus S_\infty$ of a maximal
dissipative (maximal accumulative, respectively)
operator $S_{\rm op}$ in the Hilbert space $\cH_{\rm op}=(\mul S)^\bot$ and the
purely
multi-valued relation $S_\infty$ in $\cH_\infty=\mul S$.

It follows easily from Green's identity \eqref{green1} that $A_\Theta$ is
dissipative (accumulative)
if $\Theta$ is dissipative (accumulative, respectively).  Krein's formula
\eqref{krein2}
is valid for $\lambda\in\dC^-$ if $A_\Theta$ is maximal dissipative and for
$\lambda\in\dC^+$
if $A_\Theta$ is maximal accumulative.  Moreover, it is easy to see that
Lemma~\ref{pr.inv_bdd}
remains valid for $\lambda\in\dC^-$ ($\lambda\in\dC^+$) if $\Theta$ is a
maximal dissipative (maximal accumulative, respectively) relation in $\cG$ such
that
$0\notin\sigma_{\rm ess}(\Theta)$.
This leads to the following variants of Theorems~\ref{th.A_Th_self_adj} and
\ref{th.resolv_diff1}.

\begin{theorem}
Let $A$ be a closed symmetric relation in $\cH$ and let
$\{\cG,\Gamma_0,\Gamma_1\}$
be a quasi boundary triple for $A^*$
with $A_i=\ker\Gamma_i$, $i=0,1$, and Weyl function $M$.
Assume that $A_1$ is self-adjoint and that $\ov{M(\lambda_0)}\in\sS_\infty(\cG)$
for some $\lambda_0\in\dC\backslash\dR$.
If $\Theta$ is a maximal dissipative (maximal accumulative) relation in $\cG$ such that
\begin{equation}
0\notin\sess(\Theta) \qquad\text{and}\qquad
\Theta^{-1}\bigl(\ran\ov{M(\lambda)}\bigr) \subset \cG_0
\end{equation}
hold for some $\lambda\in\dC^-$ ($\lambda\in\dC^+$, respectively),
then $A_\Theta=\{\hat f\in T\colon \Gamma\hat f\in \Theta\}$ is maximal
dissipative
(maximal accumulative, respectively) in $\cH$.
In particular, the second condition in \eqref{ass12} is satisfied if
$\dom\Theta\subset\cG_0$.
\end{theorem}

\begin{theorem}
\label{th.A_dis_acc}
Let $A$ be a closed symmetric relation in $\cH$ and let
$\{\cG,\Gamma_0,\Gamma_1\}$ be a quasi boundary triple for $A^*$
with $A_0=\ker\Gamma_0$, $\gamma$-field $\gamma$ and Weyl function $M$. Let
$\frA(\cG,\cH)$ be a class of operator ideals
and let $\Theta$ be a maximal dissipative (maximal accumulative) relation in
$\cG$ such that the following conditions hold:
\begin{itemize}
\item[(i)] $\ov{M(\lambda_0)}\in\sS_\infty(\cG)$ for some
$\lambda_0\in\CC\backslash\RR$;
\item[(ii)] $\gamma(\lambda_1)^*\in\frA^*(\cH,\cG)$ for some $\lambda_1\in\rho(A_0)$;
\item[(iii)] $0\notin\sess(\Theta)$
and $A_\Theta$ is maximal dissipative (maximal accumulative, respectively).
\end{itemize}
Then
\begin{equation*}
  (A_\Theta-\lambda)^{-1} - (A_0-\lambda)^{-1} \in
  (\frA\cdot\frA^*)(\cH)
\end{equation*}
for all $\lambda\in\rho(A_\Theta)\cap\rho(A_0)$.
\end{theorem}

Similarly, a dissipative and accumulative variant of
Theorem~\ref{th.resolv_diff2} is valid; the formulation of such
a theorem is left to the reader.

Next we state a generalization of Theorem~\ref{th.resolv_diff4}.  In the mixed
case that one of the
two operators $\Theta_i$ is dissipative and the other one is accumulative, one
has to assume
that one ideal is symmetrically normed, and the proof is more complicated, but
similar to the proof of Theorem~3.11 in \cite{BLLLP10}.

\begin{theorem}
\label{th.resolv_diff6}
Let $A$ be a closed symmetric relation in $\cH$ and let
$\{\cG,\Gamma_0,\Gamma_1\}$ be a
quasi boundary triple for $A^*$ with $A_0=\ker\Gamma_0$, $\gamma$-field $\gamma$
and Weyl function $M$.  Let $\frB$ be a class of operator ideals,
let $\Theta_1$ and $\Theta_2$ be dissipative or accumulative bounded operators
in $\cG$ and assume that the following hold:
\begin{itemize}
  \item[(i)]
    $\overline{M(\lambda_0)}\in\sS_\infty(\cG)$ for some $\lambda_0\in\rho(A_0)$;
  \item[(ii)]
    $0\notin\sess(\Theta_i)$ and $A_{\Theta_i}$, $i=1,2$, are either maximal dissipative or
    maximal accumulative;
  \item[(iii)]
    $\Theta_1-\Theta_2\in\frB(\cG)$;
  \item[(iv)]
    in the case that one of $A_{\Theta_i}$ is maximal dissipative and the other one
    is maximal accumulative, assume in addition that $\frB$ is a class of symmetrically
    normed ideals and that $A_{\Re\Theta_i}=A_{\Re\Theta_i}^*$ for $i=1,2$.
\end{itemize}
Then
\begin{equation}\label{resdiff777}
  (A_{\Theta_1}-\lambda)^{-1} - (A_{\Theta_2}-\lambda)^{-1} \in
  \frB(\cH)
\end{equation}
for all $\lambda\in\rho(A_{\Theta_1})\cap\rho(A_{\Theta_2})$.
If, in addition, $\frA$ is another class of operator ideals and
$\gamma(\lambda_1)^*\in\frA^*(\cH,\cG)$ for some $\lambda_1\in\rho(A_0)$, then
\begin{equation}\label{resdiff888}
  (A_{\Theta_1}-\lambda)^{-1} - (A_{\Theta_2}-\lambda)^{-1} \in
  (\frA\cdot\frB\cdot\frA^*)(\cH)
\end{equation}
for all $\lambda\in\rho(A_{\Theta_1})\cap\rho(A_{\Theta_2})$.
\end{theorem}

\begin{proof}
In the case that $A_{\Theta_1}$ and $A_{\Theta_2}$ are either both maximal
dissipative or
both maximal accumulative, the proof is exactly the same as for
Theorem~\ref{th.resolv_diff4}.
Let us treat the case that $A_{\Theta_1}$ is maximal dissipative and
$A_{\Theta_2}$ is maximal accumulative and that $\rho(A_{\Theta_1})\cap\rho(A_{\Theta_2})$
is non-empty.
Since $\Theta_1-\Theta_2\in\frB(\cG)$ we also have
\begin{equation*}
  \Re(\Theta_1-\Theta_2)\in\frB(\cG)\quad\text{and}\quad
  \Im(\Theta_1-\Theta_2)\in\frB(\cG).
\end{equation*}
Because $\Im\Theta_1\geq 0$ and $\Im\Theta_2\leq 0$, the following inequalities
are true:
\begin{equation*}
  0\leq \Im\Theta_1\leq\Im(\Theta_1-\Theta_2)\quad\text{and}\quad 0\leq
  -\Im\Theta_2\leq\Im(\Theta_1-\Theta_2).
\end{equation*}
Hence $s_k(\Im\Theta_1)\le s_k(\Im(\Theta_1-\Theta_2))$ and
$s_k(-\Im\Theta_2)\le s_k(\Im(\Theta_1-\Theta_2))$ for $k=1,2,\dots$, which
by \cite[III.\S2.2]{GK69} implies that also $\Im\Theta_1$ and $\Im\Theta_2$
belong to $\frB(\cG)$.
Therefore
\begin{equation*}
  \Theta_i-\Re\Theta_i\in\frB(\cG)\quad\text{and}\quad
  \sess(\Theta_i)=\sess(\Re\Theta_i),\qquad i=1,2.
\end{equation*}
The same reasoning as in the proof of Theorem~\ref{th.resolv_diff4} shows that
\begin{align}
  &(A_{\Theta_1}-\lambda)^{-1} - (A_{\Re\Theta_1}-\lambda)^{-1} \label{resd_AT_AReT} \\
  &\qquad=
  \ov{\gamma(\lambda)}\bigl(\Theta_1-\ov{M(\lambda)}\bigr)^{-1}
  (\Re\Theta_1-\Theta_1)
  \bigl(\Re\Theta_1-\ov{M(\lambda)}\bigr)^{-1}\gamma(\ov\lambda)^*
  \notag
\end{align}
for all $\lambda\in\dC_-$, and hence the difference of the resolvents of
$A_{\Theta_1}$ and $A_{\Re\Theta_1}$ belongs to $\frB(\cH)$, which in particular
implies that $\sess(A_{\Theta_1})=\sess(A_{\Re\Theta_1})$.
An application of Lemma~\ref{resdifflemma} yields that the resolvent difference in \eqref{resd_AT_AReT}
is in $\frB(\cH)$ for all $\lambda\in\rho(A_{\Theta_1})\cap\rho(A_{\Re\Theta_1})$.
Similar formulae hold for the resolvent differences
\begin{equation*}
  (A_{\Re\Theta_1}-\lambda)^{-1} - (A_{\Re\Theta_2}-\lambda)^{-1},\qquad
  \lambda\in\dC\backslash\dR,
\end{equation*}
and
\begin{equation*}
  (A_{\Re\Theta_2}-\lambda)^{-1} -
  (A_{\Theta_2}-\lambda)^{-1},\qquad\lambda\in\dC_+,
\end{equation*}
which also belong to $\frB(\cH)$ for $\lambda\in\rho(A_{\Re\Theta_1})\cap\rho(A_{\Re\Theta_2})$
and $\lambda\in\rho(A_{\Theta_2})\cap\rho(A_{\Re\Theta_2})$, respectively.
If $\lambda$ belongs to
\begin{equation}\label{inters4rho}
  \rho(A_{\Theta_1})\cap\rho(A_{\Theta_2})\cap\rho(A_{\Re\Theta_1})\cap\rho(A_{\Re\Theta_2}),
\end{equation}
then these three resolvent differences can be added, which yields \eqref{resdiff777} for all
such $\lambda$.  Since $\rho(A_{\Theta_1})\cap\rho(A_{\Theta_2})\ne\varnothing$ and
$\sess(A_{\Theta_i})=\sess(A_{\Re\Theta_i})$ the set in \eqref{inters4rho} is non-empty.
Another application of Lemma~\ref{resdifflemma} shows that \eqref{resdiff777} is true
for all $\lambda\in\rho(A_{\Theta_1})\cap\rho(A_{\Theta_2})$.

If, in addition, $\gamma(\lambda_1)^*\in\frA^*(\cH,\cG)$ for some, and hence for
all, $\lambda_1\in\rho(A_0)$, then $\overline{\gamma(\lambda_1)}\in\frA(\cG,\cH)$ and
therefore \eqref{resdiff888} holds for all
$\lambda\in\rho(A_{\Theta_1})\cap\rho(A_{\Theta_2})$.
\end{proof}

\section{Self-adjoint elliptic operators and spectral estimates for
resolvent differences}
\label{sec4}

In this section we study elliptic operators on domains in $\RR^n$ with smooth
compact boundary,
i.e.\ either on bounded domains or on exterior domains.  In the first subsection
we construct
a quasi boundary triple where functions in the domain of $T$ are in $H^2$ in a
neighbourhood
of the boundary and prove sufficient conditions for self-adjoint realizations.
We shall sometimes speak of an $H^2$ framework here although for exterior domains
$T$ is defined on a larger space, see Definition~\ref{domt}.
In Subsection~\ref{4.2} we apply the abstract results from Section~\ref{3.3} to
elliptic operators and prove estimates for singular values of resolvent
differences
of realizations with different boundary conditions.
In Section~\ref{4.3} self-adjoint elliptic operators on $\dR^n$ with $\delta$ and $\delta^\prime$-interactions
on smooth hypersurfaces are constructed with the help of quasi boundary triples and interface conditions
on the hypersurface. The abstract results from Section~\ref{3.3} imply
spectral estimates for the resolvent differences of the elliptic operators
with $\delta$ and $\delta^\prime$-interactions and the unperturbed
elliptic operator on $\dR^n$.

\subsection{Quasi boundary triples and Weyl functions for second order elliptic
differential expressions}\label{4.1}

Let $\Omega\subset\dR^n$ be a bounded or unbounded domain with compact
$C^\infty$-boundary $\partial\Omega$.
We consider a formally symmetric second order differential
expression
\begin{equation}\label{cl}
  (\cL f)(x) \defeq -\sum_{j,k=1}^n \left(
  \frac{\partial}{\partial x_j} \Bigl(a_{jk} \frac{\partial f}{\partial
  x_k}\Bigr)\right)(x)+ a(x)f(x),
  \quad x\in\Omega,
\end{equation}
with bounded, infinitely differentiable, real-valued coefficients
$a_{jk}\in C^\infty(\overline\Omega)$ satisfying $a_{jk}(x)=a_{kj}(x)$ for
all $x\in\overline\Omega$ and $j,k=1,\dots,n$ and a real-valued function
$a\in L^\infty(\Omega)$. Furthermore, $\cL$ is assumed to be uniformly elliptic, i.e.\
the condition
\begin{equation*}
  \sum_{j,k=1}^n a_{jk}(x)\xi_j\xi_k\geq C\sum_{k=1}^n\xi_k^2
\end{equation*}
holds for some $C>0$, all $\xi=(\xi_1,\dots,\xi_n)^\top\in\dR^n$ and $x\in\overline
\Omega$.
We note that the assumptions on the
domain $\Omega$ and the coefficients of $\cL$ can be relaxed but it is not our
aim to treat the most general setting here. We refer the reader to, e.g.\ \cite{EE87,G85,GBook08,LM72,M,W} for
possible generalizations and to \cite{AGW10,AGMT10,GM08-1,GM08-2,GM09} for recent work on
non-smooth domains.  On the other hand, we do not impose any conditions on the
growth of derivatives of $a_{jk}$ at infinity; cf.\ the stronger assumptions
in \cite[Condition 3.1]{M10}.

In the following we denote by $H^s(\Omega)$ and $H^s(\partial\Omega)$, $s\geq0$,
the usual Sobolev spaces of order $s$ of $L^2$-functions on $\Omega$ and $\partial\Omega$,
respectively. The Sobolev space $H^{-s}(\partial\Omega)$, $s>0$, of negative order is defined
as the dual space of $H^s(\partial\Omega)$; see, e.g.\
\cite[Section~7.3]{LM72} and \cite{AF03}.
The closure of $C_0^\infty(\Omega)$ in $H^s(\Omega)$ is denoted by
$H_0^s(\Omega)$.
For a function $f\in C^\infty(\overline\Omega)$ we denote the trace by
$f\vert_{\partial\Omega}$ and we set
\begin{equation*}
  \frac{\partial f}{\partial\nu_\cL}\Bigl|_{\partial\Omega} \defeq
  \sum_{j,k=1}^n a_{jk} n_j \frac{\partial f}{\partial x_k}
  \Bigl|_{\partial\Omega},
\end{equation*}
where $n(x)=(n_1(x),\dots, n_n(x))^\top$ is the unit vector at the point
$x\in\partial\Omega$ pointing out of $\Omega$.
Recall that for all $s>\frac{3}{2}$ the mapping $C^\infty(\overline\Omega)\ni
f\mapsto\bigl\{f|_{\partial\Omega},
\tfrac{\partial f}{\partial\nu_\cL}\bigl|_{\partial\Omega}\bigr\}$
extends by continuity to a continuous surjective mapping
\begin{equation}\label{tracemap}
  H^s(\Omega)\ni f\mapsto \left\{f|_{\partial\Omega},
  \frac{\partial f}{\partial\nu_\cL}\Bigl|_{\partial\Omega}\right\}
  \in H^{s-1/2}(\partial\Omega)\times H^{s-3/2}(\partial\Omega),
\end{equation}
which admits a bounded right inverse.
For $s=2$ the kernel of the mapping in \eqref{tracemap} is equal to
$H_0^2(\Omega)$.

In order to construct a quasi boundary triple for the maximal operator associated to $\cL$ in $L^2(\Omega)$
in an ``$H^2$ setting'' we fix a suitable operator $T$ as the domain of the boundary mappings.

\begin{definition}\label{domt}
The differential operator $Tf=\cL f$ (understood in the distributional sense) is defined on the domain
\begin{equation*}
\dom T=\begin{cases} H^2(\Omega) & \text{if }\Omega\text{ is bounded},\\[1ex]
                           \bigl\{f\in H^1(\Omega)\colon \cL f\in L^2(\Omega),\,
  f|_{\Omega'}\in H^2(\Omega')\bigr\} & \text{if }\Omega\text{ is unbounded},
                          \end{cases}
\end{equation*}
where in the unbounded case
$\Omega'\subset\Omega$ is a bounded subdomain of
$\Omega$
with smooth boundary such that $\partial\Omega \subset \partial\Omega'$.
\end{definition}

In the unbounded case in Definition~\ref{domt} we can choose, for instance,
$\Omega' = \Omega\cap B_R(0)$, where $B_R(0)=\{x\in\RR^n\colon \|x\|<R\}$ and
$R$ is big enough so that $\RR^n\backslash\Omega \subset B_R(0)$.
Since the condition $\cL f\in L^2(\Omega)$ implies that $f\in H^2_{\rm loc}(\Omega)$
(see, e.g.\ \cite[Theorem~2.3.2]{LM72}), it is clear that the set on the right-hand side
of $\dom T$ in the case of unbounded $\Omega$ is independent of $\Omega'$.
In both cases ($\Omega$ bounded or unbounded), functions $f$ in $\dom T$ are in $H^2$
in a neighbourhood of $\partial\Omega$,
and hence $f|_{\partial\Omega}$ and $\frac{\partial
f}{\partial\nu_{\cL}}\bigl|_{\partial\Omega}$
are well defined and have values in $H^{3/2}(\partial\Omega)$ and
$H^{1/2}(\partial\Omega)$, respectively.
Define the Dirichlet, Neumann and minimal operator associated with $\cL$ by
\begin{alignat*}{2}
  \AD f &= \cL f, \quad & \dom\AD &= \bigl\{f\in\dom T\colon
    f|_{\partial\Omega}=0\bigr\}, \\[1ex]
  \AN f &= \cL f, & \dom\AN &= \biggl\{f\in\dom T\colon
    \frac{\partial f}{\partial\nu_\cL}\Bigl|_{\partial\Omega}=0\biggr\}, \\[1ex]
  Af &= \cL f, & \dom A &= \biggl\{f\in\dom T\colon f|_{\partial\Omega}=0,\,
    \frac{\partial f}{\partial\nu_\cL}\Bigl|_{\partial\Omega}=0\biggr\}.
\end{alignat*}
In the following theorem it is shown how a quasi boundary triple can be
defined in the present situation. For the convenience of the reader the self-adjointness of
$\AN$ in the case of an unbounded domain will be shown in full detail, the remaining assertions
are essentially a consequence of Theorem~\ref{suff_cond_qbt}.

\begin{theorem}\label{qbtthm}
Let $\cL$ be the uniformly elliptic differential expression from \eqref{cl}, let
$T$, $\AD$, $\AN$, $A$ be the differential operators from above and
define the boundary mappings
\begin{equation*}
  \Gamma_0 \hat f \defeq \frac{\partial
f}{\partial\nu_\cL}\!\Bigm|_{\partial\Omega}\quad\text{and}\quad
  \Gamma_1 \hat f \defeq f |_{\partial\Omega},\qquad
  \hat f=\begin{pmatrix} f \\ Tf\end{pmatrix},\quad f\in\dom T.
\end{equation*}
Then $A$ is a densely defined closed, symmetric operator in $L^2(\Omega)$, the operators
$\AN=\ker\Gamma_0$ and $\AD=\ker\Gamma_1$ are self-adjoint extensions of $A$,
and
$\{L^2(\partial\Omega),\Gamma_0,\Gamma_1\}$ is a quasi boundary triple for
$A^*$.
Moreover,
\begin{equation} \label{halfgreen1}
  (Tf,g) = \fra[f,g]
  - (\Gamma_0f,\Gamma_1g)
\end{equation}
holds for all $f,g\in\dom(T)$, where
\begin{equation} \label{formA}
  \fra[f,g] \defeq \int_\Omega\Biggl(\sum_{j,k=1}^n a_{jk}\frac{\partial
f}{\partial x_k}
  \frac{\partial \ov{g}}{\partial x_j} + af\ov{g}\Biggr),\qquad f,g\in H^1(\Omega).
\end{equation}
\end{theorem}

\begin{proof}
If $\Omega$ is bounded, the assertions in the theorem apart from
\eqref{halfgreen1} were proved
in \cite[Proposition~4.1]{BL07}.  The proof of \eqref{halfgreen1} follows easily
from known
results; see also the proof below for the case that $\Omega$ is unbounded.

Now let $\Omega$ be unbounded.
First we show that $\AN$ as defined above is self-adjoint.
Let the symmetric quadratic form $\fra[f,g]$ be as in the theorem.
Because of the boundedness of the coefficients and the uniform ellipticity, this
quadratic form can be compared with the form
\[
  \fra_0[f,g] = \int_\Omega \grad f \cdot\ov{\grad g},
\]
which corresponds to the Laplace operator, namely, there exist
constants $c_1$, $c_2$, $d_1$, $d_2$ such that
\[
  c_1 \fra_0[f,f] + d_1\|f\|^2 \le \fraN[f,f] \le c_2 \fra_0[f,f] + d_2\|f\|^2.
\]
Since $\|f\|^2+\fra_0[f,f]=\|f\|_{H^1(\Omega)}^2$, this implies that
the form $\fraN$ is closed and bounded from below. Hence,
by \cite[Theorem~VI.2.1]{kato} there exists a self-adjoint operator $\tAN$ in $L^2(\Omega)$ with
$\dom\tAN\subset\dom\fraN=H^1(\Omega)$ which
is bounded from below and represents the form $\fraN$, i.e.\
\begin{equation} \label{defdomtAN}
  (\tAN f,g) = \fraN[f,g]
\end{equation}
for all $f\in\dom \tAN$ and $g\in H^1(\Omega)$.

We claim that the domain of $\tAN$ is equal to
\begin{equation}\label{domtAN}
  \Bigl\{f \in H^1(\Omega)\colon \cL f\in L^2(\Omega),\,
  \frac{\partial f}{\partial\nu_\cL}\Bigl|_{\partial\Omega}=0\Bigr\}
\end{equation}
and that $\tAN f=\cL f$ for $f\in\dom\tAN$.
In fact, let $f\in\dom \tAN$.  Then \eqref{defdomtAN} is true in particular for $g\in
C_0^\infty(\Omega)$,
which implies
\[
  (\tAN f,g) = \fraN[f,g] = (f,\cL g) = \langle \cL f,g\rangle,
\]
where the last term is the application of the distribution $\cL f$ to the test
function $g$; the second equality follows from the definition of distributional
derivatives.
This implies that $\cL f$ is a regular distribution and equal to
$\tAN f\in L^2(\Omega)$.
The formula
\begin{equation} \label{halfgreen2}
  (\cL u,v) = \fraN[u,v]
  - \int_{\partial\Omega}\frac{\partial u}{\partial\nu_{\cL}}\ov{v}
\end{equation}
is valid for all $u\in H^1(\Omega)$ such that $\cL u\in L^2(\Omega)$ and
all $v\in H^1(\Omega)$ such that one of the two functions has bounded
support\footnote{Indeed, for $u,v\in H^2(\Omega)$ and bounded $\Omega$, formula
\eqref{halfgreen2} is well known. Since in this case $H^2(\Omega)$ is dense in
$H^1_\cL(\Omega)\defeq\{w\in H^1(\Omega)\colon\cL w\in L^2(\Omega)\}$ equipped with the norm
$\Vert w\Vert_{H^1}+\Vert\cL w\Vert_{L^2}$ and
$\tfrac{\partial}{\partial\nu_\cL}: H^1_\cL(\Omega)\rightarrow H^{-1/2}(\partial\Omega)$
is continuous (see \cite{G68,LM72}), an approximation argument implies \eqref{halfgreen2}.}.
The derivative of $u$ under the integral is in $H^{-1/2}(\partial\Omega)$,
the trace of $v$ is in $H^{1/2}(\partial\Omega)$; so the integral is
understood as a dual pairing of $H^{-1/2}(\partial\Omega)$ and
$H^{1/2}(\partial\Omega)$.
Since boundary values of $H^1(\Omega)$-functions with bounded support exhaust
the space $H^{1/2}(\partial\Omega)$, relations \eqref{defdomtAN} and
\eqref{halfgreen2} with $u=f$ and $v=g$ yield that
$\frac{\partial f}{\partial\nu_\cL}\bigl|_{\partial\Omega}=0$.
Hence $f$ is in the set in \eqref{domtAN}.
Conversely, let $f$ be in the set in \eqref{domtAN}.
Then by \eqref{halfgreen2} we have
\[
  (\cL f,g) = \fraN[f,g]
\]
for all $g\in C^\infty(\Omega)$ with bounded support.  This implies that
$f\in\dom \tAN$ and $\tAN f=\cL f$ by \cite[Theorem~VI.2.1(iii)]{kato} since
$\{g\in C^\infty(\Omega)\colon \supp g\text{ bounded}\}$ is dense
in $H^1(\Omega)$, which implies that it is a core of $\fraN$.

We show that functions in $\dom \tAN$ are in $H^2$ in a neighbourhood of
$\partial\Omega$.
Let $R>0$ be such that $\RR^n\backslash\Omega\subset B_R(0)$ and set
$\Omega'\defeq\Omega\cap B_R(0)$.  Moreover, choose a $C^\infty$-function
$\varphi$ defined on $\Omega$ such that $\supp\varphi\subset\Omega'$,
that $\varphi(x)=1$ in a neighbourhood of $\partial\Omega$ and that
$\varphi(x)=0$
in a neighbourhood of $S_R(0)\defeq\{x\in\RR^n\colon \|x\|=R\}$.
Let $f$ be in $\dom \tAN$, i.e.\ in the set in \eqref{domtAN}.
We want to show that $\varphi f\in\dom\tAN$.
Clearly, $\varphi f\in H^1(\Omega)$.  Because
\[
  \cL(\varphi f) = \varphi(\cL f) - \sum_{j,k=1}^n\biggl[
  2a_{jk}\frac{\partial\varphi}{\partial x_j}\frac{\partial f}{\partial x_k}
  + f\frac{\partial a_{jk}}{\partial x_j}\frac{\partial\varphi}{\partial x_k}
  + a_{jk}f\frac{\partial^2\varphi}{\partial x_j\partial x_k}\biggr],
\]
$f\in H^1(\Omega)$ and the derivatives of $a_{jk}$ and $\varphi$ are uniformly
bounded
on the bounded set $\supp\varphi$, we can deduce that $\cL(\varphi f)\in
L^2(\Omega)$.
The validity of the boundary condition
$\frac{\partial(\varphi f)}{\partial\nu_\cL}\bigl|_{\partial\Omega}=0$
is clear from the fact that $\varphi(x)=1$ in a neighbourhood of
$\partial\Omega$.
It follows that $\varphi f$ is in the set in \eqref{domtAN} and hence in $\dom
\tAN$.
Now define a quadratic form $\fra_{\Omega',\rm N,D}$ in $L^2(\Omega')$ by the
formula in \eqref{formA} with domain
\[
  \dom \fra_{\Omega',\rm N,D} = \bigl\{h\in H^1(\Omega')\colon
f|_{S_R(0)}=0\bigr\}.
\]
This form defines a self-adjoint operator $A_{\Omega',\rm N,D}$:
\begin{align*}
   A_{\Omega',\rm N,D}h & = \cL h, \\[1ex]
   \dom A_{\Omega',\rm N,D} & = \biggl\{h\in H^2(\Omega')\colon h|_{S_R(0)}=0,\,
  \frac{\partial h}{\partial\nu_{\cL}}\Bigl|_{\partial\Omega}=0\biggr\}.
\end{align*}
Since $f\in\dom \tAN$ and any function $g$ in $\dom \fra_{\Omega',\rm N,D}$
can be extended by $0$ to a function $\wt g$ in $H^1(\Omega)$, we have
\[
  \bigl((\tAN f)\vert_{\Omega'},g\bigr)_{L^2(\Omega')}
  = \bigl(\tAN f,\wt g\bigr)_{L^2(\Omega)} = \fraN[f,\wt g]
  = \fra_{\Omega',\rm N,D}[f|_{\Omega'},g]
\]
for all $g\in\dom \fra_{\Omega',\rm N,D}$.  By \cite[Theorem~VI.2.1(iii)]{kato}
this implies that $f|_{\Omega'}\in\dom A_{\Omega',\rm N,D}$ and
hence $f|_{\Omega'}\in H^2(\Omega')$.

It follows that
\begin{align*}
  \dom \tAN &= \Bigl\{f \in H^1(\Omega)\colon \cL f\in L^2(\Omega),\,
  \frac{\partial f}{\partial\nu_\cL}\Bigl|_{\partial\Omega}=0\Bigr\} \\[1ex]
  &= \Bigl\{f \in H^1(\Omega)\colon \cL f\in L^2(\Omega),\,
  \frac{\partial f}{\partial\nu_\cL}\Bigl|_{\partial\Omega}=0,\,f|_{\Omega'}\in
  H^2(\Omega')\Bigr\} \\[1ex]
  &= \Bigl\{f \in\dom T\colon
  \frac{\partial f}{\partial\nu_\cL}\Bigl|_{\partial\Omega}=0\Bigr\} \\[1ex]
  &= \dom \AN
\end{align*}
and that $\AN=\tAN$ is a self-adjoint operator in $L^2(\Omega)$.
With a similar reasoning and using the quadratic form $\fraN$ restricted
to $H_0^1(\Omega)$ one can show that $\AD$ is a self-adjoint operator in $L^2(\Omega)$.

Next we show that $\{L^2(\partial\Omega),\Gamma_0,\Gamma_1\}$ is a quasi boundary triple using
Theorem~\ref{suff_cond_qbt}.
It follows from the considerations before the statement of the current theorem
that $\Gamma_0$ and $\Gamma_1$ are well defined.
Moreover,
$$\ran \Gamma=\ran\begin{pmatrix} \Gamma_0 \\ \Gamma_1\end{pmatrix}=H^{1/2}(\partial\Omega)\times
H^{3/2}(\partial\Omega)$$
(see, e.g.\ \cite[Theorem~1.8.3]{LM72}),
which is dense in $L^2(\partial\Omega)\times L^2(\partial\Omega)$.

In order to show Green's identity we first show the identity \eqref{halfgreen1}.
Let $\Omega'$ and $\varphi$ be as above and set $\psi\defeq
1-\varphi$.
If $f,g\in\dom T$, then $(\varphi f)|_{\Omega'},(\varphi g)|_{\Omega'} \in H^2(\Omega')$
and $\psi f,\psi g\in \dom \AN$.  Using \eqref{halfgreen2} and \eqref{defdomtAN}
we obtain
\begin{align*}
  \bigl(Tf,g\bigr)_{L^2(\Omega)} &= \bigl(Tf,\varphi g\bigr)_{L^2(\Omega')}
  + \bigl(T(\varphi f),\psi g\bigr)_{L^2(\Omega')} + \bigl(T(\psi f),\psi
  g\bigr)_{L^2(\Omega)} \\[1ex]
  &= \fraN[f,\varphi g]-\int_{\partial\Omega}\frac{\partial f}{\partial\nu_\cL}\ov{\varphi g} \\[1ex]
  &\quad + \fraN[\varphi f,\psi g]
  - \int_{\partial\Omega}\frac{\partial(\varphi f)}{\partial\nu_\cL}\ov{\psi g}
  \\[1ex]
  &\quad + \fraN[\psi f,\psi g] \\[1ex]
  &= \fra[f,g] - \bigl(\Gamma_0f,\Gamma_1g\bigr)_{L^2(\partial\Omega)}
\end{align*}
since $\varphi(x)=1$ and $\psi(x)=0$ in a neighbourhood of $\partial\Omega$,
which proves \eqref{halfgreen1}.
The abstract Green identity \eqref{green1} follows immediately from this and
the symmetry of $\fra$.

Now we can apply Theorem~\ref{suff_cond_qbt} to obtain that $A$ is a closed,
symmetric operator and that $\{L^2(\partial\Omega),\Gamma_0,\Gamma_1\}$ is a
quasi boundary triple. Moreover, since $T$ is an operator,
we conclude that $T^*=A$ is densely defined.
\end{proof}

Observe that for the quasi boundary triple in Theorem~\ref{qbtthm} we have
\[
  \cG_0 = \ran\Gamma_0 = H^{1/2}(\partial\Omega)\qquad\text{and}\qquad
  \cG_1 = \ran\Gamma_1 = H^{3/2}(\partial\Omega).
\]
We also note that the triple $\{L^2(\partial\Omega),\Gamma_0,\Gamma_1\}$
is not a generalized boundary triple or a boundary relation in the sense of
\cite{DM95,DHMS06}
and we refer to \cite{G68,BGW09} for a modified approach that leads to an
ordinary boundary triple for $A^*$.
One of the advantages of the quasi boundary triple in Theorem~\ref{qbtthm}
is that the corresponding Weyl function is the inverse of the usual
Dirichlet-to-Neumann map, whereas the Weyl function corresponding
to the ordinary boundary triple from \cite{G68,BGW09}
(which differs by an unbounded constant from the Dirichlet-to-Neumann map)
is more difficult to interpret; see also \cite[Proposition 4.1]{B10}.
The $\gamma$-field corresponding to the quasi boundary triple from
Theorem~\ref{qbtthm} is the Poisson operator for the Neumann problem
associated with $\cL$.  This is summarized in the following proposition,
whose proof is clear from the definitions of $\gamma(\lambda)$ and $M(\lambda)$.

\begin{proposition}\label{gammamh21}
Let $\dom T$ be as in Definition~\ref{domt} and denote for
$\varphi\in H^{1/2}(\partial\Omega)$ and $\lambda\in\rho(\AN)$ the
unique solution of
\begin{equation*}
  \cL h=\lambda h,\qquad \frac{\partial
h}{\partial\nu_\cL}\!\Bigm|_{\partial\Omega}=\varphi
\end{equation*}
in $\dom T$ by $f_\lambda(\varphi)$.
Then the $\gamma$-field
$\gamma$ and Weyl function $M$ associated with the quasi boundary triple
$\{L^2(\partial\Omega),\Gamma_0,\Gamma_1\}$ in Theorem~\ref{qbtthm} are given by
\begin{alignat*}{2}
  \gamma(\lambda)\colon& H^{1/2}(\partial\Omega)\rightarrow L^2(\Omega),\qquad &
  \varphi&\mapsto f_\lambda(\varphi), \\[1ex]
  M(\lambda)\colon& H^{1/2}(\partial\Omega)\rightarrow
  H^{3/2}(\partial\Omega),\qquad &
  \varphi&\mapsto f_\lambda(\varphi)\vert_{\partial\Omega}.
\end{alignat*}
\end{proposition}

It is known from \cite{LM72,S69} that $M(\lambda)$ can be extended to a bounded
operator acting between various Sobolev spaces. For the convenience of the reader we give a short proof based on a duality and
interpolation argument.

\begin{lemma} \label{le.interpolation}
Let $s\in[-\frac{3}{2},\frac{1}{2}]$ and $\lambda\in\rho(\AN)$.
Then $M(\lambda)$ can be extended to a bounded operator from
$H^s(\partial\Omega)$ to $H^{s+1}(\partial\Omega)$.
Moreover, the closure $\ov{M(\lambda)}$ in $L^2(\partial\Omega)$ is a compact operator in $L^2(\partial\Omega)$ with
$\ran(\ov{M(\lambda)})\subset H^1(\partial\Omega)$.
\end{lemma}

\begin{proof}
Denote by $\langle\,\cdot,\cdot\rangle_t$ the dual pairing of
$H^t(\partial\Omega)$
and $H^{-t}(\partial\Omega)$ for $t\ge0$, i.e.\ $\langle x,y\rangle_t$ is
defined for
$x\in H^t(\partial\Omega)$ and $y\in H^{-t}(\partial\Omega)$,
$\langle\,\cdot,\cdot\rangle_t$
is linear in the first and semi-linear in the second component and satisfies
\begin{equation} \label{dualpairing}
  \langle x,y\rangle_t = (x,y)
  \qquad \text{for }x\in H^t(\partial\Omega),\; y\in L^2(\partial\Omega),
\end{equation}
where $(\,\cdot,\cdot)$ denotes the inner product in $L^2(\partial\Omega)$.

In the following let $\lambda\in\rho(\AN)$.
It follows from Proposition~\ref{gammaprop}(v) that $M(\lambda)$ is closable in
$L^2(\partial\Omega)$ and from Proposition~\ref{gammaprop}(iii) that it maps
$H^{1/2}(\partial\Omega)$ into $H^{3/2}(\partial\Omega)$.
Therefore $M(\lambda)$ is closed and hence bounded from
$H^{1/2}(\partial\Omega)$
to $H^{3/2}(\partial\Omega)$.

The Banach space adjoint $(M(\ov\lambda))'$ of $M(\ov\lambda)$ is a bounded
operator
from $H^{-3/2}(\partial\Omega)$ to $H^{-1/2}(\partial\Omega)$,
where $(M(\ov\lambda))'$ is defined by
\begin{equation} \label{BAdjoint}
  \bigl\langle x,(M(\ov\lambda))'y\bigr\rangle_{1/2}
  = \bigl\langle M(\ov\lambda)x,y\bigr\rangle_{3/2},
  \qquad x\in H^{1/2}(\partial\Omega),\; y\in H^{-3/2}(\partial\Omega).
\end{equation}
Proposition~\ref{gammaprop}(v) yields that
\begin{equation} \label{HAdjoint}
  \bigl(M(\ov\lambda)x,y\bigr) = \bigl(x,M(\lambda)y\bigr), \qquad x,y\in
H^{1/2}(\partial\Omega).
\end{equation}
Combining \eqref{dualpairing}, \eqref{BAdjoint} and \eqref{HAdjoint} we obtain
for
$x,y\in H^{1/2}(\partial\Omega)$ that
\begin{align*}
  \bigl\langle x,M(\lambda)y\bigr\rangle_{1/2}
  &= \bigl(x,M(\lambda)y\bigr)
  = \bigl(M(\ov\lambda)x,y\bigr) \\[1ex]
  &= \bigl\langle M(\ov\lambda)x,y\bigr\rangle_{3/2} \\[1ex]
  &= \bigl\langle x,(M(\ov\lambda))'y\bigr\rangle_{1/2}.
\end{align*}
This implies that $M(\lambda)y=(M(\ov\lambda))'y$ for $y\in
H^{1/2}(\partial\Omega)$.
Hence the bounded operator $(M(\ov\lambda))'\colon H^{-3/2}(\partial\Omega)
\to H^{-1/2}(\partial\Omega)$ is an extension of
$M(\lambda)\colon H^{1/2}(\partial\Omega)\to H^{3/2}(\partial\Omega)$.
Now interpolation (see, e.g.\ \cite[Theorems~5.1 and 7.7]{LM72}) implies that
\begin{equation}\label{Mextbdd}
  (M(\ov\lambda))'\big|_{H^s(\partial\Omega)}\colon
  H^s(\partial\Omega)\to H^{s+1}(\partial\Omega)
\end{equation}
is bounded for $s\in[-\frac{3}{2},\frac{1}{2}]$.

Since $\ov{M(\lambda)} = (M(\ov\lambda))'|_{L^2(\partial\Omega)}$, we know
from \eqref{Mextbdd} that $\ov{M(\lambda)}$ is bounded from
$L^2(\partial\Omega)$
to $H^1(\partial\Omega)$.
Together with the compactness of the embedding of $H^1(\partial\Omega)$ into
$L^2(\partial\Omega)$ (see, e.g.\ \cite[Theorem~7.10]{W}) this shows
that $\ov{M(\lambda)}$ is a compact operator in $L^2(\partial\Omega)$.
\end{proof}

In \cite{BL07} and \cite{BLLLP10} quasi boundary triples for elliptic operators
were also
studied in the framework of the Beals space $\cD_1(\Omega)$ when $\Omega$ is
bounded with a smooth boundary.
In this setting sufficient conditions on the parameter $\Theta$ in
$L^2(\partial\Omega)$ that ensure self-adjointness of the corresponding elliptic
operator
\begin{equation*}
  A_\Theta=\cL\upharpoonright\bigl\{f\in \cD_1(\Omega)\colon \Gamma\hat
f\in\Theta\bigr\}
\end{equation*}
were obtained in \cite[Theorem~4.8]{BL07}. The next result gives a sufficient
condition
on $\Theta$ in the $H^2$-framework which also covers a large class of
Robin type
boundary conditions; cf.\ Corollary~\ref{condsacor} below. We note that $\Omega$
is allowed
to be unbounded but $\partial\Omega$ is assumed to be compact and smooth.

\begin{theorem}\label{condsa}
Let $\{L^2(\partial\Omega),\Gamma_0,\Gamma_1\}$ be the quasi boundary triple
from Theorem~\ref{qbtthm} and $\Gamma=(\Gamma_0,\Gamma_1)^\top$.
Let $\Theta$ be a self-adjoint relation in $L^2(\partial\Omega)$ such that
$0\notin\sess(\Theta)$ and
\[
  \Theta^{-1}\bigl(H^1(\partial\Omega)\bigr) \subset H^{1/2}(\partial\Omega).
\]
Then the realization $A_\Theta=\cL\upharpoonright\{f\in\dom T\colon
\Gamma\hat f\in\Theta\}$
is self-adjoint in $L^2(\Omega)$.
In particular, if $B$ is a bounded operator in $L^2(\partial\Omega)$ with
$B(H^1(\partial\Omega)) \subset H^{1/2}(\partial\Omega)$, then the realization
\begin{equation*}
  A_{B^{-1}}f=\cL f,\quad \dom A_{B^{-1}}=\left\{f\in \dom T\colon
  B\bigl(f|_{\partial\Omega}\bigr)=\frac{\partial
f}{\partial\nu_{\cL}}\!\Bigm|_{\partial\Omega}\right\}
\end{equation*}
is a self-adjoint operator in $L^2(\Omega)$.
\end{theorem}

\begin{proof}
We can directly apply Theorem~\ref{th.A_Th_self_adj} and Remark~\ref{remremrem}
since $\ran\ov{M(\lambda)} \subset H^1(\partial\Omega)$ for all $\lambda\in\rho(\AN)$
by Lemma~\ref{le.interpolation}.
\end{proof}

The next corollary is an immediate consequence of Theorem~\ref{condsa}.
In includes, in particular, classical Robin boundary conditions.

\begin{corollary}\label{condsacor}
Let $\beta\in C^1(\partial\Omega)$ be a real-valued function on
$\partial\Omega$ and $k\in C^1(\partial\Omega\times\partial\Omega)$ a
symmetric kernel on $\partial\Omega$, i.e.\ $k(x,y)=\ov{k(y,x)}$
for $x,y\in\partial\Omega$.
Then the realization
\[
  A_{B^{-1}}f=\cL f, \quad
  \dom A_{B^{-1}}=\left\{f\in\dom T\colon
  B\bigl(f|_{\partial\Omega}) = \frac{\partial f}{\partial\nu_{\cL}}\!
  \Bigm|_{\partial\Omega}\right\},
\]
where
\[
  (B\varphi)(x) = \beta\varphi(x) + \int_{\partial\Omega} k(x,y)\varphi(y)dy,
  \qquad \varphi \in L^2(\partial\Omega),
\]
is a self-adjoint operator in $L^2(\Omega)$.
\end{corollary}

Before we continue to investigate resolvent differences of self-adjoint
realizations of $\cL$,
we need the following general lemma on the singular values of operators mapping
into
Sobolev spaces; see also \cite{BLLLP10}. The proof is essentially an application
of
results on the asymptotic behaviour of eigenvalues of the Laplace--Beltrami
operator on
compact manifolds; for similar ideas see the proof of Proposition~5.4.1 in
\cite{Ag90}.

\begin{lemma} \label{le.s_emb}
Let $\Sigma$ be an $n-1$-dimensional compact manifold without boundary, let
$\cK$ be a Hilbert space and $B\in\cB(\cK,H^{r_1}(\Sigma))$
with $\ran B \subset H^{r_2}(\Sigma)$ where $r_2>r_1\geq 0$. Then $B$ is compact and
its singular values $s_k$ satisfy
\[
  s_k(B) = O\bigl(k^{-\frac{r_2-r_1}{n-1}}\bigr), \quad k\to\infty.
\]
In particular, $B\in\sS_{\frac{r_2-r_1}{n-1},\infty}(\cK,H^{r_1}(\Sigma))$ and
$B\in\sS_p(\cK,H^{r_1}(\Sigma))$
for $p>\frac{n-1}{r_2-r_1}$.
\end{lemma}

\begin{proof}
Let $\Lambda_{r_1,r_2} \defeq (I - \Delta_{\rm LB}^{\Sigma})^\frac{r_2-r_1}2$,
where $\Delta_{\rm LB}^{\Sigma}$ denotes the Laplace--Beltrami operator on $\Sigma$.
The operator $\Lambda_{r_1,r_2}$ is an isometric isomorphism from
$H^{r_2}(\Sigma)$
onto $H^{r_1}(\Sigma)$. From \cite[(5.39) and the text below]{Ag90} we obtain for the asymptotics of the
eigenvalues $\lambda_k(I-\Delta_{\rm LB}^\Sigma) \sim Ck^{\frac{2}{n-1}}$ with some constant $C$.
This implies
\[
  s_k(\Lambda_{r_1,r_2}^{-1}) = O\bigl(k^{-\frac{r_2 -r_1}{n-1}}\bigr), \quad
k\to\infty.
\]
We can write $B$ in the form
\begin{equation}\label{blambda}
  B = \Lambda_{r_1,r_2}^{-1}(\Lambda_{r_1,r_2} B).
\end{equation}
The operator $B$ is closed as an operator from $\cK$ into $H^{r_1}(\Sigma)$,
hence also closed as an operator from $\cK$ into $H^{r_2}(\Sigma)$,
which implies that it is bounded from $\cK$ into $H^{r_2}(\Sigma)$.
Therefore the operator $\Lambda_{r_1,r_2} B$ is bounded from $\cK$ into
$H^{r_1}(\Sigma)$, and hence the assertions follow from \eqref{blambda}.
\end{proof}

The next result is essentially a consequence of the previous lemma,
Lemma~\ref{le.interpolation} and general properties of the $\gamma$-field
and the Weyl function established in Section~\ref{3.1}.

\begin{proposition}\label{gammamh2}
Let $\{L^2(\partial\Omega),\Gamma_0,\Gamma_1\}$ be the quasi boundary triple
from Theorem~\ref{qbtthm}. Then for $\lambda\in\rho(\AN)$,
the associated $\gamma$-field $\gamma$, the Weyl function $M$ and the
closures $\overline{M(\lambda)}$, $\overline{\Imag M(\lambda)}$ satisfy
\begin{itemize}
\item[(i)]
$\dis \gamma(\lambda)^*
\in\sS_{\frac{3}{2(n-1)},\infty}\bigl(L^2(\Omega),L^2(\partial\Omega)\bigr)$;
\item[(ii)]
$\dis M(\lambda)
\in\sS_{\frac{1}{n-1},\infty}\bigl(H^{1/2}(\partial\Omega)\bigr)$;
\item[(iii)]
$\dis \overline{\Im
M(\lambda)}\in\sS_{\frac{3}{n-1},\infty}(L^2(\partial\Omega))$;
\item[(iv)]
$\dis \overline{M(\lambda)}\in\sS_{\frac{1}{n-1},\infty}(L^2(\partial\Omega))$.
\end{itemize}
\end{proposition}

\begin{proof}
Note that $\partial\Omega$ is a finite union of $C^\infty$-manifolds.

Assertion (i) follows from Lemma~\ref{le.s_emb} with $r_1=0$ and $r_2 =
\frac{3}{2}$ since
$\gamma(\lambda)^*$ is a bounded operator from $\cK=L^2(\Omega)$ to
$L^2(\partial\Omega)$ with
$\ran\Gamma_1 \subset H^{3/2}(\partial\Omega)$ by
Proposition~\ref{gammaprop}(ii).

(ii) By Lemma~\ref{le.interpolation}, the operator $M(\lambda)$,
$\lambda\in\rho(\AN)$,
is bounded as an operator in $H^{1/2}(\partial\Omega)$.
Hence Lemma~\ref{le.s_emb} applied with $\cK=H^{1/2}(\partial\Omega)$,
$r_1 = \frac{1}{2}$ and $r_2 = \frac{3}{2}$ yields the assertion.

(iii) From Proposition~\ref{gammaprop}(v) we obtain the relation
\[
  \overline{\Im M(\lambda)}=(\Im\lambda)
\,\gamma(\lambda)^*\overline{\gamma(\lambda)}.
\]
It follows from (i) and Lemma~\ref{splemma}(iii) that the right-hand side is in
\[
  \sS_{\frac{3}{2(n-1)},\infty}\cdot
  \sS_{\frac{3}{2(n-1)},\infty}
  = \sS_{\frac{3}{n-1},\infty}.
\]

(iv) The statement follows from Lemmas~\ref{le.interpolation} and
\ref{le.s_emb} with $\cK=L^2(\partial\Omega)$, $r_1=0$ and $r_2=1$.
\end{proof}

\begin{remark}
It is not difficult to check that
$\{L^2(\partial\Omega),\Gamma_1,-\Gamma_0\}$ is also a quasi boundary triple for
the operator $A^*$.
The corresponding Weyl function $\wh M$ is (up to a minus sign) the Dirichlet to
Neumann map from
$H^{3/2}(\partial\Omega)\rightarrow H^{1/2}(\partial\Omega)$, i.e.\
for $\lambda\in\rho(\AD)$ the operator $\wh M(\lambda)$ maps
the Dirichlet boundary value $f_\lambda(\varphi)\vert_{\partial\Omega}$ of the
solution
$f_\lambda(\varphi)\in\dom T$ of $\cL h=\lambda h$, $h\vert_{\partial\Omega}=\varphi$,
onto the
(minus) Neumann boundary value
$-\frac{\partial f_\lambda(\varphi)}{\partial\nu_\cL}\vert_{\partial\Omega}$.
One of the main reasons that we do not use the quasi boundary triple
$\{L^2(\partial\Omega),\Gamma_1,-\Gamma_0\}$ here is that the values of the
corresponding
Weyl function $\wh M$ are unbounded operators in $L^2(\partial\Omega)$.
\end{remark}

\subsection{Spectral estimates for resolvent differences of
self-adjoint elliptic operators on bounded and exterior domains}
\label{4.2}

Throughout this section let $\{L^2(\partial\Omega),\Gamma_0,\Gamma_1\}$ be the
quasi boundary triple from Theorem~\ref{qbtthm} with corresponding $\gamma$-field
and Weyl function from Proposition~\ref{gammamh21}.
If $\Omega$ is unbounded, let $\Omega'$ be as in Definition~\ref{domt};
if $\Omega$ is bounded, set $\Omega'\defeq\Omega$.
For a linear relation $\Theta$ in $L^2(\partial\Omega)$
the corresponding realization $A_\Theta$ of $\cL$ is given by
\begin{align*}
  A_\Theta f &= \cL f, \\[1ex]
  \dom A_\Theta
  &= \left\{f\in H^1(\Omega)\colon
  \cL f\in L^2(\Omega),\,f|_{\Omega'}\in H^2(\Omega'),\,
  \begin{pmatrix} \dfrac{\partial f}{\partial\nu_\cL}\!\Bigm|_{\partial\Omega}
  \\[2ex]
  f|_{\partial\Omega} \end{pmatrix} \in \Theta\right\};
\end{align*}
cf.\ \eqref{atheta}, \eqref{athetaop} and Theorem~\ref{qbtthm}.
A sufficient condition for the self-adjointness of $A_\Theta$ was given in Theorem~\ref{condsa}. In the
following we
apply the general results from Section~\ref{3.3} to resolvent differences of
self-adjoint
realizations of the elliptic differential expression $\cL$ in $L^2(\Omega)$.
The statements in the next three theorems are consequences of
Proposition~\ref{gammamh2} and
Theorems~\ref{th.resolv_diff1}, \ref{th.resolv_diff3}
and \ref{th.resolv_diff4}, respectively.

\begin{theorem}\label{thm1}
Let $\AN$ be the Neumann operator associated with $\cL$ and let
$\Theta$ be a self-adjoint relation in $L^2(\partial\Omega)$ such
that $0\notin\sess(\Theta)$ and $A_\Theta$ is a self-adjoint operator.
Then for all $\lambda\in\rho(A_\Theta)\cap\rho(\AN)$ the singular values $s_k$
of the
resolvent difference
\begin{equation}\label{resdiff1}
  (A_\Theta-\lambda)^{-1} - (\AN-\lambda)^{-1}
\end{equation}
satisfy $s_k = O\bigl(k^{-\frac{3}{n-1}}\bigr)$, $k\to\infty$, and the
expression in \eqref{resdiff1}
is in $\sS_p(L^2(\Omega))$ for all $p>\frac{n-1}{3}$.
\end{theorem}

\begin{proof}
By Proposition~\ref{gammamh2}(i) we have
$\gamma(\lambda)^*\in\sS_{\frac{3}{2(n-1)},\infty}(L^2(\Omega),
L^2(\partial\Omega))$.
Hence we can apply Theorem~\ref{th.resolv_diff1}, which yields that the
resolvent difference in \eqref{resdiff1} belongs to
\[
  \sS_{\frac{3}{2(n-1)},\infty}\cdot
  \sS_{\frac{3}{2(n-1)},\infty}
  = \sS_{\frac{3}{n-1},\infty} \subset \sS_p, \qquad p>\frac{n-1}{3},
\]
where we used Lemma~\ref{splemma}(iii) and (ii).
\end{proof}

As an immediate consequence of Theorem~\ref{thm1}
we obtain that the essential spectra of $A_\Theta$ and $\AN$ coincide,
\begin{equation*}
 \sess(A_\Theta)=\sess(\AN).
\end{equation*}
In the case of a bounded domain these sets are empty, in the unbounded case
the following proposition shows how
close eigenvalues of $A_\Theta$ have to be to eigenvalues of $\AN$.

\begin{proposition}\label{cor0}
Let $\Omega$ be unbounded, let $\AN$ be the Neumann operator associated with $\cL$ and let
$\Theta$ be a self-adjoint relation in $L^2(\partial\Omega)$ such
that $0\notin\sess(\Theta)$ and $A_\Theta$ is a self-adjoint operator.
If\, $\lambda_k$, $k=1,2,\dots$, are isolated eigenvalues of $A_\Theta$
converging
to some $\gamma\in\RR$, then there exist numbers $\mu_k$, $k=1,2,\dots$, which
are
isolated eigenvalues $\mu_k$, $k=1,2,\dots$, of $\AN$ or equal to $\gamma$
(where the number $\gamma$ may appear arbitrarily many times but
an eigenvalue only up to its multiplicity) such that
\begin{equation}\label{ellp}
  \sum_{k=1}^\infty |\lambda_k-\mu_k|^p < \infty
  \qquad \text{for all}\,\,\,\,\, p>\frac{n-1}{3},\,\,\,p\geq 1.
\end{equation}
\end{proposition}

\begin{proof}
The spectrum of $\AN$ is bounded below, which follows from \eqref{halfgreen1}
and the
ellipticity of $\cL$.  Hence also the essential spectrum of $A_\Theta$ is
bounded below,
and we can choose a number $\lambda\in\RR\cap\rho(\AN)\cap\rho(A_\Theta)$.
Because of Theorem~\ref{thm1} we can apply \cite[Theorem~II]{K87} to the
operators $(\AN-\lambda)^{-1}$ and $(A_\Theta-\lambda)^{-1}$, which yields
that there exist extended enumerations $(\alpha_k)$ and $(\beta_k)$ of the
isolated eigenvalues of $(\AN-\lambda)^{-1}$ and $(A_\Theta-\lambda)^{-1}$,
respectively, such that
\begin{equation}\label{estkato}
  \sum_{k=1}^\infty |\beta_k-\alpha_k|^p
  \le \bigl\|(A_\Theta-\lambda)^{-1}
  - (\AN-\lambda)^{-1}\bigr\|_{\frS_p(L^2(\Omega))}^p
\end{equation}
for $p>(n-1)/3$, $p\geq 1$;
by ``extended enumeration'' a sequence is meant that contains all isolated
eigenvalues of an operator according to their multiplicities plus endpoints
of the essential spectrum taken arbitrarily many times.
There exist indices $j_k$ such that
\[
  \frac{1}{\lambda_k-\lambda} = \beta_{j_k}.
\]
The corresponding values $\alpha_{j_k}$ can be written as
\[
  \alpha_{j_k} = \frac{1}{\mu_k-\lambda}
\]
where the $\mu_k$ are either isolated eigenvalues of $\AN$ or endpoints of the
essential
spectrum.  Now the estimate \eqref{estkato} implies that
\[
  \sum_{k=1}^\infty
\left|\frac{1}{\lambda_k-\lambda}-\frac{1}{\mu_k-\lambda}\right|^p < \infty.
\]
Since $\lambda_k\to\gamma$, we must have $\mu_k\to\gamma$.  Writing the
difference
of fractions as a single fraction and observing that the denominators converge
to
$\gamma-\lambda\ne0$, we can deduce the validity of \eqref{ellp}.
\end{proof}

If $n=2$ or $n=3$, then a trace formula is valid, which is stated in the next
corollary and follows directly from Corollary~\ref{co.trace_formula}.

\begin{corollary}\label{co.trace_ell}
Let the assumptions be
as in Theorem~\ref{thm1} and assume, in addition, that $n=2$ or $n=3$. Then the resolvent
difference in \eqref{resdiff1} is a trace
class operator and
\[
  \tr\bigl((A_\Theta-\lambda)^{-1}-(\AN-\lambda)^{-1}\bigr)
  = \tr\bigl(\,\ov{M'(\lambda)}\bigl(\Theta-\ov{M(\lambda)}\bigr)^{-1}\bigr)
\]
holds for all $\lambda\in\rho(A_\Theta)\cap\rho(\AN)$.
\end{corollary}

We note also that in the case $n=2$ or $n=3$ in the above corollary the wave operators of the pair $\{\AN, A_\Theta\}$ exist
(see, e.g.\ \cite[Theorem~X.4.12]{kato})
and that, in particular, the absolutely continuous parts of $\AN$ and $A_\Theta$ are unitarily equivalent
and the absolutely continuous spectra
of $\AN$ and $A_\Theta$ coincide.

The statement in the next theorem is a well known result from \cite{B62}.

\begin{theorem}\label{thm2}
Let $\AN$ and $\AD$ be the Neumann and Dirichlet operator associated with $\cL$.
Then for all $\lambda\in\rho(\AD)\cap\rho(\AN)$ the singular values $s_k$ of the
resolvent difference
\begin{equation}\label{resdiff2}
  (\AD-\lambda)^{-1} - (\AN-\lambda)^{-1}
\end{equation}
satisfy $s_k = O\bigl(k^{-\frac{2}{n-1}}\bigr)$, $k\to\infty$, and the
expression in \eqref{resdiff2}
is in $\sS_p(L^2(\Omega))$ for all $p>\frac{n-1}{2}$.
\end{theorem}

\begin{proof}
We apply Theorem~\ref{th.resolv_diff3} with
$\cG_0=\widetilde\cG_0=H^{1/2}(\partial\Omega)$.
If we set $\cK=H^\frac{1}{2}(\partial\Omega)$, $r_1=0$ and
$r_2=\tfrac{1}{2}$ in Lemma~\ref{le.s_emb}, then it follows that
the embedding operator from $H^{1/2}(\partial\Omega)$ into $L^2(\partial\Omega)$
belongs to
\begin{equation}\label{abcabc}
\sS_{\frac{1}{2(n-1)},\infty}\bigl(H^{1/2}(\partial\Omega),
L^2(\partial\Omega)\bigr).
\end{equation}
Hence Theorem~\ref{th.resolv_diff3} implies that \eqref{resdiff2} is in
\begin{equation*}
  \sS_{\frac{3}{2(n-1)},\infty}\cdot
  \sS_{\frac{1}{2(n-1)},\infty} =
  \sS_{\frac{2}{n-1},\infty},
\end{equation*}
that is, the singular values of \eqref{resdiff2} satisfy $s_k =
O(k^{-\frac{2}{n-1}})$. Lemma~\ref{splemma}(ii)
immediately gives the second statement.
\end{proof}

By taking differences of resolvent differences, the statements in the next
corollary follow directly from Theorems~\ref{thm1} and \ref{thm2}.

\begin{corollary}
\label{cor1}
Let $\Theta_1$ and $\Theta_2$ be self-adjoint relations in $L^2(\partial\Omega)$
such that $0\notin\sess(\Theta_i)$ and the realizations $A_{\Theta_i}$, $i=1,2$,
of $\cL$
are self-adjoint operators. Then
\[
  (A_{\Theta_1}-\lambda)^{-1} - (A_{\Theta_2}-\lambda)^{-1}
  \in \sS_{\frac{3}{n-1},\infty}(L^2(\Omega)),
\]
for all $\lambda\in\rho(A_{\Theta_1})\cap\rho(A_{\Theta_2})$ and
\[
  (A_{\Theta_1}-\lambda)^{-1} - (\AD-\lambda)^{-1}
  \in \sS_{\frac{2}{n-1},\infty}(L^2(\Omega)),
\]
for all $\lambda\in\rho(A_{\Theta_1})\cap\rho(\AD)$.
\end{corollary}

If the difference $\Theta_1-\Theta_2$ is itself in some ideal $\frS_{\infty,r}$,
we get an improvement of the first assertion in the previous corollary.

\begin{theorem}
\label{thm3}
Let $\Theta_1$ and $\Theta_2$ be bounded self-adjoint operators in
$L^2(\partial\Omega)$ such that
$0\notin\sess(\Theta_i)$ and the realizations $A_{\Theta_i}$, $i=1,2$, of $\cL$
from \eqref{cl}
are self-adjoint operators.
Moreover, assume that $s_k(\Theta_1-\Theta_2) = O(k^{-r})$, $k\to\infty$, for
some $r>0$.
Then for all $\lambda\in\rho(A_{\Theta_1})\cap\rho(A_{\Theta_2})$, the singular
values $s_k$ of the
resolvent difference
\begin{equation}\label{resdiff3}
  (A_{\Theta_1}-\lambda)^{-1}-(A_{\Theta_2}-\lambda)^{-1}
\end{equation}
satisfy $s_k = O\bigl(k^{-\frac{3}{n-1}-r}\bigr)$, $k\to\infty$, and the
expression in \eqref{resdiff3}
is in $\sS_p(L^2(\Omega))$ for all $p>(\frac{3}{n-1}+r)^{-1}$.
In particular, if $\Theta_1-\Theta_2\in\frS_q(L^2(\partial\Omega))$ for some
$q>0$, then
the expression in \eqref{resdiff3} is in $\frS_p(L^2(\Omega))$ for all
\[
  p > \frac{q(n-1)}{3q+n-1}\,.
\]
\end{theorem}

\begin{proof}
For $ \Theta_1-\Theta_2\in\sS_{r,\infty}$ we conclude from
Theorem~\ref{th.resolv_diff4} and Proposition~\ref{gammamh2}
that the difference of the resolvents in \eqref{resdiff3} is in
\begin{equation*}
  \sS_{\frac{3}{2(n-1)},\infty}\cdot\sS_{r,\infty}
  \cdot \sS_{\frac{3}{2(n-1)},\infty}=\sS_{\frac{3}{n-1}+r,\infty}.
\end{equation*}
The other assertions follow from Lemma~\ref{splemma}(ii) and (i) with
$r=\tfrac{1}{q}$\,.
\end{proof}

We leave it to the reader to formulate generalizations of Theorem~\ref{thm1} and Theorem~\ref{thm2} for
maximal dissipative and maximal accumulative realizations of $\cL$ by using the abstract results
in Section~\ref{3.4}.

\subsection{Elliptic operators with $\delta$ and $\delta'$-interactions on smooth hypersurfaces}\label{4.3}
In this section we investigate second order elliptic operators with
$\delta$ and $\delta^\prime$-interactions. Spectral problems for Schr\"{o}dinger operators with
$\delta$ and $\delta^\prime$-point interactions, as well as $\delta$-interactions on curves and
surfaces have been studied in, e.g.\ \cite{AGHH88,AGS87,BEKS94,E08,EF09,EI01,EK03,S88}.
In order to define self-adjoint elliptic operators in $L^2(\dR^n)$
with $\delta$ and $\delta^\prime$-interactions on a smooth compact hypersurface $\Sigma$ in $\dR^n$
we first construct suitable quasi boundary triples in Proposition~\ref{qbtthm2}.
We mention that in contrast to the approach via quadratic forms there appear no additional
technical difficulties when treating $\delta^\prime$-interactions; cf. \cite{E08}.
One of the main results in this section is Theorem~\ref{thm4}, where we obtain spectral estimates
for the resolvent differences of the operators with $\delta$ or $\delta^\prime$-interactions on the hypersurface
$\Sigma$ and the unperturbed self-adjoint realization in $L^2(\dR^n)$.

Let in the following $\Omega_{\rm i}\subset\dR^n$ be a bounded domain with
compact $C^\infty$-boundary and let $\Omega_{\rm e} \defeq \dR^n\backslash\overline{\Omega_{\rm i}}$,
so that $\partial\Omega_{\rm i} = \partial\Omega_{\rm e} \eqdef \Sigma$ and $\dR^n=\Omega_{\rm i}\,\dot\cup\,
\Sigma\,\dot\cup\,\Omega_{\rm e}$, and assume that both $\Omega_{\rm i}$ and $\Omega_{\rm e}$
are connected. In the following, $\Omega_{\rm i}$ is called interior
domain and $\Omega_{\rm e}$ exterior domain. A function $f$ defined on $\dR^n$ will often be
decomposed in the form $f_{\rm i}\oplus f_{\rm e}$, where $f_{\rm i}$ and $f_{\rm e}$ denote the
restrictions of $f$ to the interior and exterior domain, respectively.
Let $\cL$ be a formally symmetric,
uniformly elliptic differential expression as in~\eqref{cl} on the Euclidean space $\dR^n$.
The (usual) self-adjoint realization of $\cL$ in $L^2(\dR^n)$ is the operator $A_{\rm free}$ given by
\begin{equation}
\label{eq:Afree}
  A_{\rm free}f = \cL f,\quad \dom  A_{\rm free}  =  \bigl\{f \in H^1(\dR^n)\colon
\cL f\in L^2(\dR^n)\bigr\}.
\end{equation}
Observe that $A_{\rm free}$ is the unique self-adjoint operator associated with the quadratic form
corresponding to $\cL$ on $H^1(\dR^n)$; cf. \cite{kato} and \eqref{formA}.
The restrictions of $\cL$ to the interior domain $\Omega_{\rm i}$ and
exterior domain $\Omega_{\rm e}$ are denoted by
$\cL_{\rm i}$ and $\cL_{\rm e}$, respectively. Clearly, $\cL_{\rm i}$ and $\cL_{\rm e}$
are formally symmetric, uniformly elliptic differential expressions as
considered in Sections~\ref{4.1} and \ref{4.2}.
Like in Definition~\ref{domt} we introduce the operators $T_{\rm i}$ and $T_{\rm e}$ by
\begin{alignat*}{2}
  T_{\rm i}f_{\rm i} &= \cL_{\rm i}f_{\rm i}, &\quad
    \dom T_{\rm i} &= H^2(\Omega_{\rm i}),\\[1ex]
  T_{\rm e}f_{\rm e} &= \cL_{\rm e}f_{\rm e}, &\quad
    \dom T_{\rm e} &= \bigl\{f_{\rm e}\in H^1(\Omega_{\rm e})\colon
    \cL_{\rm e} f_{\rm e}\in L^2(\Omega_{\rm e}),\,
    f_{\rm e}|_{\Omega'}\in H^2(\Omega')\bigr\},
\end{alignat*}
where
$\Omega'\subset\Omega_{\rm e}$ is a bounded subdomain of
$\Omega_{\rm e}$
with smooth boundary such that $\Sigma=\partial\Omega_{\rm e} \subset \partial\Omega'$.
The Dirichlet and Neumann operators on the interior and exterior domain are defined as in Section~\ref{4.1}
by
\begin{alignat*}{2}
  A_{\rm D,\rm i} f_{\rm i}&=\cL_{\rm i}f_{\rm i},&\qquad
    \dom A_{\rm D,\rm i} &= \bigl\{f_{\rm i}\in\dom T_{\rm i}\colon f_{\rm i}\vert_\Sigma=0\bigr\},\\[1ex]
  A_{\rm D,\rm e} f_{\rm e}&=\cL_{\rm e}f_{\rm e},&\qquad
    \dom A_{\rm D,\rm e} &= \bigl\{f_{\rm e}\in\dom T_{\rm e}\colon f_{\rm e}\vert_\Sigma=0\bigr\},
\end{alignat*}
and
\begin{alignat*}{2}
  A_{\rm N,\rm i} f_{\rm i}&=\cL_{\rm i}f_{\rm i},&\qquad \dom A_{\rm N,\rm i}
  &= \left\{f_{\rm i}\in\dom T_{\rm i}\colon
    \frac{\partial f_{\rm i}}{\partial \nu_{\cL_{\rm i}}}\Bigl|_\Sigma=0\right\},\\[1ex]
  A_{\rm N,\rm e} f_{\rm e}&=\cL_{\rm e}f_{\rm e},&\qquad \dom A_{\rm N,\rm e}
  &= \left\{f_{\rm e}\in\dom T_{\rm e}\colon
    \frac{\partial f_{\rm e}}{\partial \nu_{\cL_{\rm e}}}\Bigl|_\Sigma=0\right\},
\end{alignat*}
respectively. Since $A_{\rm D,\rm i}$, $A_{\rm D,\rm e}$, $A_{\rm N,\rm i}$ and
$A_{\rm N,\rm e}$ are self-adjoint operators, it is clear that the orthogonal sums
\begin{equation}\label{eq:ADAN2}
  A_{\rm D,\rm i}\oplus A_{\rm D,\rm e}\qquad\text{and}\qquad  A_{\rm N,\rm i}\oplus A_{\rm N,\rm e}
\end{equation}
are self-adjoint operators in $L^2(\dR^n)=L^2(\Omega_{\rm i})\oplus L^2(\Omega_{\rm e})$,
and they both are restrictions of the operator $T_{\rm i}\oplus T_{\rm e}$. Note that the functions
in the domain of the operators in \eqref{eq:ADAN2} do not belong to $H^2$ in a neighbourhood of $\Sigma$
but only in one-sided neighbourhoods of $\Sigma$.
In order to treat $\delta$ and $\delta^\prime$-interactions with quasi boundary triple techniques,
we introduce the closed densely defined symmetric operators
\begin{equation}
\label{eq:wtAwhA}
  \wt A := A_{\rm free}\cap \bigl(A^{\rm i}_{\rm D} \oplus A^{\rm e}_{\rm D}\bigr)
  \quad\text{and}\quad
  \wh A := A_{\rm free}\cap \bigl(A^{\rm i}_{\rm N} \oplus A^{\rm e}_{\rm N}\bigr)
\end{equation}
in $L^2(\dR^n)$, as well as the restrictions
\begin{equation}
\label{eq:wtTwhT}
\begin{split}
  &\wt T f = \cL f, \quad \dom \wt T
  = \bigl\{f_{\rm i} \oplus f_{\rm e}\in \dom(T_{\rm i}\oplus T_{\rm e})\colon
  f_{\rm i}\vert_\Sigma = f_{\rm e}\vert_\Sigma  \bigr\},\\[1ex]
  &\wh T f = \cL f, \quad \dom \wh T
  = \left\{f_{\rm i} \oplus f_{\rm e}\in \dom (T_{\rm i}\oplus T_{\rm e})\colon
  \frac{\partial f_{\rm i}}{\partial\nu_{\cL_{\rm i}}}\Bigl|_{\Sigma}
  = -\frac{\partial f_{\rm e}}{\partial\nu_{\cL_{\rm e}}}\Bigl|_{\Sigma} \right\},
\end{split}
\end{equation}
of the operator $T_{\rm i}\oplus T_{\rm e}$ in $L^2(\dR^n)$.
In the next proposition it is shown how quasi boundary triples can be defined in this situation.

\begin{proposition}\label{qbtthm2}
Let $\wt A$  and $\wh A$ be the closed densely defined symmetric operators in
\eqref{eq:wtAwhA} and let $\wt T$ and $\wh T$  be as in \eqref{eq:wtTwhT}. Then the
following statements are true.
\begin{itemize}
\item[(i)] The triple $\{L^2(\Sigma), \wt \Gamma_0,\wt\Gamma_1\}$, where
\begin{displaymath}
  \wt \Gamma_0 \hat f = \frac{\partial f_{\rm i}}{\partial
  \nu_{\cL_{\rm i}}}\Bigl|_{\Sigma}+\frac{\partial f_{\rm e} }{\partial
  \nu_{\cL_{\rm e}}}\Bigl|_{\Sigma}\quad\text{and}\quad\wt\Gamma_1 \hat f
  = f\vert_\Sigma,\quad
  \hat f=\begin{pmatrix} f\\ \wt T f\end{pmatrix},\,\,\,f\in\dom\wt T,
\end{displaymath}
is a quasi boundary triple for $\wt A^*$ such that
\begin{equation*}
  \ker\wt\Gamma_0 = A_{\rm free}\quad\text{and}\quad \ker
  \wt \Gamma_1 = A_{\rm D, i} \oplus A_{\rm D,  e}.
\end{equation*}
\item[(ii)] The triple $\{L^2(\Sigma), \wh \Gamma_0,\wh\Gamma_1\}$, where
\begin{displaymath}
  \wh \Gamma_0 \hat f = \frac{\partial f_{\rm
  e}}{\partial \nu_{\cL_{\rm e}}}\Bigl|_{\Sigma},\quad \wh\Gamma_1 \hat f = f_{\rm e}\vert_\Sigma -
  f_{\rm i}\vert_\Sigma\quad\text{and}\quad \hat f=\begin{pmatrix} f\\ \wh T f\end{pmatrix},\,\,\,f\in\dom\wh T,
\end{displaymath}
is a quasi boundary triple for $\wh A^*$ such that
\begin{equation*}
  \ker\wh\Gamma_0 = A_{\rm N, i}\oplus
  A_{\rm N, e}\quad\text{and}\quad \ker \wh \Gamma_1 = A_{\rm free}.
\end{equation*}
\end{itemize}
\end{proposition}

\begin{proof}
We verify only assertion (ii). Item (i) can be shown in the same way and can alternatively be deduced
from \cite[Theorem 4.1, Lemma 4.2 and Proposition~4.2]{AB08}. In order to prove to (ii)
we make use of Theorem~\ref{suff_cond_qbt}. Note first that condition (a) in Theorem~\ref{suff_cond_qbt}
holds since $\ker\wh \Gamma_0 = A_{\rm N, i}\oplus A_{\rm N, e}$ is
self-adjoint; see also Theorem~\ref{qbtthm}.
It follows from \eqref{tracemap} that $\ran (\wh\Gamma_0, \wh \Gamma_1)^\top =  H^{1/2}(\Sigma)\times H^{3/2}(\Sigma)$, which is
dense in $L^2(\Sigma)\times L^2(\Sigma)$, and hence condition (b) in Theorem~\ref{suff_cond_qbt}
is also satisfied. In order to check condition (c), denote by
$(\cdot,\cdot)$, $(\cdot,\cdot)_{\rm i}$, $(\cdot,\cdot)_{\rm e}$ and $(\cdot,\cdot)_{\Sigma}$ the
inner products in $L^2(\dR^n)$, $L^2(\Omega_{\rm i})$, $L^2(\Omega_{\rm e})$ and
$L^2(\Sigma)$, respectively. For $f = f_{\rm i} \oplus
f_{\rm e}$ and $g = g_{\rm i} \oplus g_{\rm e}$ in $\dom \wh T$ we compute, with the help of
Green's identity,
\begin{equation}\label{greencompute}
\begin{split}
  (\wh T f, g) -  (f,\wh T g)
  &=(T_{\rm i}f_{\rm i}, g_{\rm i})_{\rm i} -
    (f_{\rm i},T_{\rm i} g_{\rm i})_{\rm i} + (T_{\rm e}f_{\rm e},
    g_{\rm e})_{\rm e} -  (f_{\rm e},T_{\rm e} g_{\rm e})_{\rm e} \\
  &=\left(f_{\rm i}\vert_\Sigma,\frac{\partial g_{\rm i}}{\partial\nu_{\cL_{\rm i}}}\Bigl|_\Sigma\right)_\Sigma
    - \left(\frac{\partial f_{\rm i}}{\partial \nu_{\cL_{\rm i}}}\Bigl|_\Sigma,
    g_{\rm i}\vert_\Sigma\right)_\Sigma \\
  &\qquad +\left(f_{\rm e}\vert_\Sigma,
    \frac{\partial g_{\rm e}}{\partial\nu_{\cL_{\rm e}}}\Bigl|_\Sigma\right)_\Sigma
    - \left(\frac{\partial f_{\rm e}}{\partial \nu_{\cL_{\rm e}}}\Bigl|_\Sigma,
    g_{\rm e}\vert_\Sigma\right)_{\Sigma}   \\
  &=\left(f_{\rm e}\vert_\Sigma - f_{\rm i}\vert_\Sigma, \frac{\partial g_{\rm e}}
    {\partial \nu_{\cL_{\rm e}}}\Bigl|_\Sigma\right)_{\Sigma}-\left(\frac{\partial f_{\rm e}}{\partial \nu_{\cL_{\rm e}}}\Bigl|_\Sigma,
    g_{\rm e}\vert_\Sigma - g_{\rm i}\vert_\Sigma\right)_\Sigma \\
  &=\bigl(\wh \Gamma_1 \hat f,\wh \Gamma_0 \hat g\bigr)_{\Sigma} - \bigl(\wh \Gamma_0 \hat f,\wh
    \Gamma_1 \hat g\bigr)_{\Sigma},
\end{split}
\end{equation}
where $\hat f=(f,\wh T f)^\top$ and $\hat g=(g,\wh T g)^\top$;
cf.\ the proof of Theorem~\ref{qbtthm}. Hence also condition (c) in Theorem~\ref{suff_cond_qbt}
holds. Therefore $\ker\wh\Gamma_0\cap\ker\wh\Gamma_1$ is a closed symmetric
operator in $L^2(\dR^n)$ and $\{L^2(\Sigma), \wh \Gamma_0,\wh\Gamma_1\}$ is a quasi boundary triple for
its adjoint. Since $\dom A_{\rm free}\subset H^2_{\rm loc}(\dR^n)$ and  $H^2_{\rm loc}(\dR^n)\subset\ker\wh\Gamma_1$
by the definition of the trace using approximations by $C^\infty$-functions it follows that
$A_{\rm free}\subset\ker\wh\Gamma_1$.
On the other hand, \eqref{greencompute} implies that
$\ker\wh\Gamma_1$ is a symmetric operator and therefore $A_{\rm free}=
\ker\wh\Gamma_1$. Together with $\ker\wh \Gamma_0 = A_{\rm N, i}\oplus A_{\rm N, e}$ this yields $\wh A=\ker\wh\Gamma_0\cap\ker\wh\Gamma_1$, and hence $\{L^2(\Sigma), \wh \Gamma_0,\wh\Gamma_1\}$ is a quasi boundary triple for $\wh A^*$.
\end{proof}

With the help of the quasi boundary triples from the previous proposition and the operators $\wt A$, $\wt T$,
$\wh A$ and $\wh T$, we define self-adjoint differential operators $A_{\delta,\alpha}$ and $A_{\delta^\prime,\beta}$
associated with $\cL$ and $\delta$ and $\delta^\prime$-interactions with strengths $\alpha$ and $\beta$ on $\Sigma$, respectively.
We remark that it is difficult
to treat $\delta^\prime$-interactions making use of quadratic forms, whereas the operator $A_{\delta,\alpha}$ with a $\delta$-interaction
could equivalently be defined with the help of the quadratic form; see, e.g.\ \cite{BEKS94} or \cite{E08},
where an additional minus sign appears in the boundary condition.
The statement in the next theorem is essentially a consequence of Theorem~\ref{th.A_Th_self_adj}.
We remark that in the quasi boundary triple framework also functions $\alpha,\beta$ with less
smoothness could be allowed.

\begin{theorem}\label{thm.sa2}
Let $\alpha,\beta\in C^1(\Sigma)$ be real-valued and assume that $\beta\ne0$ on $\Sigma$.
Then
\begin{displaymath}
  A_{\delta,\alpha} := \cL\upharpoonright \bigl\{\hat f\in \wt T\colon \alpha\wt\Gamma_1 \hat f = \wt\Gamma_0\hat f\bigr\}\,\,\,\,\text{and}\,\,\,\,
  A_{\delta^\prime,\beta} := \cL \upharpoonright
  \bigl\{\hat f\in\wh T\colon \wh\Gamma_1 \hat f = \beta\wh\Gamma_0\hat f \bigr\}
\end{displaymath}
are self-adjoint operators in $L^2(\dR^n)$.
\end{theorem}

Before proving the theorem we note that the interface condition $\alpha\wt\Gamma_1 \hat f = \wt\Gamma_0\hat f$,
$\hat f=\bigl(\begin{smallmatrix} f \\ \wt Tf\end{smallmatrix}\bigr)$,
has the explicit form
\begin{equation*}
  \alpha f\vert_\Sigma = \frac{\partial f_{\rm i}}{\partial\nu_{\cL_{\rm i}}}\Bigl|_\Sigma +
  \frac{\partial f_{\rm e}}{\partial\nu_{\cL_{\rm e}}}\Bigl|_\Sigma,\quad f_{\rm i}\vert_\Sigma=f_{\rm e}\vert_\Sigma,\quad f=f_{\rm i}\oplus f_{\rm e}\in\dom T_{\rm i}\oplus T_{\rm e},
\end{equation*}
and hence one can interpret the operator $A_{\delta,\alpha}$ as an elliptic operator
with $\delta$-interaction of strength~$\alpha$. The interface condition $\wh\Gamma_1 \hat f = \beta \wh\Gamma_0\hat f$, $\hat f=\bigl(\begin{smallmatrix} f \\ \wh Tf\end{smallmatrix}\bigr)$, has the explicit form
\begin{equation*}
  f_{\rm e}\vert_\Sigma-f_{\rm i}\vert_\Sigma=
  \beta \frac{\partial f_{\rm e}}{\partial\nu_{\cL_{\rm e}}}\Bigl|_\Sigma,\quad
  \frac{\partial f_{\rm i}}{\partial\nu_{\cL_{\rm i}}}\Bigl|_\Sigma
  =-\frac{\partial f_{\rm e}}{\partial\nu_{\cL_{\rm e}}}\Bigl|_\Sigma,\quad f_{\rm i}\oplus f_{\rm e}\in\dom T_{\rm i}\oplus T_{\rm e},
\end{equation*}
and therefore the operator $A_{\delta^\prime,\beta}$ can be interpreted as an elliptic operator
with $\delta^\prime$-interaction of strength~$\beta$.

\begin{proof}[of Theorem~\ref{thm.sa2}]
Only the self-adjointness of $A_{\rm \delta^\prime,\beta}$ will be shown. The
self-adjointness of $A_{\rm \delta,\alpha}$ can be checked analogously.
For the quasi boundary triple $\{L^2(\Sigma),\wh\Gamma_0,\wh\Gamma_1\}$ in Proposition~\ref{qbtthm2}(ii)
we have
\begin{equation*}
  \ran \wh \Gamma_0 = H^{1/2}(\Sigma)\quad\text{and}\quad\ran \wh \Gamma_1 = H^{3/2}(\Sigma),
\end{equation*}
so that for $\lambda\in\rho(A_{\rm N, i} \oplus A_{\rm N,e})\cap\rho(A_{\rm free})$ the
corresponding Weyl function $\wh M(\lambda)$ maps
$H^{1/2}(\Sigma)$ onto $H^{3/2}(\Sigma)$.
By the same argument as in Lemma~\ref{le.interpolation} the closure of $\wh M(\lambda)$ maps
$L^2(\Sigma)$ into $H^1(\Sigma)$, and it follows that this is a compact operator in $L^2(\Sigma)$.
In order to conclude from Theorem~\ref{th.A_Th_self_adj} with $\Theta=\beta$
that the operator $A_{\delta^\prime,\beta}$ is self-adjoint,
note that the assumptions $\beta\in C^1(\Sigma)$ and $\beta\ne0$ on $\Sigma$ imply
that the self-adjoint multiplication operator $\beta$ in $L^2(\Sigma)$ is boundedly invertible
and that $\beta^{-1}h\in H^{1/2}(\Sigma)$ for all $h\in H^1(\Sigma)$.
Hence $A_{\delta^\prime,\beta}$ is self-adjoint in  $L^2(\dR^n)$ by Theorem~\ref{th.A_Th_self_adj}.
\end{proof}

Let $\wt\gamma$ be the $\gamma$-field associated with the quasi boundary triple
$\{L^2(\Sigma),\wt\Gamma_0,\wt\Gamma_1\}$ in Proposition~\ref{qbtthm2}(i) and
let $\wh\gamma$ be the $\gamma$-field associated with the quasi boundary triple
$\{L^2(\Sigma),\wh\Gamma_0,\wh\Gamma_1\}$ in Proposition~\ref{qbtthm2}(ii).
The same reasoning as in the proof of Proposition~\ref{gammamh2}(i) yields
\begin{equation}
\label{eq:gammacoupling}
\begin{split}
  &\wt \gamma(\lambda)^*
  \in\sS_{\frac{3}{2(n-1)},\infty}\bigl(L^2(\dR^n),L^2(\Sigma)\bigr), \quad
  \lambda\in\rho(A_{\rm free}),\\
  &\wh \gamma(\lambda)^*
  \in\sS_{\frac{3}{2(n-1)},\infty}\bigl(L^2(\dR^n),L^2(\Sigma)\bigr), \quad
  \lambda\in\rho(A_{\rm N, i}\oplus A_{\rm N, e}).
\end{split}
\end{equation}

In the following preparatory lemma we show spectral estimates for the
resolvent differences of
$A_{\rm free}$, $A_{\rm D,i}\oplus A_{\rm D,e}$ and $A_{\rm N,i}\oplus A_{\rm N,e}$.

\begin{lemma}
\label{lem:FreeDirNeu}
Let $A_{\rm free}$, $A_{\rm D,i}\oplus A_{\rm D,e}$ and $A_{\rm N,i}\oplus A_{\rm N,e}$
be the self-adjoint operators associated with $\cL$ in $L^2(\dR^n)$ defined
in~\eqref{eq:Afree} and \eqref{eq:ADAN2}, respectively.
The singular values $s_k$ of  the resolvent differences
\begin{equation}\label{resdiff5}
\begin{split}
  &(A_{\rm free} - \lambda)^{-1} - (A_{\rm D, i} \oplus A_{\rm D, e}
  -\lambda)^{-1},\quad  \lambda \in \rho(A_{\rm free})\cap \rho(A_{\rm D, i}
  \oplus A_{\rm D, e}), \\
  &(A_{\rm free} - \lambda)^{-1} - (A_{\rm N, i} \oplus A_{\rm N, e}
  -\lambda)^{-1},\quad  \lambda\in \rho(A_{\rm free})\cap\rho(A_{\rm N, i}
  \oplus A_{\rm N, e}),
\end{split}
\end{equation}
satisfy $s_k = O\bigl(k^{-\frac{2}{n-1}}\bigr), ~k\rightarrow\infty$,
and the expressions in \eqref{resdiff5} are in $\sS_p(L^2(\dR^n))$ for
all $p > \frac{n-1}{2}$.
\end{lemma}

\begin{proof}
In order to show the statement for the first resolvent difference in \eqref{resdiff5}
we apply Theorem~\ref{th.resolv_diff3} with $\cG_0=\widetilde\cG_0=H^{1/2}(\Sigma)$ and the quasi
boundary triple $\{L^2(\Sigma),\wt\Gamma_0,\wt\Gamma_1\}$
in the same form as in the proof of Theorem~\ref{thm2}. Lemma~\ref{le.s_emb} implies that
the embedding operator from $H^{1/2}(\Sigma)$ into $L^2(\Sigma)$
belongs to
\begin{equation*}
  \sS_{\frac{1}{2(n-1)},\infty}\bigl(H^{1/2}(\Sigma),L^2(\Sigma)\bigr);
\end{equation*}
cf.\ \eqref{abcabc}. According to Proposition~\ref{qbtthm2}(i) we have
$A_{\rm free}=\ker\wt\Gamma_0$ and $A_{\rm D, i}\oplus A_{\rm D, e}=\ker\wt\Gamma_1$.
Moreover, for $\lambda\in\rho(A_{\rm free})\cap \rho(A_{\rm D, i}
\oplus A_{\rm D, e})$ the corresponding Weyl function $\widetilde M(\lambda)$ maps
$H^{1/2}(\Sigma)$ onto $H^{3/2}(\Sigma)$, and is bounded when regarded as an operator in $L^2(\Sigma)$;
cf.\ Lemma~\ref{le.interpolation}.
Then Theorem~\ref{th.resolv_diff3} implies that the resolvent difference
\begin{equation}\label{resdiff9}
  (A_{\rm free} - \lambda)^{-1} - (A_{\rm D, i} \oplus A_{\rm D, e}-\lambda)^{-1}
\end{equation}
belongs to
\begin{displaymath}
  \sS_{\frac{3}{2(n-1)},\infty}\cdot
  \sS_{\frac{1}{2(n-1)},\infty}
  = \sS_{\frac{2}{n-1},\infty};
\end{displaymath}
cf.\ the proof of Theorem~\ref{thm2}. Hence the singular values $s_k$
of \eqref{resdiff9} satisfy $s_k=O(k^{-\frac{2}{n-1}})$
and by Lemma~\ref{splemma}(ii) the difference \eqref{resdiff9}
is in $\sS_p(L^2(\dR^n))$ for all $p > \frac{n-1}{2}$.

It remains to show the statements for the second resolvent difference in~\eqref{resdiff5}.
For this we note that by Theorem~\ref{thm2} the singular values $s_k$ of the resolvent
differences
\begin{alignat*}{2}
  & (A_{\rm D,i}-\lambda)^{-1} - (A_{\rm N,i}-\lambda)^{-1}, \qquad
  && \lambda\in\rho(A_{\rm D,i})\cap\rho(A_{\rm N,i}), \\[1ex]
  &(A_{\rm D,e}-\lambda)^{-1} - (A_{\rm N,e}-\lambda)^{-1},  \qquad
  && \lambda\in\rho(A_{\rm D,e})\cap\rho(A_{\rm N,e}),
\end{alignat*}
satisfy $s_k = O(k^{-\frac{2}{n-1}})$, $k\rightarrow\infty$. This implies
that the singular values $s_k$ of the orthogonal sum
\begin{equation*}
  (A_{\rm D,i}\oplus A_{\rm D,e}-\lambda)^{-1} - (A_{\rm N,i}\oplus A_{\rm N,e}-\lambda)^{-1},
\end{equation*}
also satisfy $s_k = O(k^{-\frac{2}{n-1}})$, $k\rightarrow\infty$,
for $\lambda\in\rho(A_{\rm D,i}\oplus A_{\rm D,e})\cap\rho(A_{\rm N,i}\oplus A_{\rm N,e})$.
Together with the properties of \eqref{resdiff9} we conclude that the singular values $s_k$ of the
second resolvent difference in~\eqref{resdiff5}
satisfy $s_k = O(k^{-\frac{2}{n-1}})$, $k\rightarrow\infty$, and
the statement on the Schatten--von Neumann class follows again from Lemma~\ref{splemma}(ii).
\end{proof}

The next theorem is the main result in this subsection. We compare the resolvent of the
unperturbed operator $A_{\rm free}$ with
the resolvents of the self-adjoint operators $A_{\delta,\alpha}$ and $A_{\delta^\prime,\beta}$
modelling $\delta$ and $\delta^\prime$-interactions on $\Sigma$.

\begin{theorem}\label{thm4}
Let $\alpha,\beta\in C^1(\Sigma)$ be real-valued and assume that $\beta\ne0$ on $\Sigma$.
Further, let $A_{\rm free}$ be the self-adjoint elliptic operator
associated with $\cL$ in~\eqref{eq:Afree} and let $A_{\delta,\alpha}$
and $A_{\delta^\prime,\beta}$ be the self-adjoint operators from Theorem~\ref{thm.sa2}.
Then the following statements are true.
\begin{itemize}

\item[(i)] For all $\lambda\in \rho(A_{\rm \delta,\alpha})\cap  \rho(A_{\rm free})$
the singular values $s_k$ of the resolvent difference
\begin{equation}
\label{resdiff6}
  (A_{\delta,\alpha} - \lambda)^{-1} - (A_{\rm free} -\lambda)^{-1}
\end{equation}
satisfy $s_k = O(k^{-\frac{3}{n-1}})$, $k\rightarrow\infty$, and the
expression in~\eqref{resdiff6} is in $\sS_p(L^2(\dR^n))$ for all
$p > \frac{n-1}{3}$.

\item[(ii)] For all $\lambda\in \rho(A_{\delta^\prime,\beta})\cap  \rho(A_{\rm free})$
the singular values $s_k$ of the resolvent difference
\begin{equation}
\label{resdiff7}
  (A_{\delta^\prime,\beta} - \lambda)^{-1} - (A_{\rm free} -\lambda)^{-1}
\end{equation}
satisfy $s_k = O(k^{-\frac{2}{n-1}})$, $k\rightarrow\infty$, and the
expression in~\eqref{resdiff7} is in $\sS_p(L^2(\dR^n))$ for all
$p > \frac{n-1}{2}$.
\end{itemize}
\end{theorem}

\begin{proof}
(i) It follows from Theorem~\ref{thm.sa2} that the self-adjoint operator $A_{\delta,\alpha}$
corresponds to the self-adjoint linear relation
\[
  \wt\Theta=\left\{ \binom{\alpha h}{h} \colon h\in L^2(\Sigma)\right\}
\]
via the quasi boundary triple $\{L^2(\Sigma),\wt\Gamma_0,\wt\Gamma_1\}$, i.e.\
\[
  A_{\delta,\alpha}=\left\{\hat f\in\wt T\colon
  \begin{pmatrix}\wt\Gamma_0\hat f\\ \wt\Gamma_1\hat f\end{pmatrix}\in\wt\Theta\right\}.
\]
In order to apply Theorem~\ref{th.resolv_diff1}, we note that the closures of the values of the
Weyl function $\widetilde M(\lambda)$, $\lambda\in\rho(A_{\rm free})$, associated
with $\{L^2(\Sigma),\wt\Gamma_0,\wt\Gamma_1\}$ are compact operators in $L^2(\Sigma)$;
cf.\ Lemma~\ref{le.interpolation}.
Since $\alpha$ is assumed to be in $C^1(\Sigma)$, it follows that $\wt\Theta^{-1}$ is an everywhere defined
bounded operator in $L^2(\Sigma)$; in particular, $0\notin\sess(\wt\Theta)$. Therefore we can apply
Theorem~\ref{th.resolv_diff1}. Together with \eqref{eq:gammacoupling} we conclude that the
resolvent difference in \eqref{resdiff6} belongs to
\[
  \sS_{\frac{3}{2(n-1)},\infty}\cdot
  \sS_{\frac{3}{2(n-1)},\infty}
  = \sS_{\frac{3}{n-1},\infty}.
\]
This shows the statement on the singular values. By Lemma~\ref{splemma}(ii)
the resolvent difference \eqref{resdiff6} belongs to the classes  $\sS_p(L^2(\dR^n))$,
$p > \frac{n-1}{3}$.

(ii) This statement is an immediate consequence of Lemma~\ref{lem:FreeDirNeu} and
Theorem~\ref{thm5} below,  which is of independent interest.
\end{proof}

The following theorem tells us that $A_{\delta',\beta}$ is close to the direct sum of
the Neumann operators in the sense of spectral estimates for the resolvent differences.

\begin{theorem}
\label{thm5}
Let $\beta\in C^1(\Sigma)$ be real-valued and assume that $\beta\ne0$ on $\Sigma$.
Further, let $A_{\rm free}$ and $A_{\delta^\prime,\beta}$ be as above and let
$A_{\rm N, i}\oplus A_{\rm N, e}$ be the orthogonal sum of the Neumann operators
on the interior and exterior domain from \eqref{eq:ADAN2}.
Then for all $\lambda\in \rho(A_{\delta^\prime,\beta})\cap \rho(A_{\rm
N, i} \oplus A_{\rm N, e}) $ the singular values $s_k$ of the
resolvent difference
\begin{equation}
\label{resdiff8}
  (A_{\delta^\prime,\beta} - \lambda)^{-1} - (A_{\rm N, i} \oplus
  A_{\rm N, e} -\lambda)^{-1}
\end{equation}
satisfy $s_k = O(k^{-\frac{3}{n-1}})$, $k\rightarrow\infty$, and the
expression in \eqref{resdiff8} is in $\sS_p(L^2(\dR^n))$ for all
$p > \frac{n-1}{3}$.
\end{theorem}

\begin{proof}
According to Theorem~\ref{thm.sa2} the self-adjoint operator
$A_{\delta^\prime,\beta}$ is given by $\ker(\wh\Gamma_1-\wh\Theta\wh\Gamma_0)$,
where $\wh\Theta=\beta$ is the multiplication operator by $\beta$ in $L^2(\Sigma)$.
The assumptions $\beta\in C^1(\Sigma)$ and $\beta\ne0$ on $\Sigma$ imply
that $0\notin\sess(\wh\Theta)$. Note also that the closures of the values
$\widehat M(\lambda)$, $\lambda\in\rho(A_{\rm N,i}\oplus A_{\rm N,e})$, associated with
$\{L^2(\Sigma),\wh\Gamma_0,\wh\Gamma_1\}$ are compact in $L^2(\Sigma)$;
cf.\ Lemma~\ref{le.interpolation}.
Thus we can apply Theorem~\ref{th.resolv_diff1}, and as in the proof of Theorem~\ref{thm4}(i)
we obtain the statement.
\end{proof}

\begin{remark}
Let $T_1$ and $T_2$ be self-adjoint operators in a Hilbert space $\cH$.
We write
\begin{equation*}
  T_1 \stackrel{\gamma}{\text{\textendash\textendash\textendash\textendash}} T_2
\end{equation*}
if the difference of the resolvents of $T_1$ and $T_2$ belongs to $\sS_{\gamma,\infty}(\cH)$.
With this notation, Lemma~\ref{lem:FreeDirNeu},
Theorem~\ref{thm4} and Theorem~\ref{thm5} can be illustrated as follows:
\begin{equation*}
  A_{\rm N,i}\oplus A_{\rm N,e}
  \stackrel{\frac{3}{n-1}}{\text{\textendash\textendash\textendash\textendash\textendash}}
  A_{\rm\delta^{\prime},\beta}
  \stackrel{\frac{2}{n-1}}{\text{\textendash\textendash\textendash\textendash\textendash}}
  A_{\rm free}
  \stackrel{\frac{3}{n-1}}{\text{\textendash\textendash\textendash\textendash\textendash}}
  A_{\rm\delta,\alpha}
  \stackrel{\frac{2}{n-1}}{\text{\textendash\textendash\textendash\textendash\textendash}}
  A_{\rm D,i}\oplus A_{\rm D,e}
\end{equation*}
\end{remark}

\subsection*{Acknowledgements}
M.~Langer was supported by the Engineering and Physical Sciences Research
Council (EPSRC) of the UK, grant EP/E037844/1.
V.~Lotoreichik was supported by the personal grant 2.1/30-04/035 of
the government of St.~Petersburg and the Leonard Euler programme of DAAD, grant
50077360.


\end{document}